%% file: main.tex
\documentclass[journal,twoside,web]{ieeecolor}
\usepackage{generic}
\usepackage{cite}
\usepackage{textcomp}

\usepackage{array}

\usepackage{cite}
\usepackage{amsmath,amsthm,amssymb,amsfonts}

\usepackage{xpatch}
\makeatletter
\xpatchcmd{\@thm}{\thm@headpunct{.}}{\thm@headpunct{}}{}{}
\makeatother

\usepackage{bm}
\usepackage[T1]{fontenc}

\usepackage{algorithm,algorithmic}
\usepackage{graphicx}
\usepackage{textcomp}
\usepackage{xcolor}

\usepackage[short]{optidef}

\usepackage{caption}

\usepackage{aliascnt}

\usepackage{booktabs}   
\usepackage{tabularx}   
\newcolumntype{Y}{>{\centering\arraybackslash}X}                 
\newcolumntype{L}{>{\centering\arraybackslash}p{0.045\textwidth}}
\newcolumntype{W}{>{\centering\arraybackslash}p{0.053\textwidth}}
\newcolumntype{S}{>{\centering\arraybackslash}p{0.048\textwidth}}
\renewcommand{\tabcolsep}{3pt}                                    

\newtheorem{theorem}{Theorem}
\newtheorem{assumption}{Assumption}

\newtheorem{lemma}{Lemma}

\newtheorem{remark}{Remark}
\newtheorem{example}{Example}

\DeclareMathOperator{\col}{col}

\DeclareMathOperator{\diag}{diag}

\input{styledsymbols} 

\usepackage{fancyhdr}
\allowdisplaybreaks 

\usepackage{colortbl}
\definecolor{LightGray}{gray}{0.9}
\definecolor{LightCyan}{rgb}{0.4863    0.6157    0.5922}
\definecolor{LightGreen}{rgb}{0.8824    0.8000    0.6941}
\definecolor{LightOrange}{rgb}{0.8314    0.7294    0.6784}

\newcommand{\lowtildezeta}{%
  \mathord{\ooalign{%
    $\zeta$\cr
    \hidewidth\raisebox{1.65ex}{\scalebox{0.65}{$\scriptscriptstyle\bm{\sim}$}}\hidewidth\cr
  }}%
}

\usepackage{multirow}
\usepackage{array}
\newcolumntype{M}[1]{>{\centering\arraybackslash}m{#1}}
\newcolumntype{N}{@{}m{0pt}@{}}

\usepackage{fancyhdr}
\usepackage[T1]{fontenc}     
\usepackage[scaled]{helvet}  

\makeatletter
\def\@myheaderfont{\scriptsize\textsf}
\def\ps@myheadings{%
  \def\@oddhead{%
    \vbox{%
      \hbox to \textwidth{\@myheaderfont \rightmark \hfil \@myheaderfont \thepage}%
      \vskip0.4ex
      {\color{subsectioncolor}\hrule height 1pt}%
    }%
  }%
  \def\@evenhead{%
    \vbox{%
      \hbox to \textwidth{\@myheaderfont \thepage \hfil \@myheaderfont \leftmark}%
      \vskip0.4ex
      {\color{subsectioncolor}\hrule height 1pt}%
    }%
  }%
  \def\@oddfoot{}\def\@evenfoot{}%
}
\let\ps@headings\ps@myheadings
\let\ps@IEEEtitlepagestyle\ps@myheadings
\makeatother

\usepackage[hyphens]{url}

\usepackage[
    colorlinks=false,                          
    pdfborder={0 0 1},                        
    linkbordercolor={1 0.3 0.3},               
    citebordercolor={0 1 0},                   
    urlbordercolor={0 0 1},                    
    breaklinks=true                            
]{hyperref}

\hypersetup{pdflinkmargin=0.5pt}

\makeatletter
\ifpdftex
\newcommand*{\UniformReferenceStartLink}[2]{%
  \pdfstartlink height 0.69em depth 0.23em attr{#1}goto name{#2}\relax
}
\xpatchcmd{\find@pdflink}
  {\Hy@StartlinkName}
  {\ifstrequal{#1}{link}{\UniformReferenceStartLink}{\Hy@StartlinkName}}
  {}{\PackageError{reference-links}{Could not set uniform link dimensions}{}}
\fi
\ifxetex
\newcommand*{\UniformReferenceLinkStrut}{%
  \begingroup
    \setbox\z@=\hbox{%
      \rlap{\raise 0.69em\hbox{\XeTeXLink@space}}%
      \rlap{\lower 0.23em\hbox{\XeTeXLink@space}}%
    }%
    \ht\z@=0pt \dp\z@=0pt \wd\z@=0pt
    \box\z@
  \endgroup
}
\xapptocmd{\hyper@linkstart}{%
  \ifstrequal{#1}{link}{\UniformReferenceLinkStrut}{}%
}
  {}{\PackageError{reference-links}{Could not set uniform XeLaTeX link dimensions}{}}
\fi
\makeatother

\begin{document}
\title{Heterogeneous Distributed \\ Zeroth-Order  Nonconvex Optimization \\ with Communication Compression}
\author{Haonan Wang, Xinlei Yi, \IEEEmembership{Member, IEEE}, Yiguang Hong, \IEEEmembership{Fellow, IEEE}, 
and Minghui Liwang, \IEEEmembership{Senior Member, IEEE}
\thanks{This work was supported in part by the National Natural Science Foundation of China under Grants 62503365, 62573319, 62271424, and 62088101.}
\thanks{H. Wang, X. Yi, Y. Hong, and M. Liwang are with the Department of Control Science and Engineering,
 College of Electronics and Information Engineering, Tongji University,
 Shanghai 201804, China.
 X. Yi, Y. Hong, and M. Liwang are also with the State Key Laboratory of Autonomous Intelligent Unmanned Systems,
  and Frontiers Science Center for Intelligent Autonomous Systems, Ministry of Education, 
  and the Shanghai Institute of Intelligent Science and Technology, Tongji University, Shanghai 200092, China
 (e-mail: hnwang@tongji.edu.cn; xinleiyi@tongji.edu.cn; yghong@iss.ac.cn; minghuiliwang@tongji.edu.cn).}
}

\maketitle

\begin{abstract}
Distributed zeroth-order optimization is increasingly applied in heterogeneous scenarios
 where agents possess distinct data distributions and objectives.
This heterogeneity poses fundamental challenges for convergence analysis,
 as existing convergence analyses rely on relatively strong assumptions to ensure theoretical guarantees.
  Specifically, at least one of the following three assumptions is usually required:
(i) data homogeneity across agents,
(ii) $\mathcal{O}(\mathit{pn})$ function evaluations per iteration with $\mathit{p}$ denoting the dimension and $\mathit{n}$ the number of agents, or
(iii) the Polyak--\L{}ojasiewicz (P--L) or strong convexity condition with a known corresponding constant.
To overcome these limitations, we propose a \underline{He}terogeneous \underline{D}istributed \underline{Z}eroth-\underline{O}rder \underline{C}ompressed (HEDZOC) algorithm, 
which is based on a two-point zeroth-order gradient estimator and a general class of compressors.
 Without assuming data homogeneity, we develop the analysis covering  three settings: general nonconvex functions, functions satisfying the P--L condition without knowing the P--L constant, and those with a known constant.
To the best of our knowledge, the proposed HEDZOC algorithm is the first distributed zeroth-order method that establishes convergence without relying on 
 the above three assumptions.
Moreover, it achieves linear speedup convergence rate, which is comparable to state-of-the-art results attainable under data homogeneity and exact communication assumptions.
 Finally, experiments on heterogeneous adversarial example generation 
 validate the theoretical results.

\end{abstract}


\section{Introduction} \label{sec:introduction}


Distributed optimization has recently gained considerable attention for its broad applications in areas such as large-scale machine learning and networked systems \cite{Tsianos_Consensus_2012,Nedic_DistributedControl_2018,Yang_Survey_2019}.
Typically, consider $n$ agents, each equipped with a private local cost function 
$f_i:\mathbb{R}^p\to\mathbb{R}$ (not necessarily convex).
 They collaborate to 
 achieve consensus and optimal model parameter by 
 solving the following optimization problem:
\begin{align}\label{zerosg:eqn:xopt}				
 \min_{x\in \mathbb{R}^p} f(x) = \frac{1}{n}\sum_{i=1}^nf_i(x) = \frac{1}{n}\sum_{i=1}^n\mathbb{E}_{\xi_i}[F_i(x,\xi_i)],
\end{align}
where $x\in \mathbb{R}^p$ is the model parameter, 
$\xi_i\in\Xi_i$ is a local data sample drawn from the local distribution $\mathcal{D}_i$, 
and $F_i(x,\xi_i): \mathbb{R}^{p}\times \Xi_i \mapsto \mathbb{R}$ is a stochastic realization of the local cost function $f_i$.
To solve problem \eqref{zerosg:eqn:xopt}, various algorithms have been proposed, 
from convex methods \cite{Nedic_Distributed_2009,Boyd_ADMM_2011,Shi_EXTRA_2015} and later extending to nonconvex formulations \cite{Di_NEXT_2016,Tatarenko_Nonconvex_2017,Yi_Linear_2021}.
More recently, the stochastic nonconvex setting \cite{xin2020variance,Yi_PrimalDual_2022} has been increasingly studied, particularly due to its relevance to neural network training.

However, most existing methods are first-order (gradient-based).
In many practical scenarios one can only query finite number of function valuations  (samplings), such as optimization with 
black-box models \cite{Zhang_Revisiting_2024} and bandit feedback \cite{liu2010distributed}.
These issues naturally call for zeroth-order (gradient-free) optimization methods, typically classified by how their gradient estimators are constructed.
Many existing distributed zeroth-order algorithms construct gradient estimators using two-point differences in both convex \cite{yuan_randomized_2015} and nonconvex settings \cite{Tang_Zeroth_2020,Yi_Zerothorder_2022}. 
Such a  scheme offers a favorable balance between the number of samplings and the convergence rate, and for example linear speedup convergence is achieved in \cite{Yi_Zerothorder_2022}.
In contrast, some distributed zeroth-order methods rely on $2p$-point sampling \cite{wang_distributed_2019,Tang_Zeroth_2020,Yi_Linear_2021}.  
While their convergence rates are comparable to those of first-order methods, they incur a heavy sampling burden when the dimension $p$ and the number of agents $n$ are large.
In the centralized setting, there also exist zeroth-order methods based on one-point sampling \cite{flaxman2004online}, 
as well as zeroth-order Hessian-based algorithms \cite{conn_global_2009}.

Communication is another key aspect of distributed optimization and often becomes a practical bottleneck. 
This has motivated extensive studies on communication-efficient designs, such as compression \cite{Rabbat_Quantized_2005}, event-triggered communication \cite{Lemmon_Event_2010}, and asynchronous schemes \cite{nedic_asynchronous_2010}.
Among these directions, compression is particularly appealing due to its simplicity and ease of implementation for reducing communication overhead.
In distributed convex optimization, a variety of compressors have been applied, including quantizers \cite{yi2014quantized}, 
unbiased compressors \cite{alistarh2017qsgd}, and contractive compressors \cite{koloskova2019decentralized}, 
and these have been further extended to nonconvex settings \cite{basu2020qsparse}. 
Another line of research aims to unify these compressors. 
\cite{liao2022compressed} introduced a general class of compressors with bounded relative compression error that encompasses both unbiased and contractive compressors. 
\cite{Yi_CommunicationCompression_2023} further analyzed compressor classes with global and local absolute-error bounds.
Additionally, several studies focus on accelerating convergence rates of compressed algorithms.
Notably, some of compressed methods attain convergence rates comparable to uncompressed first-order \cite{Yi_CommunicationCompression_2023}, stochastic gradient descent \cite{basu2020qsparse}, and zeroth-order algorithms \cite{Wang_CZSD_2025}.

In the context of distributed optimization, data heterogeneity can significantly affect algorithmic convergence, arising in areas such as meta-learning \cite{Hospedales_Meta_2021} and large language model adaptation \cite{zhao_survey_2025}. 
This challenge was first recognized in federated learning \cite{li_federated_2020_mag},
where multiple local steps and system heterogeneity such as stragglers can further exacerbate its impact.
To address this issue, many algorithms have been proposed in both federated 
\cite{li_federated_2020,wang_objective_2020} and distributed settings \cite{lian2017can}. 
They typically impose assumptions that limit heterogeneity, which we term homogeneity assumptions. 
Notably, when each agent performs only one local update and all agents participate in every round, 
the algorithms reduce to standard distributed 
 first-order methods,
for which homogeneity assumptions are unnecessary \cite{Yi_PrimalDual_2022}.
In contrast, distributed zeroth-order methods typically need data homogeneity, 
without which zeroth-order gradient estimation errors can be further amplified, leading to an inherently non-vanishing estimator variance \cite{Mu_VRZOD_2026}.
Indeed, distributed zeroth-order algorithms based on two-point sampling in both federated~\cite{Ling_Federated_2024}
 and distributed settings~\cite{yuan_randomized_2015,Tang_Zeroth_2020,Yi_Zerothorder_2022}, consistently require data homogeneity to obtain convergence guarantees.
Alternately, methods relying on $2p$-point sampling per iteration~\cite{Yi_Linear_2021,Mu_VRZOD_2026,Wang_Orthogonal_2024}
construct gradient estimators that approximate the first-order gradient with arbitrarily small error.
This allows data homogeneity to be removed, and the convergence analysis closely follows standard first-order arguments.
However, their $\mathcal{O}(pn)$ sampling burden per iteration grows rapidly with the problem dimension and network size, limiting their scalability in practice. 
To the best of our knowledge,~\cite{Yi_Zerothorder_2022} is the only work that 
provided the
convergence without assuming data homogeneity or $\mathcal{O}(pn)$ function evaluations, but it requires a different condition that the global cost function satisfies the Polyak--\L{}ojasiewicz (P--L) condition with a known corresponding constant.
As a result, convergence guarantees for general convex and nonconvex objectives are still lacking.

As discussed above, existing distributed zeroth-order methods are mostly developed for the homogeneous setting and
establish convergence only under at least one of the following three relatively strong assumptions:  
(i) data homogeneity, which is often unrealistic in practical heterogeneous settings;  
(ii) $\mathcal{O}(pn)$ function evaluations per iteration, which leads to substantial sampling burden; or  
(iii) the P--L condition with a known constant, which applies only to a specific class of nonconvex functions and is also unrealistic since the constant is typically unknown.
Moreover, communication overhead remains a key bottleneck in distributed systems.

\subsection{Main Contributions}

This paper aims to relax the three relatively strong assumptions for distributed zeroth-order optimization
and reduce the communication burden through communication compression. 
Specifically, the main contributions of this paper are summarized as follows.

(i)
We study the distributed zeroth-order optimization based on two-point function evaluations in a heterogeneous setting with general nonconvex (and convex) objectives, where agents have distinct data distributions and cost functions.
To further reduce the communication burden via communication compression,
we propose a \underline{He}terogeneous \underline{D}istributed \underline{Z}eroth-\underline{O}rder \underline{C}ompressed (HEDZOC) algorithm.
In contrast to existing distributed zeroth-order methods that rely on at least one of the three relatively strong assumptions, i.e., data homogeneity, $\mathcal{O}(pn)$
 function evaluations per iteration, and the P--L condition, HEDZOC achieves first provable convergence without any of them.



(ii)
 Under heterogeneous data distributions and communication compression,
  we first establish linear speedup convergence rate $\mathcal{O}(\sqrt{p}/(\sqrt{nT}))$ in the general nonconvex setting,     
matching
 the best theoretical guarantees achieved under data homogeneity and exact communication assumptions  \cite{Yi_Zerothorder_2022,Wang_CZSD_2025},
 where $T$ denotes the total number of iterations.
To the best of our knowledge, for general distributed two-point zeroth-order optimization under data heterogeneity,
 this is the first result proving (linear speedup) convergence.
Moreover, under the P--L condition, the proposed algorithm attains a rate of $\mathcal{O}(p/(nT^{\theta}))$ for any $\theta\in(0.5,1)$ without prior knowledge of the P--L constant, 
and further improves to $\mathcal{O}(p/(nT))$ when the P--L constant is known, again exhibiting linear speedup.



(iii)
Existing relatively strong assumptions in the literature are primarily imposed to bound the error of zeroth-order gradient estimators, but they fail in general heterogeneous distributed two-point zeroth-order optimization,
making it crucial to re-establish such bounds.
To this end, we scale the variance of the two-point zeroth-order gradient estimator to the optimality gap and treat it as a perturbation term.
Within a Lyapunov framework, we design a special stepsize to regulate this perturbation and use mathematical induction to keep the optimality gap bounded,  thereby yielding a desired bound of the zeroth-order gradient estimation error.
The above analytical methodology itself represents a substantive innovation.

\subsection{Organization and Notations}
The rest of this paper is organized as follows.
Section~\ref{Section:Existing} provides a detailed analysis of existing relatively strong assumptions and the challenges in removing them.
Then  Section~\ref{Section:Preliminaries} introduces the problem formulation in the heterogeneous setting.
Section~\ref{section:Algorithm} presents the proposed HEDZOC algorithm, Section~\ref{Section:Proof} develops the preliminary convergence analysis, and Section~\ref{section:main results} states the main results.
Section~\ref{Section:Simulations} gives numerical experiments to validate the theoretical findings.
Finally,  Section~\ref{Section:Conclusion} concludes the paper.

 {\bf Notations:} 
Let $[n]=\{1,\dots,n\}$ for any positive integer $n$.
Denote the Euclidean norm by $\|\cdot\|$.
$\mathbb{B}^p$ and $\mathbb{S}^p$ denote the unit ball and unit sphere in $\mathbb{R}^p$, respectively. $\mathrm{Unif}(\cdot)$ represents the uniform distribution.
For any $x\in\mathbb{R}^p$, $[x]_l$ denotes its $l$-th entry, and $\col(x_1,\ldots,x_k)$ stacks vectors $x_i\in\mathbb{R}^{p_i}$.
For a differentiable function $f$, $\nabla f$ denotes its gradient.
${\bf 1}_n$ and ${\bf I}_n$ stand for the $n$-dimensional all-one vector and the $n\times n$ identity matrix, respectively.
$\rho(\cdot)$ stands for the spectral radius of a matrix, and $\rho_2(\cdot)$ denotes its minimum positive eigenvalue. 
$\diag(t_1,\ldots,t_n)$ denotes a diagonal matrix with diagonal entries $t_i$. $A\otimes B$ denotes the Kronecker product.
For a positive semidefinite matrix $A$, define $\|x\|_A=\sqrt{\langle x,Ax\rangle}$.
The subscript $i$ indexes the agent and $k$ indexes the iteration. 
Specifically, $x_{i,k}\in\mathbb{R}^p$ denotes agent~$i$'s local estimate of the solution to the global problem \eqref{zerosg:eqn:xopt} at the $k$-th iteration. $v_{i,k}\in\mathbb{R}^p$ is the dual variable.
We also denote  
$\bsx_k=\col(x_{1,k}, \dots,x_{n,k})$,	$\bsv_k=\col(v_{1,k},\dots,v_{n,k})$, 
 $\bar{x}_k=\frac{1}{n}({\bm 1}_n^\top\otimes{\bf I}_p)\bsx_k$, $\bar{\bsx}_k={\bm 1}_n\otimes\bar{x}_k$,   
	$\tilde{f}(\bsx_k)=\sum_{i=1}^{n}f_i(x_{i,k})$, $\bsg_k=\nabla\tilde{f}(\bsx_k)$,  $\bsg^0_k=\nabla\tilde{f}(\bar{\bsx}_k)$, 
	 $\bar{\bsg}_k^0=\bsH\bsg^0_{k}={\bm 1}_n\otimes\nabla f(\bar{x}_k)$, $\bsH=\frac{1}{n}{\bm 1}_n{\bm 1}_n^\top\otimes{\bf I}_p$,
 $\bsE=E\otimes{\bf I}_p$, and $E={\bf I}_n-\tfrac{1}{n}{\bm1}_n{\bm1}_n^{\top}$. 	
For convenience, constants used throughout the paper are provided in Appendix~\ref{appendix:constant:Thm}, 
including \(\kappa\), \(\varepsilon\), \(\tilde{\varepsilon}\), \(a\), \(\tilde a\), \(\tilde a'\), \(b\), \(c\), and \(\tilde d\) families with arbitrary subscripts or arguments.

\section{Existing Assumptions and Challenges} \label{Section:Existing}

In this section, we review the existing assumptions in distributed zeroth-order (ZO) optimization 
and discuss the challenges when trying to relax or remove them.


We first explain why some widely used assumptions on data homogeneity, $\mathcal{O}(pn)$ function evaluations per iteration, and the P--L condition
are adopted,
from the perspective of bounding the estimation error in the two-point zeroth-order gradient estimator. 
\cite{Yi_Zerothorder_2022} proposed the following random gradient estimator, 
which estimates the gradient of the local cost function $f_i(x_i)$ at the $k$-th iteration 
using function evaluations of the stochastic realization $F_i(x_i,\xi_i)$:
\begin{align}
  g^z_{i,k}=\frac{p(F_i(x_{i,k}+ \mu _{i,k} \zeta_{i,k},\xi_{i,k})-F_i(x_{i,k},\xi_{i,k}))}{ \mu _{i,k}}\zeta_{i,k}, \label{dbco:gradient:model2-st}
\end{align}
where 
$\mu_{i,k}>0$ is an exploration parameter and $\zeta_{i,k}\in\mathbb{S}^p$ is a uniformly distributed random vector indicating the distance and direction of the two sampling points, respectively;
$\xi_{i,k}\in\Xi_i$ is the local data sample drawn from the distribution~$\mathcal{D}_i$.

It has been shown that $g^z_{i,k}$ is an unbiased estimator of 
the gradient of a smoothed version of $f_i$. 
Specifically,
$\mathbb{E}_{\xi_{i,k},\zeta_{i,k}\in\mathbb{S}^p}[g^z_{i,k}]
  = \nabla \hat f_i(x_{i,k},\mu_{i,k}),$
  where 
$\hat f_i(x,\mu) =
\mathbb{E}_{\hat{\zeta}\in\mathbb{B}^p}[f_i(x+\mu\hat{\zeta})]$.
However, $g^z_{i,k}$ is generally biased with respect to the true gradient $\nabla f_i(x_{i,k})$.
A more challenging issue lies in the non-vanishing variance of the estimator $g^z_{i,k}$.  
Even when the stochastic function $F_i(\cdot,\xi_{i,k})$ is $\ell$-smooth, 
the variance of $g^z_{i,k}$ remains non-vanishing and is bounded by its mean square
\begin{align}\label{zerosg:lemma:uniformsmoothing-equ6-1}
  \mathbb{E}_{\zeta_{i,k}}[\|g^z_{i,k}\|^2]
  \le 2p\|\nabla_x F_i(x_{i,k},\xi_{i,k})\|^2
  + \tfrac{1}{2}p^2\mu_{i,k}^2\ell^2.
\end{align}
In the distributed optimization setting, 
$\|\nabla_x F_i(x_{i,k},\xi_{i,k})\|^2$ typically does not vanish 
even when the system reaches the global optimum, 
since $\nabla f_i(x) \neq \nabla f(x)$ due to data heterogeneity among agents.
Furthermore, the above variance bound is amplified by the problem dimension~$p$, 
which makes the convergence analysis of distributed zeroth-order algorithms 
particularly challenging.

\subsection{Data homogeneity}

To obtain a tractable bound on the variance of the gradient estimator \eqref{zerosg:lemma:uniformsmoothing-equ6-1}, 
existing works typically assume either the Lipschitz condition on local cost functions
or data homogeneity among agents, 
which constitute the first class of assumptions in the literature.

\begin{assumption}\label{ass:data homo}
The following assumptions are commonly imposed in existing literature, e.g., \cite{yuan_randomized_2015,Yi_Zerothorder_2022,Wang_CZSD_2025,Tang_Zeroth_2020,Ling_Federated_2024},
listed from strong to weak, and assumed to hold for all \(i \in [n]\), \(x \in \mathbb{R}^p\), and \(\xi\) drawn from the corresponding data distribution.%
\begin{subequations}\label{eq:gradient-assumptions}
\begin{align}
&(\text{Lipschitz condition}) 
  &&   \hspace{-1.7em}\|\nabla_x F_i(x,\xi)\| \le G, \label{eq:grad-assump-lip}\\
&(\text{Weak Lipschitz condition}) \quad & \nonumber\\
  &&&   \hspace{-8.1em}\|\nabla_x F_i(x,\xi)\| \le \hat{G}(\xi),~\mathbb{E}_{\xi}[\hat{G}(\xi)^2]\le\tilde{G}^2, \label{eq:grad-assump-lip:xi}\\
&(\text{Strong data homogeneity}) && 
 \hspace{-1.7em} \|\nabla f_i(x)-\nabla f(x)\|^2 \le \sigma^2, \label{eq:grad-assump-strong}\\
&(\text{Weak data homogeneity}) \quad & \nonumber\\
  &&& \hspace{-8.1em} \|\nabla f_i(x)-\nabla f(x)\|^2
    \le \eta^2 \|\nabla f(x)\|^2 + \sigma^2,
    \label{eq:grad-assump-weak}
\end{align}
\end{subequations}
where $G,\tilde{G}>0$, \( \hat{G}(\xi) \) is a nonnegative function that depends only on \( \xi \),  and $\sigma, \eta \ge 0$ are constants that characterize 
the degree of data heterogeneity.
\end{assumption}

With the Lipschitz condition, 
the variance of the estimator $g^z_{i,k}$ can be properly bounded, 
which is essential for establishing the convergence of distributed zeroth-order methods. 
Similarly, the data homogeneity assumptions
 ensure that the variance can be bounded in terms of 
$\|\nabla f(x)\|$. 
As the system approaches a stationary point when $\|\nabla f(x)\|\to0$, 
the variance consequently diminishes, leading to a tractable convergence analysis.

However, the Lipschitz condition does not hold even for simple quadratic functions, and the weak Lipschitz condition does not differ substantially in practice.
Moreover, the data homogeneity assumptions require not only the local data distributions 
to be close to each other, but also the local cost functions 
to share similar functional forms.
When this consistency does not exist, the assumptions are no longer valid 
or hold only with very large constants.
In fact, the weak data homogeneity can fail even for convex,
smooth, and bounded-below cost functions,
as illustrated by the following example.

\begin{example}\label{ex:weak-homogeneity}
Consider a binary classification problem with the squared hinge loss
and the linear classifier
$h(x;\lowtildezeta)\hspace{-0.1em}=\hspace{-0.1em}x_1\lowtildezeta_1+x_2\lowtildezeta_2+x_3$,
where $x\hspace{-0.em}=\hspace{-0.em}(x_1,x_2,x_3)^{\hspace{-0.15em}\top}\hspace{-0.2em}\in\hspace{-0.1em}\mathbb{R}^3$ is the model parameter
and $\lowtildezeta\hspace{-0.2em}=\hspace{-0.2em}(\lowtildezeta_1,\lowtildezeta_2)^{\hspace{-0.15em}\top}\hspace{-0.2em}\in\hspace{-0.1em}\mathbb{R}^2$ is the input feature.
The training data $\xi\hspace{-0.2em}=\hspace{-0.2em}(\lowtildezeta,y)$
with class label $y\hspace{-0.2em}\in\hspace{-0.2em}\{-1,1\}$ are distributed across two agents.
One has only positive samples
$\xi_1\hspace{-0.2em}=\hspace{-0.2em}((s^2,s)^{\hspace{-0.15em}\top}\hspace{-0.3em},1)$ on a right-opening parabola,
and the other only negative samples  
$\xi_2\hspace{-0.2em}=\hspace{-0.2em}\allowbreak((-s^2,-s)^{\hspace{-0.15em}\top}\hspace{-0.3em},-1)$ on a left-opening parabola,
where $s$ is uniformly distributed over $[-1,1]$.
Their local cost functions based on the squared hinge loss are
\begin{align*}
f_1(x)&=\tfrac12\mathbb{E}_s[1-x_1s^2-x_2s-x_3]_+^2,\\
f_2(x)&=\tfrac12\mathbb{E}_s[1-x_1s^2-x_2s+x_3]_+^2,
\end{align*}
where $[z]_+=\max\{z,0\}$.
Both cost functions are convex, $3$-smooth, and bounded-below.
Nevertheless, the weak data homogeneity fails for
$f=(f_1+f_2)/2$, 
as the heterogeneous data distributions induce dominant local gradient components
 that cancel out in their average.
\end{example}

\begin{proof}
The squared hinge loss $\phi(z)=\tfrac12[1-z]_+^2$
is convex and $1$-smooth since $\phi'(z)=-[1-z]_+$.
Convexity of each $f_i$ follows from affine composition
and expectation, while the chain rule and $|s|\leq1$
imply that both are $3$-smooth.

To show that the weak data homogeneity assumption fails,
consider $x_r=(r^4,2r^3,0)^\top$ for $r\geq3$.
Direct calculation of the expectations in the gradient expressions gives
\[
\nabla f_1(x_r)\hspace{-0.1em}=\hspace{-0.1em}m_r(-q_r,r^{-1}\hspace{-0.2em},-1)^{\hspace{-0.1em}\top}\hspace{-0.2em},
~
\nabla f_2(x_r)\hspace{-0.1em}=\hspace{-0.1em}m_r(-q_r,r^{-1}\hspace{-0.2em},1)^{\hspace{-0.1em}\top}\hspace{-0.2em},
\]
where 
\[
m_r=\frac{2(1+r^2)^{3/2}}{3r^2},
~
q_r=\frac{6r^2+1}{5r^4}.
\]
Consequently, as $r\to\infty$,
\begin{align*}
\|\nabla f_i(x_r)-\nabla f(x_r)\|^2
&=m_r^2\rightarrow\infty,~i=1,2,\\
\|\nabla f(x_r)\|^2
&=m_r^2(q_r^2+r^{-2})\rightarrow4/9.
\end{align*}
Therefore, no finite constants $\eta, \sigma$ can satisfy
$\|\nabla f_i(x)-\nabla f(x)\|^2
\leq\eta^2\|\nabla f(x)\|^2+\sigma^2$
for all $x$.
\end{proof}

\subsection{$\mathcal{O}(\mathit{pn})$ function evaluations}
When Assumption~\ref{ass:data homo} does not hold, 
a natural alternative is to increase the number of function evaluations per iteration, 
thereby reducing the variance of the zeroth-order gradient estimator.
Accordingly, many methods have been proposed based on $\mathcal{O}(pn)$ function evaluations per iteration:
\begin{align}
&(\text{Coordinate-wise gradient estimator 
	 \cite{wang_distributed_2019,Tang_Zeroth_2020,Yi_Linear_2021}})  \nonumber\\
 &g^p_{i,k}=\sum_{j=1}^{p}\frac{F_i(x_{i,k}+ \mu _{i,k} e_j,\xi_{i,k})-F_i(x_{i,k},\xi_{i,k})}{ \mu _{i,k}}e_j, \label{g:Coordinate-wise}\\
 &\text{where $\{e_1,\dots,e_p\}$ denotes the standard basis vectors in $\mathbb{R}^p$.}  \nonumber\\
&(\text{Variance-reduced gradient estimator \cite{Mu_VRZOD_2026}})  \nonumber\\
  &  g^{vd}_{i,k} =
  \begin{cases}
  &g^{pd}_{i,k}(x_{i,k},\mu_{i,k}), 
   \text{with a small probability},\\
    &g^{1d}_{i,k}(x_{i,k},u_{i,k},j_k) 
  - g^{1d}_{i,k}(x_{i,k-1},u_{i,k-1},j_k)\\
  &\qquad~\hspace{0.27em} + g^{vd}_{i,k-1},
   \hspace{0.03em}
   \text{otherwise},
  \end{cases}     \label{g:DVR}
\end{align}
where $g^{pd}_{i,k}(x_{i,k},\mu_{i,k})$ is the deterministic version of $g^p_{i,k}$,
and $g^{1d}_{i,k}(x_{i,k},\mu_{i,k},j_k)
= p\,\frac{f_i(x_{i,k}+\mu_{i,k}e_{j_k}) - f_i(x_{i,k})}{\mu_{i,k}}\,e_{j_k}$ 
with $j_k\in[p]$ being selected uniformly at random as the index of a single coordinate direction.
Each evaluation of $g^{pd}_{i,k}$ provides an accurate reference for subsequent recursive updates.
\begin{align}
  &\hspace{-0.4em}(\text{Gaussian perturbation estimator using $\mathcal{O}(T)$ sampling \cite{Hajinezhad_ZONE_2019}})\nonumber\\
 & g^{\mathcal{N}}_{i,k}=\sum_{j=1}^{m_k}\frac{F_i(x_{i,k}\hspace{-0.15em}+\hspace{-0.15em} \mu _{i,k} \tilde{\zeta}_{i,k,j},\xi_{i,k})\hspace{-0.15em}-\hspace{-0.12em}F_i(x_{i,k},\xi_{i,k})}{ \mu _{i,k}}\tilde{\zeta}_{i,k,j}, \label{g:Gaussian}
\end{align}
where 
$\tilde{\zeta}_{i,k,j}\!\sim\!\mathcal{N}(0,I_p)$ is a standard Gaussian,
and $m_k$ is the number of sampled directions, typically of order $\mathcal{O}(T)$.

Other variants replace the canonical basis $\{e_1,\ldots,e_p\}$ 
with alternative orthogonal directions \cite{Wang_Orthogonal_2024}, 
but the overall sampling complexity remains $\mathcal{O}(pn)$ per iteration. 
There also exist variance-reduction strategies designed for sampled data, 
yet they target stochastic noise rather than zeroth-order gradient variance, 
and are therefore beyond the scope of this paper.

The zeroth-order gradient estimators in \eqref{g:Coordinate-wise}--\eqref{g:Gaussian} can approximate first-order gradients with arbitrarily small error,
enabling convergence rates comparable to those of first-order methods,
 and convergence analyses that closely follow their first-order counterparts.
However, these estimators typically require $\mathcal{O}(pn)$ or even more function evaluations per iteration,
which leads to a substantial sampling burden. 
If in \eqref{g:DVR} the $\mathcal{O}(p)$ sampling estimator $g^{pd}_{i,k}$ is computed with a probability smaller than $1/p$, 
the sampling burden can be effectively reduced.
	However, this comes at the cost of a convergence rate deteriorating from $\mathcal{O}(1/T)$ to $\mathcal{O}(p/T)$, introducing a potentially large factor of $p$.
In addition, the $\mathcal{O}(p)$ function evaluations required in certain iterations can cause fluctuations in GPU memory usage and inefficient utilization of computational resources \cite{jeon2019analysis}.
	More critically, this method is limited to deterministic settings due to fundamental challenges in stochastic extension.
Specifically, using two-point zeroth-order gradient estimators in most iterations can amplify the stochastic variance by a factor of $p$, 
	as reflected by the $p\|\nabla_x F_i(x_{i,k},\xi_{i,k})\|^2$ term in \eqref{zerosg:lemma:uniformsmoothing-equ6-1}, undermining the very goal of variance reduction.
	Furthermore, obtaining an accurate gradient reference requires $\mathcal{O}(p)$ function evaluations per data sample across the entire dataset, incurring prohibitive computational costs.

	\subsection{Polyak--\L{}ojasiewicz condition / Strong Convexity}

If Assumption~\ref{ass:data homo} fails
and the zeroth-order gradient is constructed using two-point sampling,
existing convergence analysis is available only when the global cost function satisfies the P--L condition with a known P--L constant~\cite{Yi_Zerothorder_2022}.

\begin{assumption}\label{nonconvex:ass:fil} 
	The global cost function $f(x)$ satisfies the Polyak--\L{}ojasiewicz condition, i.e., there exists a constant $\nu>0$ such that for any $x\in\mathbb{R}^p$,
\begin{align}
 \frac{1}{2}\|\nabla f(x)\|^2\ge \nu( f(x)-f^*), 
 \label{nonconvex:equ:plc}
\end{align}
where $f^* = \inf_{x\in\mathbb{R}^p}f(x)>-\infty$.
\end{assumption}
\begin{remark}
	It should be emphasized that 
	the P--L condition does not imply convexity,
	in contrast to the convexity-based assumptions adopted in
	\cite{Nedic_Distributed_2009,Boyd_ADMM_2011,Shi_EXTRA_2015,yuan_randomized_2015,wang_distributed_2019,flaxman2004online,conn_global_2009,Rabbat_Quantized_2005,Lemmon_Event_2010,nedic_asynchronous_2010,yi2014quantized,alistarh2017qsgd,koloskova2019decentralized}.
	The P--L condition 
	has been shown in various nonconvex problems,
	such as loss functions of wide neural networks~\cite{liu2022loss}.
\end{remark}

Under Assumption~\ref{nonconvex:ass:fil} with a known P--L constant,
the distributed zeroth-order algorithm proposed in~\cite{Yi_Zerothorder_2022}
is capable of handling heterogeneous data across agents. Specifically, 
\begin{subequations}\label{nonconvex:kia-algo-dc-compress}
  \begin{align}
    x_{i,k+1} &= x_{i,k}
    - \alpha_k \Big(
      \beta_k \sum\nolimits_{j=1}^{n} L_{ij} x_{j,k}
      + \gamma_k v_{i,k}
      + g^z_{i,k} \Big), \label{nonconvex:kia-algo-dc-x-compress}\\
    v_{i,k+1} &= v_{i,k}
    + \alpha_k \gamma_k \sum\nolimits_{j=1}^{n} L_{ij} x_{j,k},
    \label{nonconvex:kia-algo-dc-v-compress}
  \end{align}
\end{subequations}
where 
$\alpha_k$, $\beta_k$, and $\gamma_k$ are positive algorithm parameters at iteration~$k$ with $\alpha_k$ being the stepsize,  
$L=[L_{ij}]$ is the weighted graph Laplacian of the network,
and $g^z_{i,k}$ is computed by~\eqref{dbco:gradient:model2-st}.
For convergence, take the Lyapunov function
  $V_k  = V_{c,k}+V_{d,k}+V_{x,k}+V_{o,k}$, where $V_{c,k}$ and $V_{d,k}$ respectively denote the consensus and dual terms,
$V_{x,k}$ is the cross term, and $V_{o,k} = n(f(\bar{x}_k)-f^{*})$ represents the optimality term.

The key idea of the analysis is to leverage the P--L condition to offset the variance of the two-point zeroth-order gradient estimator. 
Specifically, the variance term is first bounded by the optimality term~$V_{o,k}$,
under a mild assumption that each local function is bounded-below, i.e., $f_i^* = \inf_x f_i(x) > -\infty$.
The P--L condition then provides a negative term that counteracts the variance contribution.
Together, these steps yield a recursive Lyapunov relation that ensures convergence.
The main steps of the analysis are summarized as follows.
\begin{align} \label{difficulty:1}
V_{k+1}
&\le (1 - a_{1,k})(V_{c,k} + V_{d,k} + V_{x,k})
+ a_{2,k}V_{o,k} + V_{o,k}  \nonumber\\
&\quad  + a_{3,k} - a_{4,k} n\|\nabla f(\bar{x}_k)\|^2 \nonumber\\
&\le (1 - a_{1,k})(V_{c,k} + V_{d,k} + V_{x,k}) \nonumber\\
&\quad + (1 + a_{2,k} - 2\nu a_{4,k})V_{o,k} + a_{3,k} \nonumber\\
&\le (1 - a_{5,k})V_k + a_{3,k},
\end{align}
where $a_{i,k}$ for $k=1,...,5$ are positive constants determined by the algorithm parameters.
We note that the P--L constant~$\nu$ needs to be known to properly choose $a_{4,k}$ so that $(1+a_{2,k}-2\nu a_{4,k})<1$, ensuring a contractive Lyapunov recursion.

In the absence of the P--L condition, 
 the analytical framework is invalid
 since \eqref{difficulty:1} reduces to a divergent form
\begin{align} \label{difficulty:2}
V_{k+1} \le (1 + a_{2,k})V_k + a_{3,k}.
\end{align}
However, the P--L condition holds only for a specific subset of nonconvex functions, with strong convexity being a sufficient but more restrictive special case.
Moreover, verifying the P--L condition and determining its constant
are challenging, as $f^*$ is typically unknown and finitely many
evaluations cannot establish the existence of a positive constant
$\nu$ satisfying~\eqref{nonconvex:equ:plc} for all $x$,
let alone determine its value.
To the best of our knowledge, for general nonconvex or convex functions,
as well as for cases where the P--L condition holds but the P--L constant is unknown,
existing works lack theoretical convergence guarantees.
Furthermore, none of existing studies have simultaneously relaxed these assumptions and incorporated communication compression techniques.
These limitations motivate this paper.

\section{Problem Formulation} \label{Section:Preliminaries}




In this section, we present our problem formulation 
with some standard assumptions.

Let us study the problem \eqref{zerosg:eqn:xopt} in a heterogeneous setting, where the local data distributions $\mathcal{D}_i$ may differ significantly across agents and the local cost functions $f_i$ are not required to share the same functional form. 
To solve \eqref{zerosg:eqn:xopt} in a distributed manner, each agent~$i$ maintains a local model parameter $x_i \in \mathbb{R}^{p}$ and estimates a zeroth-order gradient using only two function evaluations of its local stochastic cost function. 
Meanwhile, the agents exchange compressed information over a communication network to cooperate. 
Together, these elements enable minimizing the global cost function while achieving consensus among all local model parameters.

Existing distributed zeroth-order methods \cite{yuan_randomized_2015,Yi_Zerothorder_2022,Wang_CZSD_2025,Tang_Zeroth_2020,Ling_Federated_2024} commonly impose the data homogeneity assumptions in Assumption~\ref{ass:data homo}.
Such assumptions can be restrictive in heterogeneous settings, as they require similarity across agents in both local data distributions $\mathcal{D}_i$ and local cost functions $f_i$.
We emphasize that our analysis does not require data homogeneity. Instead, we adopt the following mild assumption.
\begin{assumption}\label{ass:i:lower-bounded}
Each local cost function $f_i(x)$ is bounded-below, i.e.,
$f_i^* = \inf_{x \in \mathbb{R}^p} f_i(x) > -\infty,~\forall i \in [n]$.
\end{assumption}
\begin{remark}
	It should be highlighted that Assumption~\ref{ass:i:lower-bounded} 
	serves as a mild alternative to the data homogeneity assumptions (Assumption~\ref{ass:data homo}).
	This requirement is easily satisfied in practice,
	since most machine learning loss functions (e.g., squared loss, logistic loss, hinge loss) 
	are nonnegative and hence bounded-below.
	In many engineering and control scenarios,
	the cost functions are often quadratic and thus naturally fulfill this property.
	Even if a cost function take negative values, 
	the assumption remains valid as long as the infimum is finite.
\end{remark}
Under Assumption~\ref{ass:i:lower-bounded}, the global cost function
$f(x) = \frac{1}{n}\sum_{i=1}^n f_i(x)$ is also bounded-below, i.e.,
\begin{align}\label{eq:global_lower_bound}
	f^* = \inf_{x \in \mathbb{R}^p} f(x) \ge \frac{1}{n}\sum\nolimits_{i=1}^n f_i^*> -\infty. 
\end{align}


Then we discuss communication compression for information exchange among agents over a network.

	To achieve consensus and optimize the global cost function, 
	agents communicate over a communication network, 
	which is modeled as a graph 
	\( \mathcal{G} = (\mathcal{V}, \mathcal{E}, A) \).
	Here, \( \mathcal{V} = [n] \) denotes the set of agents, 
	\( \mathcal{E} \subseteq \mathcal{V} \times \mathcal{V} \) is the edge set, 
	and \( A =[A_{ij}]\) represents the weighted adjacency matrix. 
	Specifically, \( A_{ij} > 0 \) if there exists a directed edge 
	from agent \( j \) to agent \( i \) (i.e., \( (j,i) \in \mathcal{E} \)), 
	and \( A_{ij} = 0 \) otherwise.
	The neighbor set of agent \( i \) is defined as 
	\( \mathcal{N}_i = \{ j \in \mathcal{V} : (j,i) \in \mathcal{E} \} \).
	The degree matrix is given by 
	\( D = \mathrm{diag}(d_1, d_2, \dots, d_n) \),
	where \( d_i \) denotes the \( i \)-th row sum of \( A \),
	and the Laplacian matrix of the graph is defined as \( L = D - A \).
A graph is said to be undirected if \( (i,j) \in \mathcal{E} \) implies \( (j,i) \in \mathcal{E} \), and in this case \( A_{ij} = A_{ji} \).
It is connected if any two nodes are linked by a path, i.e., a sequence of edges joining them.

	\begin{assumption}\label{ass:graph}
	The communication graph \( \mathcal{G} \) is undirected and connected.
	\end{assumption}



To improve communication efficiency, we consider a general class of compressors 
exhibiting bounded relative compression error \cite{Yi_CommunicationCompression_2023}.

\begin{assumption}\label{nonconvex:ass:compression}
The (possibly stochastic) compressor $\mathcal{C}:\mathbb{R}^p \!\to\! \mathbb{R}^p$ is assumed to satisfy
\begin{align}\label{nonconvex:ass:compression_equ_scaling} 
  \mathbb{E}_{\mathcal{C}}\Big[\Big\|\frac{\mathcal{C}(x)}{r}-x\Big\|^2\Big]\le (1-\delta)\|x\|^2,~\forall x\in\mathbb{R}^p, 
\end{align}
where $\delta\in(0,1]$ and $r>0$ are given constants, and 
$\mathbb{E}_{\mathcal{C}}[\cdot]$ represents the expectation over the randomness inherent in~$\mathcal{C}$.
\end{assumption}

By the Cauchy--Schwarz inequality, we have
\begin{align}\label{nonconvex:ass:compression_equ}
\mathbb{E}_{\mathcal{C}}\!\left[\|\mathcal{C}(x)-x\|^2\right]
\le \delta_0\|x\|^2,\quad \forall x\in\mathbb{R}^p,
\end{align}
where $\delta_0 = 2r^2(1-\delta) + 2(1-r)^2$.

Assumption~\ref{nonconvex:ass:compression} covers a broad range of compressors, 
including unbiased, biased contractive, and even certain biased non-contractive ones.
 Representative examples are as follows.
\begingroup
\predisplaypenalty=-10000 
 \begin{align*} 
&(\text{Unbiased $k$-bit quantizer:}~r=1+p/4^k,~\delta=1/r)  \\
 &\mathcal{C}_1(x)\hspace{-0.1em}=\hspace{-0.1em}\frac{\|x\|_\infty}{2^{k-1}}\text{sign}(x)\hspace{-0.1em}\circ\hspace{-0.15em}
 \Big\lfloor\frac{2^{k-1}\left|{x}\right|}{\|x\|_\infty}+\varpi\Big\rfloor,b_v=(k+1)p+b_1.\\
 &(\text{Top-$k$ compressor:}~r=1,~\delta=k/p) \\
 &\hspace{6em} \mathcal{C}_2(x)=\sum\nolimits_{j=1}^{k}[x]_{t_j}e_{t_j},~b_v=kb_1. \\
  &(\text{Rand-$k$ compressor:}~r=1,~\delta=k/p) \\
 &\hspace{6em} \mathcal{C}_3(x)=\sum\nolimits_{j=1}^{k}[x]_{r_j}e_{r_j},~b_v=kb_1.\\
 &(\text{Norm-sign compressor:}~r=p/2,~\delta=1/p^2) \\
 &\hspace{6em} \mathcal{C}_4(x)=\frac{\|x\|_\infty}{2}\text{sign}(x),~b_v=p+b_1. 
\end{align*}
\endgroup
Here, $b_v$ represents the number of transmitted bits per vector, and $b_1=64$ corresponds to 64-bit floating-point representation, ensuring exact precision.
 $\varpi \sim \mathrm{Unif}([0,1]^p)$ is a random vector.
 $\text{sign}(\cdot)$, $\circ$, $\lfloor\cdot\rfloor$, and $\left|{\cdot}\right|$ are the element-wise sign, Hadamard product, floor function, and absolute value, respectively.
The index $t_j$ refers to one of the $j$-th largest-magnitude components of $x$, while $r_j$ denotes a randomly selected non-repetitive coordinate index.


Finally, we make the following standard assumptions on problem~\eqref{zerosg:eqn:xopt} throughout the paper.

\begin{assumption}\label{zerosg:ass:zeroth-smooth}
	For almost all $\xi_i$, the local stochastic cost function $F_i(\cdot,\xi_i)$ is $\ell$-smooth, i.e.,  there exists a constant $\ell>0$ such that for any $i\in[n]$ and $x,y\in\mathbb{R}^p$,
	\begin{align}\label{nonconvex:smooth}
		\|\nabla F_i(x,\xi_i)-\nabla F_i(y,\xi_i)\|\le \ell\|x-y\|.
	\end{align}
\end{assumption}
\begin{assumption}\label{zerosg:ass:zeroth-variance}
	The stochastic gradients have bounded state-dependent variance, i.e., there exist constants $\eta_1, \sigma_1 \ge 0$ such that for any $i \in [n]$ and $x \in \mathbb{R}^p$,
	\begin{align}
		&\mathbb{E}_{\xi_i}[\|\nabla_xF_i(x,\xi_i)-\nabla f_i(x)\|^2]\le\eta^2_1\|\nabla f_i(x)\|^2+\sigma^2_1.
	\end{align}
\end{assumption}
Assumptions~\ref{zerosg:ass:zeroth-smooth} and \ref{zerosg:ass:zeroth-variance} are standard in stochastic optimization, e.g., \cite{Yi_Zerothorder_2022,Wang_CZSD_2025,Ling_Federated_2024}. 
Assumption~\ref{zerosg:ass:zeroth-variance} reduces to the bounded variance assumption when $\eta_1=0$ \cite{Yi_PrimalDual_2022,wang_objective_2020,lian2017can}, 
and is more general than the Lipschitz condition \cite{yuan_randomized_2015,koloskova2019decentralized,basu2020qsparse}.



\section{Algorithm Design} \label{section:Algorithm}

In this section, we propose a \underline{He}terogeneous \underline{D}istributed \underline{Z}eroth-\underline{O}rder \underline{C}ompressed (HEDZOC) algorithm. 
HEDZOC integrates the distributed zeroth-order algorithm~\eqref{nonconvex:kia-algo-dc-compress} 
with the compressors in Assumption~\ref{nonconvex:ass:compression}. 
However, directly combining these two components often results in large compression error, 
making it difficult for the algorithm to converge.
To mitigate this error, we employ an indirectly compressed variable~$\hat{x}_{i,k}$ 
to replace~${x}_{i,k}$ in~\eqref{nonconvex:kia-algo-dc-compress}, instead of directly using~$\mathcal{C}(x_{i,k})$.
 Specifically, we construct the compressed variable as
\begin{align}
\hat{x}_{i,k} = y_{i,k} + \mathcal{C}(x_{i,k} - y_{i,k}), \label{nonconvex:kia-algo-dc-compact-xhat}
\end{align}
where the auxiliary variable $y_{i,k}$ serves to alleviate the compression error.
Combining~\eqref{nonconvex:ass:compression_equ} with~\eqref{nonconvex:kia-algo-dc-compact-xhat} yields
\begin{align}
  \mathbb{E}_{\mathcal{C}}[\|x_{i,k} - \hat{x}_{i,k}\|^2]
  &= \mathbb{E}_{\mathcal{C}}[\|x_{i,k} - y_{i,k} - \mathcal{C}(x_{i,k} - y_{i,k})\|^2] \nonumber\\
  &\le \delta_0\mathbb{E}_{\mathcal{C}}[\|x_{i,k} - y_{i,k}\|^2].
\end{align}
Hence, the compression error $\|x_{i,k} - \hat{x}_{i,k}\|^2$ vanishes as $y_{i,k}$ approaches $x_{i,k}$.
To further reduce communication overhead, we introduce another auxiliary variable $z_{i,k}$ to compute $\sum_{j=1}^{n} L_{ij} y_{j,k}$.
This allows each agent to transmit only $\mathcal{C}(x_{i,k} - y_{i,k})$, rather than sending both $y_{i,k}$ and $\mathcal{C}(x_{i,k} - y_{i,k})$ as required for transmitting $\hat{x}_{i,k}$.
	Based on \eqref{nonconvex:kia-algo-dc-compress} and \eqref{nonconvex:kia-algo-dc-compact-xhat}, we propose HEDZOC (Algorithm~\ref{nonconvex:algorithm-pdgd}).
Although the algorithm does not explicitly handle data heterogeneity,
our main contribution lies in its convergence analysis, which is challenging due to 
	 the absence of data homogeneity and $\mathcal{O}(pn)$ function evaluations, 
	 as well as the coupling between the zeroth-order gradient estimation and communication
	  compression errors.

	\begin{algorithm}[!tb]
		\caption{\underline{He}terogeneous \underline{D}istributed \underline{Z}eroth-\underline{O}rder \underline{C}ompressed (HEDZOC) Algorithm}
		\label{nonconvex:algorithm-pdgd}
		\begin{algorithmic}[1]
			\STATE \textbf{Input}: positive sequences $\{\alpha_k\}$, $\{\beta_k\}$, $\{\gamma_k\}$ and $\{ \mu _{i,k}\}$; positive parameter $\omega$.
			\STATE \textbf{Initialize}: 
			 $ x_{i,0}\in\mathbb{R}^p$ and $v_{i,0}\hspace{-0.02em}=\hspace{-0.02em}y_{i,0}\hspace{-0.02em}=\hspace{-0.02em}z_{i,0}\hspace{-0.02em}=\hspace{-0.02em}{\bf 0}_p$, $\forall i\in[n]$.\par
			\FOR{$k=0,1,\dots$}
			\FOR{$i=1,\dots,n$ \textbf{in parallel}} 
			\STATE \textbf{Compression}:				
			\vspace{-0.8em} 
			\begin{subequations}\label{zerosg:algorithm-random-pd}
				\begin{align}
					q_{i,k}&=\mathcal{C}(x_{i,k}-y_{i,k}). \label{nonconvex:kia-algo-dc-q}
				\end{align}\vspace{-1.4em} 
			\STATE  \textbf{Communication}: Send $q_{i,k}$ to $\mathcal{N}_i$ and receive $q_{j,k}$ from $j\in\mathcal{N}_i$.
			\STATE \textbf{Stochastic zeroth-order gradient}:\\
			(i) Sample $\xi_{i,k} \sim \mathcal{D}_i$;\\
			(ii) Sample $\zeta_{i,k} \sim \mathrm{Unif}(\mathbb{S}^p)$;\\
			(iii)  Sample $F_i(x_{i,k},\xi_{i,k})$, $F_i(x_{i,k}+\mu _{i,k}\zeta_{i,k},\xi_{i,k})$;\\
			(iv)  Compute $g^z_{i,k}$ using \eqref{dbco:gradient:model2-st}.
			\STATE \textbf{Update auxiliary variables}:
				\vspace{-0.6em} 
				\begin{align}
					y_{i,k+1}&=y_{i,k}+\omega q_{i,k}, \label{nonconvex:kia-algo-dc-a}\\
					z_{i,k+1}&=z_{i,k}+\omega\sum\nolimits_{j=1}^{n}L_{ij}q_{j,k}. \label{nonconvex:kia-algo-dc-b}
				\end{align}\vspace{-1.2em} 
			\STATE \textbf{Update primal and dual variables}:
			\vspace{-0.4em} 				
			\begin{align}
					x_{i,k+1} &= x_{i,k}-\alpha_k\beta_k\Big(z_{i,k}+\sum\nolimits_{j=1}^{n}L_{ij}q_{j,k}\Big)\nonumber\\
					&\qquad-\alpha_k(\gamma_k v_{i,k}+g^z_{i,k} 	), \label{nonconvex:kia-algo-dc-x}\\
					v_{i,k+1} &=v_{i,k}+ \alpha_k\gamma_k\Big(z_{i,k}+\sum\nolimits_{j=1}^{n}L_{ij}q_{j,k}\Big).  \label{nonconvex:kia-algo-dc-v}
				\end{align}\vspace{-1.2em} 
			\end{subequations}
			\ENDFOR
			\ENDFOR
			\STATE  \textbf{Output}: $\{x_{i,k}\}$.
		\end{algorithmic}
	\end{algorithm}

	\section{Preliminary Convergence Analysis} \label{Section:Proof}

\begin{figure*}[ht]
  \centering
  \includegraphics[width=0.88\textwidth]{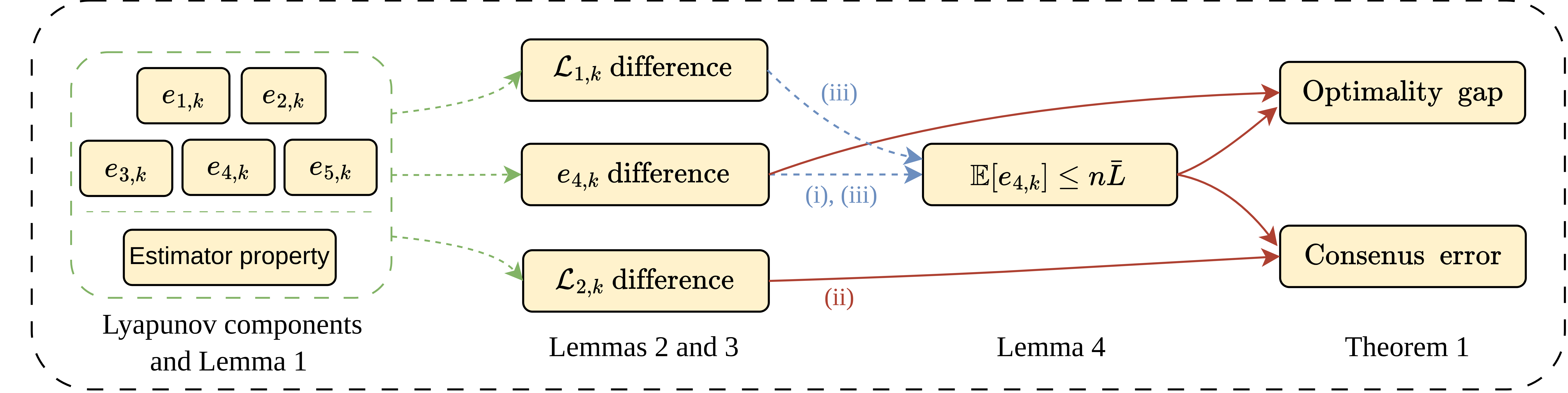}
  {\captionsetup{justification=centering}
	\caption{Logical flow of of the preliminary convergence analysis and the proof of Theorem~\ref{Thm:nonconvex}, where (i)--(iii) correspond to the three key techniques discussed in Section~\ref{Section:Proof}.}
   \label{Fig:proof}}
\end{figure*}

In this section, we provide a preliminary analysis of Algorithm~\ref{nonconvex:algorithm-pdgd} under data heterogeneity in the nonconvex setting.
As discussed in Section~\ref{Section:Existing}, existing relatively strong assumptions in the literature are primarily introduced to bound the zeroth-order gradient estimation error
for convergence analysis.
However, it
is not directly available in our heterogeneous distributed two-point zeroth-order optimization setting.
Accordingly, the most crucial technique lies in proving that the variance of the two-point zeroth-order gradient estimator remains bounded under data heterogeneity, without relying on these assumptions.
To this end, we scale the estimator variance to the optimality gap, view it as a perturbation term, and then design a special stepsize to regulate its growth. Finally, the boundedness is established by mathematical induction.

Inspired by \cite{Yi_Zerothorder_2022,Yi_CommunicationCompression_2023}, 
we construct a Lyapunov function  
\begin{align*}
	\mathcal{L}_{1,k} = \sum\nolimits_{i=1}^{5}e_{i,k}.
\end{align*}
The component terms are defined as follows.
\begin{align*}
	&{\rm (Consensus~error)}  &e_{1,k}&=\tfrac{1}{2}\|\bsx_{k}\|^2_{\bsE},\\
	&{\rm (Dual~term)}        &e_{2,k}&=\tfrac{1}{2}\Big\|\bsv_{k}+\tfrac{1}{\gamma_{k}}\bsg_{k}^0\Big\|^2_{\frac{\beta_k+\gamma_k}{\gamma_k}\bsF},\\
	&{\rm (Cross~term)}       &e_{3,k}&=\bsx_{k}^\top\bsE\bsF\Big(\bsv_{k}+\tfrac{1}{\gamma_{k}}\bsg_{k}^0\Big),\\
	&{\rm (Optimality~gap)}       &e_{4,k}&=\tilde{f}(\bar{\bsx}_{k})-nf^*,\\
	&{\rm (Compression~error)}&e_{5,k}&=\|\bsx_{k}-\bsy_{k}\|^2.
\end{align*}
Here, 
$\bsy_k=\col(y_{1,k},\dots,y_{n,k})$ and
$\bsF=F_M\otimes{\bf I}_p$.
The matrix $F_M$ is given by
\begin{align*}
	F_M=\begin{bmatrix}q & Q\end{bmatrix}
	\begin{bmatrix}
		\lambda_{n+1}^{-1}&0\\
		0&\Lambda_1^{-1}
	\end{bmatrix}
	\begin{bmatrix}
		q^\top\\ Q^\top
	\end{bmatrix},
\end{align*}
where $\Lambda_1=\mathrm{diag}([\lambda_2,\dots,\lambda_n])$ and 
$[q\ Q]\in\mathbb{R}^{n\times n}$ arise from the eigendecomposition of the Laplacian matrix $L$
with $0<\lambda_2\le\dots\le\lambda_n$ being its nonzero eigenvalues, 
$q=\tfrac{1}{\sqrt{n}}\mathbf{1}_n$ corresponds to the zero eigenvalue, 
$Q\in\mathbb{R}^{n\times(n-1)}$ collects the remaining orthonormal eigenvectors, 
and $\lambda_{n+1}$ is an arbitrary constant in $[\lambda_2,\lambda_n]$.

We first rescale the estimator variance in terms of the optimality gap $e_{4,k}$ using Lemma~\ref{Lemma:grad prop},
and derives the iterative difference of $\mathcal{L}_{1,k}$ as a preparatory step in Lemma~\ref{Lemma:Lyap:fix}. The main analysis is then enabled by the following three key techniques.

(i) We  
leverage the \(1/n\) scaling introduced by averaging across \(n\) agents, which effectively reduces the variance in the analysis (Lemma~\ref{zerosg:lemma:sg2-T:23}).

(ii) We consider another Lyapunov function by excluding $e_{4,k}$ while retaining the other components, i.e.,
\begin{align*}
	\mathcal{L}_{2,k} = \sum\nolimits_{i=1}^{3} e_{i,k} + e_{5,k},
\end{align*} 
 thereby mitigating $e_{4,k}$'s influence  
 and facilitating a contractive recursion structure (Lemma~\ref{zerosg:lemma:sg2-T:23}).
This will also be used in the next section to analyze the consensus error $e_{1,k}$ in \eqref{consensus:bound}.

 (iii) More importantly, by designing a special stepsize to regulate the growth of 
$e_{4.k}$, we apply mathematical induction to its recursive relation to rigorously establish its boundedness (Lemma~\ref{L3_bound}).

Finally, the boundedness of $e_{4,k}$ implies bounded variance of the two-point zeroth-order gradient estimator.
This result is essential for establishing convergence of the Lyapunov function, 
and serves as a key technical step toward the convergence guarantee in the general nonconvex setting (Theorem~\ref{Thm:nonconvex}) presented in the next section.
The overall logical flow of the preliminary convergence analysis and the proof of Theorem~\ref{Thm:nonconvex} is illustrated in Fig.~\ref{Fig:proof}.


\subsection{Lyapunov Analysis}


This subsection analyzes the iterative difference of the Lyapunov function $\mathcal{L}_{1,k}$.
The derivations involve additional technical ingredients compared with previous studies in two aspects:
(i) variance rescaling for the two-point zeroth-order gradient estimator, and (ii) decoupling between zeroth-order gradient estimation and communication compression errors.

Our treatment of variance rescaling for the zeroth-order gradient estimator differs from that in prior works.
Specifically, under Assumptions~\ref{ass:i:lower-bounded}, \ref{zerosg:ass:zeroth-smooth} and \ref{zerosg:ass:zeroth-variance},
 we bound the variance of the two-point zeroth-order gradient estimator in terms of the optimality gap $e_{4,k}$ in \eqref{zerosg-a5:rand-grad-esti2}.
The following lemma formalizes this bound and collects several basic properties of the two-point zeroth-order gradient estimator used in the main proofs.

\begin{lemma}\label{zerosg:lemma:grad-st} \label{Lemma:grad prop}
	Under Assumption~\ref{zerosg:ass:zeroth-smooth}, let $\{\bsx_k\}$ be the sequence generated by Algorithm~\ref{nonconvex:algorithm-pdgd}. 
	Then
	\begin{subequations}
		\begin{align}
		\mathbb{E}_{\mathcal{B}_k}[\bsg_k^z] 
		&= \bsg^\mu_k, \label{zerosg:rand-grad-esti1}\\[4pt]
		\|\bsg_k^0-\bsg^\mu_k\|^2 
		&\le 2\ell^2\|\bsx_k\|_{\bsE}^2 + 2n\ell^2\mu_k^2, \label{zerosg:rand-grad-esti8}\\
		\|\bar{\bsg}_k^0-\bar{\bsg}^\mu_k\|^2 
		&\le 2\ell^2\|\bsx_k\|_{\bsE}^2 + 2n\ell^2\mu_k^2, \label{zerosg:rand-grad-esti9}\\
		\mathbb{E}_{\mathcal{B}_k}\big[\|\bar{\bsg}_k^z\|^2\big] 
		&\le \tfrac{1}{n}\,\mathbb{E}_{\mathcal{B}_k}\big[\|\bsg_k^z\|^2\big] + \|\bar{\bsg}^\mu_k\|^2, \label{zerosg:rand-grad-esti5}\\
		\mathbb{E}_{\mathcal{B}_k}\big[\|\bsg_k^0-\bsg_k^z\|^2\big] 
		&\le 4\ell^2\|\bsx_k\|_{\bsE}^2 + 4n\ell^2\mu_k^2 
			+ 2\,\mathbb{E}_{\mathcal{B}_k}\big[\|\bsg_k^z\|^2\big], \label{zerosg:rand-grad-esti6}\\
		\|\bsg^0_{k+1}-\bsg^0_k\|^2 
		&\le \alpha_k^2\ell^2\|\bar{\bsg}_k^z\|^2 
		\le \alpha_k^2\ell^2\|\bsg_k^z\|^2, \label{zerosg:gg-rand-pd}\\
		\|\bar{\bsg}^0_k\|^2 
		&\le 2n\ell\big(f(\bar{x}_k)-f^*\big). \label{zerosg:rand-grad-smooth}\\
		&\hspace{-8.3em}\text{If Assumption~\ref{ass:i:lower-bounded} and \ref{zerosg:ass:zeroth-variance} also hold, then}\nonumber\\
		\mathbb{E}_{\mathcal{B}_k}[\|\bsg_k^z\|^2] 
		&\le 16p(1+\eta_1^2)\ell e_{4,k} + \tfrac{1}{2}np^2\ell ^2\mu_k^2 \nonumber\\
			+ 8np(1+\eta_1^2)\check{\sigma}_2^2 
			 + &4np\sigma_1^2 
			+ 8p(1+\eta_1^2)\ell ^2\|\bsx_k\|_{\bsE}^2,
			\label{zerosg-a5:rand-grad-esti2}\\
		\|\bsg^0_{k+1}\|^2 
			&\le 2\big(\alpha_k^2\ell ^2\|\bsg_k^z\|^2 + 2\ell e_{4,k} + n\check{\sigma}_2^2\big), \label{zerosg-a5:vkLya-2}
		\end{align}
	\end{subequations}
	where $\check{\sigma}^2_2=2\ell f^*-\frac{2\ell }{n}\sum_{i=1}^{n}f_i^*\ge0$, $\bsg_k^z=\col(g^z_{1,k},\dots,g^z_{n,k})$, 	
		 $\bar{\bsg}^z_k=\bsH\bsg^z_k$,
		 $g^\mu_{i,k}=\nabla \hat{f}_{i}(x_{i,k},\mu_{i,k})$, 
$\bsg^\mu_k=\col(g^\mu_{1,k},\dots,g^\mu_{n,k})$, 
$\bar{\bsg}^\mu_k=\bsH\bsg^\mu_k$, 
$\mu_k=\max_{i\in[n]}\{\mu_{i,k}\}$, and  $\mathcal{B}_k$ denotes the $\sigma$-algebra generated by the independent random variables 
		$\xi_{1,k}, \dots, \xi_{n,k}, \zeta_{1,k}, \dots, \zeta_{n,k}$.
	\end{lemma}
\begin{proof}
See Appendix~\ref{proof:lemma1}.
\end{proof}

Each component term of the Lyapunov function, i.e., $e_{1,k}$--$e_{5,k}$,
couples the zeroth-order gradient estimation error with the communication compression error due to the coexistence of the compressed variable $q_{i,k}$ and  the zeroth-order gradient estimator $g^z_{i,k}$.
To handle this issue, we separate the randomness of compression and gradient estimation at each iteration by conditioning on the $\sigma$-algebras $\mathcal{C}_k$ and $\mathcal{B}_k$, which enables their decoupling in the analysis.
Specifically,  $\mathcal{C}_k$ denotes the $\sigma$-algebra generated by the randomness of the compressor in the $k$-th iteration,
		and we define $\mathcal{A}_k = \mathcal{B}_k \cup \mathcal{C}_k$.


Leveraging Lemma~\ref{Lemma:grad prop}, decoupling the zeroth-order gradient estimation and communication compression errors, 
and choosing suitable algorithm parameters, 
we derive the iterative differences of 
$\mathcal{L}_{1,k}$ in Lemma~\ref{Lemma:Lyap:fix}.

\begin{lemma}\label{Lemma:Lyap:fix}
		Under  Assumptions~\ref{ass:i:lower-bounded}--\ref{zerosg:ass:zeroth-variance}, take $\alpha_k=\alpha=\epsilon_2/\gamma$, $\beta_k=\beta=\epsilon_1\gamma$, $\gamma_k=\gamma$, and $\omega\le 1/r$, where
	$\epsilon_1>\kappa_1$, $\epsilon_2\in(0,\kappa_2(\epsilon_1))$, and $\gamma\ge\tilde{\kappa}_0(\epsilon_1,\epsilon_2)$ for some positive constants
$\kappa_1$, $\kappa_2(\epsilon_1)$, and $\tilde{\kappa}_0(\epsilon_1,\epsilon_2)$. Let $\{\bsx_k\}$ be the sequence generated by Algorithm~\ref{nonconvex:algorithm-pdgd}. 
	Then
	\begin{align}
	&\mathbb{E}_{\mathcal{A}_k}[\mathcal{L}_{1,k+1}]\nonumber\\
	&\le   \mathcal{L}_{1,k}- a_1\|\bsx_k\|^2_{\bsE}
	-2 a_2\Big\|\bsv_k+\frac{1}{\gamma}\bsg_{k}^0\Big\|^2_{\bsF} -2 a_3\|\bsx_{k}-\bsy_{k}\|^2 \nonumber\\
	&\quad -\frac{1}{4}\alpha\|\bar{\bsg}^0_{k}\|^2  + p\tilde{a}_8\alpha^2e_{4,k} + pn\tilde{ a}_4\alpha^2+pn\tilde{ a}_5\alpha\mu_k^2.
	\label{zerosg:sgproof-vkLya2T}
	\end{align}

	\end{lemma}
	\begin{proof}
	See Appendix~\ref{appendix:iterative:Lemma}.
\end{proof}

Inequality~\eqref{zerosg:sgproof-vkLya2T} faces a similar difficulty to \eqref{difficulty:2}.
Without the P--L condition, the term $-\frac{1}{4}\alpha\|\bar{\bsg}^0_{k}\|^2$ may not directly guarantee the convergence of $\mathcal{L}_{1,k}$. 
Therefore, the optimality gap $e_{4,k}$ is treated as a perturbation term, and a proper bound on $p\tilde{a}_8\alpha^2 e_{4,k}$ must be established.

\subsection{Bound on the Optimality Gap} \label{section:bound}

	This subsection establishes an upper bound for the optimality gap $ e_{4,k} $. 
The analysis is nontrivial, as $e_{4,k}$ is inherently non-vanishing in the nonconvex setting by definition.
More critically, $e_{4,k}$ cannot be directly enforced to be non-increasing, reflecting  the same underlying difficulty as in~\eqref{difficulty:2} and \eqref{zerosg:sgproof-vkLya2T}.
To address this issue, we first exploit the $1/n$-scaling property \eqref{zerosg:rand-grad-esti5} introduced by averaging across \(n\) agents to restrain the growth of $e_{4,k}$,
and adopt the Lyapunov function $\mathcal{L}_{2,k} $ that excludes $e_{4,k}$ to suppress its influence.

	\begin{lemma}\label{zerosg:lemma:sg2-T:23}
		Under the same conditions as in Lemma~\ref{Lemma:Lyap:fix}, let $\{\bsx_k\}$ be the sequence generated by Algorithm~\ref{nonconvex:algorithm-pdgd}.
		Then
		\begin{subequations} 
			\begin{align}
				&\mathbb{E}_{\mathcal{A}_k}[e_{4,k+1}]
				\le  e_{4,k} + \frac{p}{n}a_9\alpha^2e_{4,k} + \|\bsx_k\|^2_{2\alpha  \ell ^2\bsE}\nonumber\\
				&\quad  -\frac{1}{4}\alpha\|\bar{\bsg}_{k}^0\|^2 + p a_7\alpha^2
					+(n+p) \ell ^2\alpha\mu^2_k, \label{zerosg:v4kspeed}\\
				&\mathbb{E}_{\mathcal{A}_k}[\mathcal{L}_{2,k+1}]\nonumber\\
				& \le   \mathcal{L}_{2,k}- a_1\|\bsx_k\|^2_{\bsE}-2 a_2\Big\|\bsv_k+\frac{1}{\gamma}\bsg_{k}^0\Big\|^2_{\bsF} - 2a_3\|\bsx_{k}-\bsy_{k}\|^2 \nonumber\\
				&\quad + a_6\alpha^2\|\bar{\bsg}^0_{k}\|^2  + p\tilde{a}_8\alpha^2e_{4,k} + pn\tilde{ a}_4\alpha^2+pn\tilde{ a}_5^{\prime}\alpha\mu_k^2.
				\label{zerosg:sgproof-vkLya2T-bounded}
			\end{align}
	\end{subequations}
	\end{lemma}
\begin{proof}
	See Appendix~\ref{appendix:proof:lemma3}.
\end{proof}

 The introduction of the factor $1/n$ for the second term on the right-hand side of \eqref{zerosg:v4kspeed} effectively
 slows down the growth of $e_{4,k}$ in the analysis.
Meanwhile,~\eqref{zerosg:sgproof-vkLya2T-bounded} admits a contractive reformulation,
since $\mathcal{L}_{2,k}$ can be upper bounded by
$\|\bsx_k\|_{\bsE}^2+\|\bsv_k+\frac{1}{\gamma}\bsg_k^0\|_{\bsF}^2+\|\bsx_k-\bsy_k\|^2$ by definition.

Building on Lemmas~\ref{Lemma:Lyap:fix} and \ref{zerosg:lemma:sg2-T:23}, we carefully select $\alpha_k$ to regulate the growth of $e_{4,k}$ and establish its boundedness via mathematical induction.


\begin{lemma} \label{L3_bound}
	Under Assumptions~\ref{ass:i:lower-bounded}--\ref{zerosg:ass:zeroth-variance}, consider the sequence $\{\bsx_k\}$ generated by Algorithm~\ref{nonconvex:algorithm-pdgd} 
	with	
	\begin{align}\label{zerosg:step:eta2-sm}
		&~\alpha_k=\epsilon_3\frac{\sqrt{n}}{\sqrt{pT}}, ~\beta _k=\epsilon_1\gamma _k, ~\gamma _k=\frac{\epsilon_2}{\alpha_k}, ~\mu _{i,k}\le\frac{\kappa_{\mu}\sqrt{p\alpha _k}}{\sqrt{n+p}}, \nonumber\\
		&~~~\omega\le\frac{1}{r}, ~T\ge\max\{\frac{n\epsilon_3^2(\tilde{\kappa}_0(\epsilon_1,\epsilon_2))^2}{p\epsilon_2^2},~\frac{n^3\tilde{\kappa}_T(\epsilon_1,\epsilon_2)}{p}\}, \nonumber\\
		&~\quad\epsilon_1>\kappa_1, ~\epsilon_2\in ( 0,\kappa_2( \epsilon_1 ) ),~\epsilon_3\in[\underline{\kappa}_3,\bar{\kappa}_3],~\kappa_{\mu}>0, 
	\end{align}	
	where $\tilde{\kappa}_T(\epsilon_1,\epsilon_2)$, $\underline{\kappa}_3$ and $\bar{\kappa}_3$ are some  positive constants. 
	 Then there exists a constant $\bar{L}>0$ such that
	\begin{align}
		&\qquad \qquad~~~~ \mathbb{E}[e_{4,k}]\le n\bar{L},~\forall k\in[0,T], \label{L3_bound_ineq}\\
		&~\bar{L} \le 4a_{11}\bar{\kappa}_3^2 + 4\bar{\kappa}_3^3 
				+ 4\frac{e_{4,0}}{n} + 8\ell^2\bar{\kappa}_3\frac{{\mathcal{L}}_{1,0}}{n^2} = \mathcal{O}(1). \label{Lbar:bound}
	\end{align}
\end{lemma}

\begin{proof}

		We prove \eqref{L3_bound_ineq} by  mathematical induction. 

		For $k=0$, the result follows directly from the definition.

		For $k>0$, suppose that the statement holds for $\tau = 0,1,\ldots,k-1$, 
		namely,
		\[
			\mathbb{E}[e_{4,\tau}] \le n\bar{L}, \quad \forall\, \tau=0,1,\ldots,k-1.
		\]

		From $\gamma_k=\gamma=\frac{\epsilon_2\sqrt{pT}}{\epsilon_3\sqrt{n}}$ and $T\ge  \frac{n\epsilon_3^2(\tilde{\kappa}_0(\epsilon_1,\epsilon_2))^2}{p\epsilon_2^2}$, we have $\gamma\ge\tilde{\kappa}_0(\epsilon_1,\epsilon_2)$. 
		Thus, all the conditions in Lemmas~\ref{Lemma:Lyap:fix} and \ref{zerosg:lemma:sg2-T:23} are satisfied. 
		So \eqref{zerosg:sgproof-vkLya2T} and \eqref{zerosg:v4kspeed} hold.

		Denote $\mathcal{F}_k$ = $\bigcup_{t=0}^k \mathcal{A}_t$.
		By taking expectation with respect to $\mathcal{F}_T$ and summing \eqref{zerosg:sgproof-vkLya2T} over $ \tau\in[0,k-1]$, we obtain
		\begin{align}
			&\sum_{\tau=0}^{k-1}\mathbb{E}[\|\bsx_\tau\|^2_{\bsE}]\hspace{-0.1em}
			\le\hspace{-0.1em}\frac{1}{a_1}\hspace{-0.1em}\Big({\mathcal{L}}_{1,0} \hspace{-0.1em}+ \hspace{-0.1em}
				pn(\tilde{a}_8\bar{L}\hspace{-0.1em}+\hspace{-0.1em}\tilde{a}_4\hspace{-0.1em}+\hspace{-0.1em}\tilde{a}_5\kappa_{\mu}^2)\alpha^2k \Big)\hspace{-0.1em}.
			\label{x_bound}
		\end{align}

		Similarly, for any $k\in[0,T]$, combining \eqref{zerosg:v4kspeed} with \eqref{x_bound} gives
		\begin{align}
			&\mathbb{E}[e_{4,k}]
			\le e_{4,0} + \frac{2\ell^2}{a_1}\alpha\Big({\mathcal{L}}_{1,0} + 
				pn(\tilde{a}_8\bar{L}+\tilde{a}_4+\tilde{a}_5\kappa_{\mu}^2)\alpha^2k \Big) \nonumber\\
				&\quad + pa_9\alpha^2\bar{L}k
				+ p(a_7 + \ell^2\kappa_{\mu}^2)\alpha^2k \nonumber\\
			&= e_{4,0} + \frac{2\ell^2{\mathcal{L}}_{1,0}}{a_1}\alpha
				+ (a_9p\alpha^2k + \frac{2\ell^2\tilde{a}_8}{a_1}pn\alpha^3k)\bar{L} \nonumber\\
				&\quad + (a_7+\ell^2\kappa_{\mu}^2)p\alpha^2k 
				+ \frac{2\ell^2(\tilde{a}_4+\tilde{a}_5\kappa_{\mu}^2)}{a_1}pn\alpha^3k \nonumber\\
			&= e_{4,0} + \frac{2\ell^2{\mathcal{L}}_{1,0}}{a_1}\alpha
				+ (a_9p\alpha^2k + \tilde{a}_{10}pn\alpha^3k)\bar{L} \nonumber\\
				&\quad + a_{11}p\alpha^2k + \tilde{a}_{12}pn\alpha^3k \label{nL:iterat}\\
			&\le n(a_9\epsilon_3^2 + \tilde{a}_{10}\epsilon_3^3\frac{n\sqrt{n}}{\sqrt{pT}})\bar{L} 
				+ n(a_{11}\epsilon_3^2 + \tilde{a}_{12}\epsilon_3^3\frac{n\sqrt{n}}{\sqrt{pT}}) \nonumber\\
				&\quad	+ e_{4,0} + \frac{2\ell^2\epsilon_3}{a_1}\frac{{\mathcal{L}}_{1,0}\sqrt{n}}{\sqrt{pT}}, \nonumber\\
			&\le n(a_9\epsilon_3^2 + \epsilon_3^3)\bar{L} 
				+ n(a_{11}\epsilon_3^2 + \epsilon_3^3) \nonumber\\
			&\quad	+ e_{4,0} + 2\ell^2\epsilon_3\frac{{\mathcal{L}}_{1,0}}{n}, \label{T:use}
		\end{align}
		where the second inequality holds due to $k\le T$ and $\alpha=\epsilon_3\frac{\sqrt{n}}{\sqrt{pT}}$;
		and the last inequality holds due to $T\ge\frac{n^3\tilde{\kappa}_T(\epsilon_1,\epsilon_2)}{p}$.

		As long as $a_9\epsilon_3^2 + \epsilon_3^3 <1$ and
		\begin{align}
			\bar{L}\ge \frac{a_{11}\epsilon_3^2 + \epsilon_3^3
				+ \frac{e_{4,0}}{n} + 2\ell^2\epsilon_3\frac{{\mathcal{L}}_{1,0}}{n^2}}{1-a_9\epsilon_3^2-\epsilon_3^3},
				\nonumber
		\end{align}
		one has $\mathbb{E}[e_{4,k}]\le n\bar{L}$.
		Thus \eqref{L3_bound_ineq} follows from $\epsilon_3\le\bar{\kappa}_3$ and 
		$a_9\epsilon_3^2 + \epsilon_3^3 \le 3/4 < 1$.
		Moreover, 
		an upper bound of $\bar{L}$ can be derived as
		\begin{align}
			&\bar{L} \le \max\{ \frac{e_{4,0}}{n},~
			\max_{\epsilon_3}\frac{a_{11}\epsilon_3^2 + \epsilon_3^3
				+ \frac{e_{4,0}}{n} + 2\ell^2\epsilon_3\frac{{\mathcal{L}}_{1,0}}{n^2}}{1-a_9\epsilon_3^2-\epsilon_3^3} \} \nonumber\\
			&\le \max\Big\{ \frac{e_{4,0}}{n},\nonumber\\
				&\qquad\qquad~4\max_{\epsilon_3}\Big( a_{11}\epsilon_3^2 + \epsilon_3^3
				+ \frac{e_{4,0}}{n} +  2\ell^2\epsilon_3\frac{{\mathcal{L}}_{1,0}}{n^2} \Big) \Big\}\nonumber\\
			&\le \max\Big\{ \frac{e_{4,0}}{n},~4a_{11}\bar{\kappa}_3^2 + 4\bar{\kappa}_3^3 
				+ 4\frac{e_{4,0}}{n} + 8\ell^2\bar{\kappa}_3\frac{{\mathcal{L}}_{1,0}}{n^2} \Big\} \nonumber\\
			&= 4a_{11}\bar{\kappa}_3^2 + 4\bar{\kappa}_3^3 
				+ 4\frac{e_{4,0}}{n} + 8\ell^2\bar{\kappa}_3\frac{{\mathcal{L}}_{1,0}}{n^2}. \label{nL:00}
			\end{align}
				\looseness=-1
				Since $a_{11}$ and $\bar{\kappa}_3$ do not depend on 
				the communication network, dimension $p$, and compression
				parameters $\delta$ and $r$,
				together with $e_{4,0}\hspace{-0.1em}=\hspace{-0.15em}\mathcal{O}(n)$ and $\mathcal{L}_{1,0}\hspace{-0.15em}=\hspace{-0.15em}\mathcal{O}(n)$,  \eqref{nL:00} yields \eqref{Lbar:bound}.
\end{proof}  

\begin{remark}
It should be highlighted that we take a special stepsize $\alpha_k$ and require the total number of iterations $T$ to be greater than a certain number in \eqref{zerosg:step:eta2-sm},
 where the constant $\epsilon_3\le \bar{\kappa}_3$ in $\alpha_k $ is additionally introduced compared with~\cite{Yi_Zerothorder_2022,Wang_CZSD_2025}.
These designs are used in \eqref{T:use} and \eqref{nL:00} to regulate the growth of $e_{4,k}$ and ensure that $\bar{L} = \mathcal{O}(1)$, independent of 
the communication network, dimension $p$,
and communication compression parameters $\delta$ and $r$.
Consequently, the magnitude of $\mathbb{E}[e_{4,k}] = \mathcal{O}(n)$ remains at the same order as the initial optimality gap $e_{4,0}$, 
indicating that it is well-regulated  within a favorable bound. 
Then~\eqref{zerosg:sgproof-vkLya2T} yields
$\mathbb{E}_{\mathcal{A}_k}[\mathcal{L}_{1,k+1}]
\le \mathcal{L}_{1,k} 
- \tfrac{1}{4}\alpha\|\bar{\bsg}^0_{k}\|^2  
+ \alpha^2 \mathcal{O}(pn),$
and summing this inequality establishes the convergence rate
of the global cost function gradient  $\|\bar{\bsg}^0_{k}\hspace{-0.05em}\|^{\hspace{-0.05em}2}\hspace{-0.2em}.$
\end{remark}

Finally, the boundedness of $e_{4,k}$ ensures that the variance of the two-point zeroth-order gradient estimator remains bounded via~\eqref{zerosg-a5:rand-grad-esti2},
since the exploration parameter $\mu_k$ can be chosen arbitrarily small, and the consensus error $\|\bsx_k\|_{\bsE}^2$ vanishes as the algorithm approaches a stationary point.
This result is essential for proving convergence of the Lyapunov function, and thus for the convergence guarantees of the algorithm presented in the next section.

  \section{Main Results} \label{section:main results}
Building on Section~\ref{Section:Proof}, we now present the convergence guarantees of Algorithm~\ref{nonconvex:algorithm-pdgd}.
It should be highlighted that these guarantees hold with heterogeneous data distributions,
 using a two-point sampling strategy rather than $\mathcal{O}(pn)$ sampling per iteration, and under a communication compression framework.

\subsection{General Nonconvex Setting}
We begin with the general nonconvex setting.
To the best of our knowledge,
Theorem~\ref{Thm:nonconvex} first addresses the case of general nonconvex functions in heterogeneous settings.

	\begin{theorem}\label{Thm:nonconvex}
	Under the same conditions as in Lemma~\ref{L3_bound},
	 consider  the sequence $\{\bsx_k\}$ generated by Algorithm~\ref{nonconvex:algorithm-pdgd}. Then
	\begin{subequations}
	\begin{align}
	&~~\frac{1}{T}\sum\nolimits_{k=0}^{T-1}\mathbb{E}[\|\nabla f(\bar{x}_k)\|^2]
	=\mathcal{O}\Big(\frac{\sqrt{p}}{\sqrt{nT}}\Big)+\mathcal{O}\Big(\frac{n}{T}\Big),\label{zerosg:coro-sg-sm-equ3}\\
	&~~\mathbb{E}[f(\bar{x}_{T})]-f^*=\mathcal{O}(1),\label{zerosg:coro-sg-sm-equ4}\\
	&~~\mathbb{E}\Big[\frac{1}{n}\sum\nolimits_{i=1}^{n}\|x_{i,T}-\bar{x}_T\|^2\Big]
	=\mathcal{O}\Big(\frac{n}{T}\Big).\label{zerosg:coro-sg-sm-equ3.1}
	\end{align}
	\end{subequations}
	\end{theorem}
	\begin{proof}

	
As discussed in Section~\ref{Section:Proof} and Remark~3, 
the key technical step is to prove that the variance of the two-point zeroth-order gradient estimator remains bounded under data heterogeneity in the general nonconvex setting (Lemma~\ref{L3_bound} and \eqref{zerosg-a5:rand-grad-esti2}). 
With this result in place, we now establish convergence.



	
	Under the setting of Theorem~\ref{Thm:nonconvex}, it is clear that all the conditions in Lemmas~\ref{zerosg:lemma:sg2-T:23} and \ref{L3_bound} are satisfied.
	Noting that $e_{4,T}=n(f(\bar{x}_T)-f^*)$, \eqref{L3_bound_ineq} and \eqref{Lbar:bound} directly yield \eqref{zerosg:coro-sg-sm-equ4}.

	From \eqref{zerosg:rand-grad-smooth} and \eqref{L3_bound_ineq},
	\begin{align}\label{nonconvex:gg3:le}
	\mathbb{E}[\|\bar{\bsg}^0_k\|^2]=n\mathbb{E}[\|\nabla f(\bar{x}_k)\|^2]\le2\ell \mathbb{E}[e_{4,k}]\le2n\ell\bar{L}.
	\end{align}

	Denote	$\hat{\mathcal{L}}_{2,k}=\|\bsx_k\|^2_{\bsE}+\Big\|\bsv_k	+\frac{1}{\gamma}\bsg_k^0\Big\|^2_{\bsF}+\|\bsx_{k}-\bsy_{k}\|^2$.
	Then
	\begin{align}
		&\mathcal{L}_{2,k}
		\ge\frac{1}{2}\|\bsx_{k}\|^2_{\bsE}
			+\frac{1}{2}\Big(1+\frac{\beta_k}{\gamma_k}\Big)\Big\|\bsv_k+\frac{1}{\gamma_k}\bsg_k^0\Big\|^2_{\bsF}\nonumber\\
			&\quad-\frac{\gamma_k}{2\beta_k\rho_2(L)}\|\bsx_{k}\|^2_{\bsE}
			-\frac{\beta_k}{2\gamma_k}\Big\|\bsv_k+\frac{1}{\gamma_k}\bsg_k^0\Big\|^2_{\bsF} 
			+\|\bsx_{k}-\bsy_{k}\|^2\nonumber\\
		&\ge\varepsilon_7\Big(\|\bsx_{k}\|^2_{\bsE}
			+\Big\|\bsv_k+\frac{1}{\gamma_k}\bsg_k^0\Big\|^2_{\bsF}\Big)
			+\|\bsx_{k}-\bsy_{k}\|^2 \nonumber\\
		&\ge\varepsilon_7\hat{\mathcal{L}}_{2,k}\ge0,\label{zerosg:vkLya3}
	\end{align}
	where the first inequality holds due to the definition of $\mathcal{L}_{2,k}$ and the Cauchy--Schwarz inequality; and the last inequality holds due to $0<\varepsilon_7<\frac{1}{2}$. Similarly, one has
	\begin{align}\label{zerosg:vkLya3.1}
		\mathcal{L}_{2,k}\le\varepsilon_6\hat{\mathcal{L}}_{2,k}.
	\end{align}

	Taking the expectation of \eqref{zerosg:sgproof-vkLya2T-bounded} with respect to $\mathcal{F}_T$,
	and using \eqref{L3_bound_ineq}, \eqref{nonconvex:gg3:le}--\eqref{zerosg:vkLya3.1} and \eqref{zerosg:step:eta2-sm}, we obtain, 
	\begin{align}\label{zerosg:vkLya4-bound-tilde}
	&\mathbb{E}[\mathcal{L}_{2,k+1}]\le(1-\tilde{d}_1)\mathbb{E}[\mathcal{L}_{2,k}] ~+ \nonumber\\
	&\frac{n^2\bar{\kappa}_3^2(\frac{2a_6\ell\bar{L}}{p} + \tilde{a}_8\bar{L}+\tilde{ a}_4+\tilde{ a}_5^{\prime}\kappa_{\mu}^2)}{T},~\forall k\in[0,T-1].
	\end{align}

	From $\epsilon_1>\kappa_1 \ge\frac{13}{2\rho_2(L)}$, we have $\varepsilon_6>1$. From $\epsilon_2\in(0,\kappa_2(\epsilon_1))$, we have $ a_2 = \frac{1}{16}\epsilon_2-\big(\frac{5}{4}+2(1+c_1^{-1})\big)\rho(L)\epsilon_2^2<\frac{1}{16}\epsilon_2<\frac{1}{2}$.
	Moreover, $\tilde{d}_1>0$ and $a_2>0$ have been established in the proof of Lemma~\ref{Lemma:Lyap:fix}.
	Thus,
	\begin{align}\label{zerosg:vkLya2-a1-bounded-thm2}
	0<\tilde{d}_1\le\frac{2 a_2}{\varepsilon_6}<1.
	\end{align}
	Since $\mathcal{L}_{2,0}=\mathcal{O}(n)$, the combination of \eqref{zerosg:vkLya4-bound-tilde}--\eqref{zerosg:vkLya2-a1-bounded-thm2}, \eqref{zerosg:vkLya3},
	and Lemma~5 in \cite{Yi_Zerothorder_2022,Yi_Zerothorder_2021} leads to 
	\begin{align} \label{consensus:bound}
		&\mathbb{E}\Big[\sum\nolimits_{i=1}^{n}\|x_{i,k}-\bar{x}_k\|^2\Big] = \mathbb{E}[\|\bsx_k\|_{\bsE}^2]\le\frac{1}{\varepsilon_7}\mathbb{E}[\mathcal{L}_{2,k}]\nonumber\\
		&=\mathcal{O}(\frac{n^2}{T}) + \mathcal{O}\big(n(1-\tilde{d}_1)^k\big),
	\end{align}
	which further gives \eqref{zerosg:coro-sg-sm-equ3.1}.

	Similarly, the expected sum of \eqref{zerosg:v4kspeed} over $k \in [0, T-1]$ with respect to $\mathcal{F}_T$, together with \eqref{consensus:bound} and \eqref{zerosg:step:eta2-sm}, yields
	\begin{align}
		&\frac{1}{T}\sum\nolimits_{k=0}^{T}\mathbb{E}[\|\nabla f(\bar{x}_k)\|^2]=\frac{1}{nT}\sum\nolimits_{k=0}^{T}\mathbb{E}[\|\bar{\bsg}_{k}^0\|^2]\nonumber\\
		&\le 4\Big(\frac{e_{4,0}}{nT\alpha}
		+\frac{2 \ell ^2}{nT}\mathcal{O}(n^2) + \frac{p\alpha( a_9\bar{L} + a_7 + \ell^2k_\mu^2)}{n}\Big) \nonumber\\
		&\le 4\Big(\frac{f(\bar{x}_0)-f^*}{\underline{\kappa}_3}+ (a_9\bar{L} + a_7+ \ell ^2\kappa_{\mu}^2)\bar{\kappa}_3\Big)\frac{\sqrt{p}}{\sqrt{nT}}
		+\mathcal{O}\Big(\frac{n}{T}\Big). \label{zerosg:thm-sg-sm-equ4p}
	\end{align}
	Since $a_7$, $a_9$, $\underline{\kappa}_3$, and $\bar{\kappa}_3$ do not depend on 
		the communication network, dimension $p$,
and communication compression parameters $\delta$ and $r$, 
	 \eqref{zerosg:thm-sg-sm-equ4p} implies \eqref{zerosg:coro-sg-sm-equ3}.
Moreover, with
$\underline{\kappa}_3,\bar{\kappa}_3
=\mathcal{O}((1+\ell\sqrt{1+\eta_1^2})^{-1})$, $\kappa_\mu=\mathcal{O}(1)$,
 and $\mathcal{L}_{1,0}=\mathcal{O}(n)$, 
 the coefficient of $\sqrt{p/(nT)}$ in
\eqref{zerosg:thm-sg-sm-equ4p} is
\setlength{\belowdisplayskip}{0pt}
\begin{align} \label{nonconvex:constant}
\mathcal{O}\Big(
&\frac{\ell\sigma_1^2}{1+\ell\sqrt{1+\eta_1^2}}
+\frac{\ell(1+\eta_1^2)\check\sigma_2^2}
{1+\ell\sqrt{1+\eta_1^2}} \nonumber\\
&+(1+\ell{\textstyle\sqrt{1+\eta_1^2}})
\big(f(\bar{x}_0)-f^*+1\big)
+\frac{\ell^2}{n}
\Big).
\end{align}
\end{proof}

\begin{remark}
	 \looseness=-1
		Note that the leading-term coefficient omitted from
		\eqref{zerosg:coro-sg-sm-equ3}, given by
		\eqref{nonconvex:constant}, is independent of 
		the communication network, dimension $p$,
		and communication compression parameters $\delta$ and $r$.
		As a result, Algorithm~\ref{nonconvex:algorithm-pdgd} attains linear speedup with a convergence rate of $\mathcal{O}(\sqrt{p}/\sqrt{nT})$, leading to faster convergence as the number of agents increases.
		To the best of our knowledge, this is the first distributed zeroth-order algorithm that establishes convergence for general nonconvex (and convex) functions without assuming data homogeneity or $\mathcal{O}(pn)$ function evaluations per iteration, while matching the fastest proven rate for finding stationary points under data homogeneity and exact communication assumptions.
		Specifically, the HEDZOC algorithm has a faster proven rate than~
		~\cite{Tang_Zeroth_2020,yuan_randomized_2015}
		 and is comparable to~\cite{Yi_Zerothorder_2022,Ling_Federated_2024}.
		However,~\cite{yuan_randomized_2015,Tang_Zeroth_2020} 
		required the Lipschitz condition, 
		and~\cite{Yi_Zerothorder_2022,Ling_Federated_2024} 
		assumed data homogeneity, 
		neither of which is needed in this paper.
		In addition, the proposed algorithm explicitly accounts for communication compression.
		Compared with distributed zeroth-order methods that rely on $\mathcal{O}(p)$ samplings per agent per iteration,
		our algorithm exhibits a slower convergence rate, which is reasonable given its two-point sampling design.
Although~\cite{Ling_Federated_2024} further alleviated the dimension dependence of the convergence rate under an additional assumption of a low-intrinsic-dimensional Hessian, this assumption is beyond the scope of this paper.		
	\end{remark}

	\begin{remark}
The communication compression slightly slow down convergence,
but its parameters $\delta$ and $r$ implicitly affect only the coefficient
of the lower-order $\mathcal{O}(n/T)$ term in
\eqref{zerosg:coro-sg-sm-equ3}.
The simulation results in Section~\ref{Section:Simulations}
further support the limited impact of communication compression in practice.
		Compared with existing distributed compressed zeroth-order algorithms, 
		our method achieves a faster proven convergence rate than~\cite{singh2024decentralized} 
		and matches the rate reported in~\cite{Wang_CZSD_2025}. 
		However,~\cite{singh2024decentralized}
		required the Lipschitz condition, and~\cite{Wang_CZSD_2025}
		depended on the weak data homogeneity assumption, whereas our results are established without requiring any such assumptions.
	\end{remark}

\subsection{P--L Setting with Unknown Constant}

Next, we 
consider
 a practical scenario where the global cost function satisfies the P--L condition but the corresponding constant is unknown. As a result, the algorithm can be implemented without verifying the P--L condition or knowing its constant, which is often difficult in practice \cite{Yi_Linear_2021}.

 Following the general nonconvex setting, we first prove boundedness of $e_{4,k}$ in Lemma~\ref{L3_bound:fix}, which implies bounded variance
of the two-point zeroth-order gradient estimator, and then establish convergence in Theorem~\ref{Thm:PL:unknow}.
 \begin{lemma} \label{L3_bound:fix}
		Under Assumptions~\ref{nonconvex:ass:fil}--\ref{zerosg:ass:zeroth-variance}, consider the sequence $\{\bsx_k\}$ generated by Algorithm~\ref{nonconvex:algorithm-pdgd} with	
	\begin{align}\label{zerosg:step:eta1}
		&\alpha_k=\frac{1}{(T+1)^{\theta}},~\beta_k=\epsilon_1\gamma_k,~\gamma_k=\frac{\epsilon_2}{\alpha_k},~\theta\in(0.5,1),
			 \nonumber\\
		&\quad~~\mu _{i,k}\le\frac{\kappa_{\mu}\sqrt{p\alpha _k}}{\sqrt{n+p}},~\omega\le\frac{1}{r},~T\ge\kappa_T(\epsilon_1, \epsilon_2, \theta),
			 \nonumber\\
		&~\quad\quad\quad~\epsilon_1>\kappa_1, ~\epsilon_2\in(0,\kappa_2(\epsilon_1)),~\kappa_{\mu}>0,
	\end{align}
	where $\kappa_T(\epsilon_1, \epsilon_2, \theta)$ is a positive constant.
 Then there exists a constant $\bar{L}>0$ such that
	\begin{align}
		&\qquad \quad~ \mathbb{E}[e_{4,k}]\le n\bar{L},~\forall k\in[0,T], \label{L3_bound_ineq:fix}\\
		&\bar{L} \le \frac{2e_{4,0}}{n} + 4\ell^2\frac{{\mathcal{L}}_{1,0}}{n}
				+ \frac{a_{11}}{2a_9} + 2 = \mathcal{O}(1). \label{Lbar:bound:fix}
	\end{align}
\end{lemma}

\begin{proof}
	
	We prove \eqref{L3_bound_ineq:fix} by  mathematical induction. 

		For $k=0$, the result follows directly from the definition.

		For $k>0$, suppose that the statement holds for $\tau = 0,1,\ldots,k-1$, 
		namely,
		\[
			\mathbb{E}[e_{4,\tau}] \le n\bar{L}, \quad \forall\, \tau=0,1,\ldots,k-1.
		\]

    	Given that $\gamma_k = \gamma = \epsilon_2(T+1)^{\theta}$ and $T \ge \kappa_T(\epsilon_1, \epsilon_2, \theta)\ge(\frac{\tilde{\kappa}_0(\epsilon_1, \epsilon_2)}{\epsilon_2})^{\frac{1}{\theta}}$, one has $\gamma \ge \tilde{\kappa}_0(\epsilon_1, \epsilon_2)$. 
		Thus, all the conditions in Lemmas~\ref{Lemma:Lyap:fix} and \ref{zerosg:lemma:sg2-T:23} are satisfied, guaranteeing the validity of \eqref{zerosg:sgproof-vkLya2T} and \eqref{zerosg:v4kspeed}. 		
		Furthermore,  similar to the proof of Lemma~\ref{L3_bound}, we obtain \eqref{nL:iterat}.
	
		Then, from \eqref{nL:iterat}, $\alpha_k=\alpha=\frac{1}{(T+1)^{\theta}}$, and $k\le T$, we have
		\begin{align}
			&\mathbb{E}[e_{4,k}]
				\le e_{4,0} + \frac{2\ell^2{\mathcal{L}}_{1,0}}{a_1(T+1)^{\theta}} 
				+ \Big( \frac{a_9pT}{(T+1)^{2\theta}}+ \frac{ \tilde{a}_{10}pnT}{(T+1)^{3\theta}} \Big)\bar{L} \nonumber\\
				&\quad + \frac{a_{11}pT}{(T+1)^{2\theta}}+ \frac{ \tilde{a}_{12}pnT}{(T+1)^{3\theta}} 
				\nonumber\\
			&\quad \le e_{4,0} + 2\ell^2{\mathcal{L}}_{1,0} + \frac{1}{2}n\bar{L} + \frac{a_{11}n}{4a_9} + n, \label{T:use2}
		\end{align}
		where the last inequality holds due to $T\ge\kappa_T(\epsilon_1, \epsilon_2, \theta)\ge\max\{(\frac{1}{a_1})^{\frac{1}{\theta}},~(\frac{4a_9p}{n})^{\frac{1}{2\theta-1}},~(4\tilde{a}_{10}p)^{\frac{1}{3\theta-1}},~(\tilde{a}_{12}p)^{\frac{1}{3\theta-1}}\}$.

		As long as 
		\begin{align}
			\bar{L}\ge \frac{2e_{4,0}}{n} + 4\ell^2\frac{{\mathcal{L}}_{1,0}}{n}
				+ \frac{a_{11}}{2a_9} + 2, \nonumber
		\end{align}
		one has $\mathbb{E}[e_{4,k}]\le n\bar{L}$.
		Moreover, 
		an upper bound of $\bar{L}$ can be derived as
		\begin{align} \label{nL:fix:0}
			&\bar{L} \le \max\{ \frac{e_{4,0}}{n},~
				\frac{2e_{4,0}}{n} + 4\ell^2\frac{{\mathcal{L}}_{1,0}}{n}
				+ \frac{a_{11}}{2a_9} + 2 \} \nonumber\\
			&=	\frac{2e_{4,0}}{n} + 4\ell^2\frac{{\mathcal{L}}_{1,0}}{n}
				+ \frac{a_{11}}{2a_9} + 2.
			\end{align} 
			\looseness=-1
				Since $a_{9}$ and $a_{11}$ 
		do not depend on 
				the communication network, dimension $p$, and compression
				parameters $\delta$ and $r$,
		 together with $e_{4,0}\hspace{-0.15em}=\hspace{-0.15em}\mathcal{O}(n)$ and $\mathcal{L}_{1,0}\hspace{-0.15em}=\hspace{-0.15em}\mathcal{O}(n)$,  \eqref{nL:fix:0} yields \eqref{Lbar:bound:fix}.
	\end{proof}

\begin{remark}
	\looseness=-1
	It should be highlighted that we take a special stepsize $\alpha_k$ and require the total number of iterations $T$ to be greater than a certain number in \eqref{zerosg:step:eta1}.
	These designs are used in \eqref{T:use2} and \eqref{nL:fix:0} to regulate the growth of $e_{4,k}$ and ensure that $\bar{L} = \mathcal{O}(1)$, independent of 
	the communication network, dimension $p$, and communication compression
				parameters $\delta$ and $r$.
\end{remark}

\begin{theorem}\label{Thm:PL:unknow}
	Under the same conditions as in Lemma~\ref{L3_bound:fix}, consider the sequence $\{\bsx_k\}$ generated by Algorithm~\ref{nonconvex:algorithm-pdgd}. 
	Then
	\begin{subequations}
		\begin{align}
			&~~~\mathbb{E}\Big[\frac{1}{n}\sum_{i=1}^{n}\|x_{i,T}-\bar{x}_T\|^2\Big]
			=\mathcal{O}\Big(\frac{p}{T^{2\theta}}\Big), \label{zerosg:thm-sg-diminishing-equ1.1bounded}\\
			&~~~
			\mathbb{E}[f(\bar{x}_{T})-f^*]
			=\mathcal{O}\Big(\frac{p}{nT^\theta}\Big)+\mathcal{O}\Big(\frac{p}{T^{2\theta}}\Big).
			\label{zerosg:thm-sg-diminishing-equ1bounded}
		\end{align}
	\end{subequations}
\end{theorem}
\begin{proof}

It is clear that all the conditions in Lemmas~\ref{zerosg:lemma:sg2-T:23} and \ref{L3_bound:fix} are satisfied.

From Assumptions~\ref{nonconvex:ass:fil}, 
\begin{align}\label{nonconvex:gg3}
\|\bar{\bsg}^0_k\|^2=n\|\nabla f(\bar{x}_k)\|^2\ge2\nu n(f(\bar{x}_k)-f^*)=2\nu e_{4,k}.
\end{align}
Similarly, combining \eqref{zerosg:rand-grad-smooth} and \eqref{L3_bound_ineq:fix} gives
\begin{align}\label{nonconvex:gg3:le:fix}
	\mathbb{E}[\|\bar{\bsg}^0_k\|^2]=n\mathbb{E}[\|\nabla f(\bar{x}_k)\|^2]\le2\ell \mathbb{E}[e_{4,k}]\le2n\ell\bar{L}.
\end{align}


 Taking expectaion of \eqref{zerosg:sgproof-vkLya2T-bounded} with respect to $\mathcal{F}_T$, 
 and combining it with \eqref{nonconvex:gg3:le:fix} and \eqref{zerosg:vkLya3}--\eqref{zerosg:vkLya3.1}, we obtain
 \begin{align}\label{zerosg:vkLya4-bound-tilde:fix}
&\mathbb{E}[\mathcal{L}_{2,k+1}]\le(1-\tilde{d}_1)\mathbb{E}[\mathcal{L}_{2,k}]~+ \nonumber\\
&\frac{pn(2a_6\ell\bar{L}/p + \tilde{a}_8\bar{L}+\tilde{ a}_4+\tilde{ a}_5^{\prime}\kappa_{\mu}^2)}{(T+1)^{2\theta}}, ~\forall k\in[0,T-1].
\end{align}
Since $\mathcal{L}_{2,0} = \mathcal{O}(n)$, the combination of \eqref{zerosg:vkLya4-bound-tilde:fix}, \eqref{zerosg:vkLya2-a1-bounded-thm2}, 
and Lemma~5 in \cite{Yi_Zerothorder_2022,Yi_Zerothorder_2021} yields
\begin{align} \label{consensus:bound:fix}
	\mathbb{E}[\|\bsx_k\|_{\bsE}^2]=\mathcal{O}(\frac{pn}{T^{2\theta}}) + + \mathcal{O}\big(n(1-\tilde{d}_1)^k\big),
\end{align}
which yields \eqref{zerosg:thm-sg-diminishing-equ1.1bounded}.

Similarly, the expectation of \eqref{zerosg:v4kspeed} with respect to $\mathcal{F}_T$, 
along with \eqref{nonconvex:gg3}, \eqref{nonconvex:gg3:le:fix} and \eqref{consensus:bound:fix}, leads to
\begin{align}\label{sguT:thm-sg-equ2bounded-proof}
&\mathbb{E}[e_{4,k+1}] \le \Big(1-\frac{\nu\alpha}{2}\Big)\mathbb{E}[e_{4,k}] \nonumber\\
&\quad + 2\alpha \ell^2\mathbb{E}[\|\bsx_k\|^2_{\bsE}] 
	+ p\alpha^2(a_9\bar{L} + a_7 + \ell^2k_\mu^2) \nonumber\\
&\le\Big(1-\frac{\min\{\nu,1\}}{2(T+1)^\theta}\Big)\mathbb{E}[e_{4,k}] + \mathcal{O}(\frac{pn}{(T+1)^{3\theta}}) \nonumber\\
	&\quad + \mathcal{O}\big( \frac{n(1-\tilde{d}_1)^k}{(T+1)^{\theta}} \big)  + \frac{p(a_9\bar{L} + a_7 + \ell^2k_\mu^2)}{(T+1)^{2\theta}}.
\end{align}
Although the magnitude of $\nu$ is unknown, the coefficient
$1-\min\{\nu,1\}/(2(T+1)^\theta)$
always lies in $(0,1)$.
Unrolling \eqref{sguT:thm-sg-equ2bounded-proof} and applying
Lemma~5 in \cite{Yi_Zerothorder_2022,Yi_Zerothorder_2021}
yield \eqref{zerosg:thm-sg-diminishing-equ1bounded},
where the coefficient of the leading term $p/(nT^\theta)$ is
\setlength{\belowdisplayskip}{0pt}
\begin{align}
\frac{1}{\min\{\nu,1\}}\mathcal{O}\Big(
&\ell\sigma_1^2+\ell(1+\eta_1^2)\check\sigma_2^2 \nonumber\\
&+\ell^2(1+\eta_1^2)
\big(f(\bar{x}_0)-f^*+1+\ell^2\big)\Big). \label{constant:unknow}
\end{align}
\end{proof}

\begin{remark}
	\looseness=-1
	It should be highlighted 
	 that the P--L constant is not used, 
as identifying this constant is often difficult in real applications \cite{Yi_Linear_2021}. 
Importantly, our algorithm does not require verifying the P--L condition or knowing its constant in order to be implemented. 
Moreover, the leading-term coefficient omitted from~\eqref{zerosg:thm-sg-diminishing-equ1bounded},
given by \eqref{constant:unknow}, is independent of
the communication network, dimension $p$,
and communication compression parameters $\delta$ and $r$.
	To the best of our knowledge, this is the first distributed zeroth-order algorithm that establishes convergence under the P--L condition without requiring knowledge of the P--L constant, data homogeneity, or $\mathcal{O}(pn)$ function evaluations per iteration.
	In contrast to~\cite{Yi_Zerothorder_2022}, which relied on the data homogeneity assumption, our analysis removes this assumption and 
	covers a broad class of compressors.
\end{remark}

	
\begin{remark}
	When communication is uncompressed, Algorithm~\ref{nonconvex:algorithm-pdgd} reduces to the distributed zeroth-order method in~\cite{Yi_Zerothorder_2022}. 
	Therefore, Theorems~\ref{Thm:nonconvex} and~\ref{Thm:PL:unknow} rigorously extend the 
	results
	of~\cite{Yi_Zerothorder_2022} from homogeneous to heterogeneous data distributions.
\end{remark}

\subsection{P--L Setting with Known Constant}

Finally, Theorem~\ref{Thm:PL:know} considers the case where the global cost function satisfies the P--L condition with a known P--L constant, which enables a faster convergence rate.

\begin{theorem}\label{Thm:PL:know}
Under Assumptions~\ref{nonconvex:ass:fil}--\ref{zerosg:ass:zeroth-variance}, 
consider the sequence $\{\bsx_k\}$ generated by Algorithm~\ref{nonconvex:algorithm-pdgd} 
with
\begin{align}\label{zerosg:step:eta1t1}
	&~\alpha_k=\frac{\epsilon_2}{\epsilon_4(k+m)}, \beta_k=\epsilon_1\gamma_k,
	\gamma_k=\frac{\epsilon_2}{\alpha_k}, \mu_{i,k}\le \frac{\kappa_{\mu}\sqrt{p\alpha_k}}{\sqrt{n+p}},\nonumber\\
	&\hspace{0.2em}~~~\quad \epsilon_1>\kappa_1, ~\epsilon_2\in(0,\kappa_2(\epsilon_1)), ~\epsilon_4\in\big[\frac{\kappa_4\nu\epsilon_2}{4}, \frac{\nu\epsilon_2}{4}\big),\nonumber\\
	&\hspace{0.2em}m\ge\kappa_m(\epsilon_1,\epsilon_2,\epsilon_4,\nu),~\omega\le1/r,~\kappa_4\in(0,1), ~\kappa_{\mu}>0, 
	\end{align}
where 
$\kappa_m(\epsilon_1,\epsilon_2,\epsilon_4,\nu)$ is a positive constant.
Then
	\begin{subequations}
	\begin{align}
	&\mathbb{E}\Big[\frac{1}{n}\sum\nolimits_{i=1}^{n}\|x_{i,T}-\bar{x}_T\|^2\Big]
	=\mathcal{O}\Big(\frac{p}{T^{2}}\Big), \label{zerosg:thm-sg-diminishing-equ2.1bounded}\\
	&\mathbb{E}[f(\bar{x}_{T})-f^*]
	=\mathcal{O}\Big(\frac{p}{nT}\Big)+\mathcal{O}\Big(\frac{p}{T^{2}}\Big) \nonumber\\
		&\qquad + \mathcal{O}\Big(\frac{p^{\vartheta}}{T^{\vartheta}}\Big)+\mathcal{O}\Big(\frac{p}{T^{3}}\Big),~\forall T\in\mathbb{N}_+,
	\label{zerosg:thm-sg-diminishing-equ2bounded}
	\end{align}
	\end{subequations}
	where $\vartheta=\nu\epsilon_2/(2\epsilon_4)\in(2,2/\kappa_4]$.
	\end{theorem}
\begin{proof}
	The key distinction from the previous proofs is the setting of time-varying algorithm parameters, which makes the analysis more challenging.
	 Nevertheless, with the known P--L constant, this setting allows for a further improvement in the theoretical convergence rate.
	The detailed proof is provided in Appendix~\ref{Proof:thm:know}.
\end{proof}

\begin{remark}
	Note that  the leading-term coefficient omitted from~\eqref{zerosg:thm-sg-diminishing-equ2bounded}
is independent of the communication network, dimension $p$,
and communication compression parameters $\delta$ and $r$.
Thus, Algorithm~\ref{nonconvex:algorithm-pdgd}
achieves linear speedup with a convergence rate of
$\mathcal{O}(p/(nT))$.
Compared to~\cite{Yi_Zerothorder_2022}, our method supports a general class of compressors. 
Furthermore, unlike~\cite{Wang_CZSD_2025}, our analysis does not rely on any data homogeneity assumptions and further improves the convergence exponent~$\vartheta$ in $\mathcal{O}\big(p^{\vartheta}/T^{\vartheta}\big)$ by a factor of~$1/3$ under the same algorithm parameter setting.
\end{remark}

\section{Simulations} \label{Section:Simulations}

In this section, we study black-box adversarial example generation against a deep neural network (DNN) classifier trained on the MNIST dataset. 
Unlike the homogeneous benchmark setting in which all agents cooperate to learn a consensus perturbation that attacks the same target digit class~\cite{Yi_Zerothorder_2022}, 
we aim to optimize a perturbation that can simultaneously fool the DNN classifier on multiple digit classes.
To this end, we formulate the task as the distributed zeroth-order optimization problem
\eqref{zerosg:eqn:xopt}
with the following attack cost functions: 
\begin{align*}
&F_i(x,\xi_i)  =\frac{c}{m_i}\sum_{j=1}^{m_i} \max \Big\{ H_{t_{i,j}} \big(\frac{1}{2}\tanh ( \tanh^{-1} 2 s_{i,j} + x)\big) \\
&\quad - \max_{k \neq t_{i,j}} H_k\big(\frac{1}{2}\tanh ( \tanh^{-1} 2 s_{i,j} + x)\big) ,~0 \Big\}  \\
&\quad + \frac{1}{m_i}\sum_{j=1}^{m_i}\Big\|\frac{1}{2}\tanh ( \tanh^{-1} 2 s_{i,j} + x) - s_{i,j} \Big\|_2^2.
\end{align*}
Here, $x$ is the perturbation.
$H(\cdot)=\col(H_0(\cdot),\dots,H_9(\cdot))$ denotes the well-trained DNN classifier, 
which takes an input image and outputs classification scores for digits $0$--$9$. 
$c=1.5$ balances the trade-off between attack success rate and perturbation magnitude. 
$m_i = 3$ denotes the number of randomly sampled images in each iteration, 
and $\xi_i = (s_{i,j}, t_{i,j})$ represents the stochastic data pair, 
where $s_{i,j}\in\mathbb{R}^p$ is the sampled image and $t_{i,j}\in\{0,1,\ldots,9\}$ is the corresponding label. 
We set $n = 10$ agents, and $p = 28\times 28$ as the image dimension. 
Agent interactions follow a random Erd\H{o}s--R\'enyi graph with connection probability $0.4$.

\begin{table*}[t]
\centering
\small
\caption{Performance under different heterogeneity levels with the 4-bit compressor (25{,}000 iterations).}
\label{Table:HG}
\begin{tabular*}{0.85\textwidth}{@{\extracolsep{\fill}}ccccccccccc@{}}
\toprule
~Heterogeneity level& 1 & 2 & 3 & 4 & 5 & 6 & 7 & 8 & 9 & 10 \\
\midrule
~Attack success rate & 84.0\% & 83.0\% & 83.1\% & 80.2\% & 80.8\% & 80.0\% & 79.1\% & 81.1\% & 79.5\% & 81.4\% \\
~$\ell_2$ distortion & 7.44 & 7.10 & 7.01 & 6.86 & 6.72 & 6.79 & 6.76 & 6.80 & 6.47 & 6.24 \\
\bottomrule
\end{tabular*}
\vspace{0.5em}
\end{table*}

\subsection{Convergence Rate and Communication Efficiency} \label{sec:exp:1}
We first compare the proposed algorithm with a state-of-the-art method to demonstrate its convergence rate and communication efficiency. 
To the best of our knowledge, among existing distributed two-point zeroth-order methods, 
only the algorithm in~\cite{Yi_Zerothorder_2022} addresses data heterogeneity under the P--L condition.
Nevertheless, this method can be viewed as an uncompressed special case of our algorithm, and thus is capable of handling heterogeneous data in the general nonconvex setting. 
Therefore, we compare the proposed algorithm (implemented with four different compressors, as introduced in Section~\ref{Section:Preliminaries}) against the method in~\cite{Yi_Zerothorder_2022}.
We attack five digit classes $\{0,2,4,6,8\}$ by assigning each agent one class and training it on local samples from that class only, thereby inducing data heterogeneity.




\looseness=-1
Fig.~\ref{Fig:iteration} depicts the evolution of the attack loss with respect to iterations. 
It can be observed that the proposed HEDZOC algorithm adapts well to various compressors, achieving convergence rates comparable to that under exact communication.
Since all methods employ the two-point sampling scheme, the attack loss curves with respect to function evaluations and iterations exhibit similar trends.
Fig.~\ref{Fig:Bits} reveals that compression greatly improves communication efficiency. 
Notably, when HEDZOC with the Norm-sign compressor has already converged, the ZO algorithm and those with milder compression (e.g., Rand-\allowbreak200 and Rand-400) remain far from convergence.

\begin{figure}[!t]
\centering
  \includegraphics[width=0.495\textwidth]{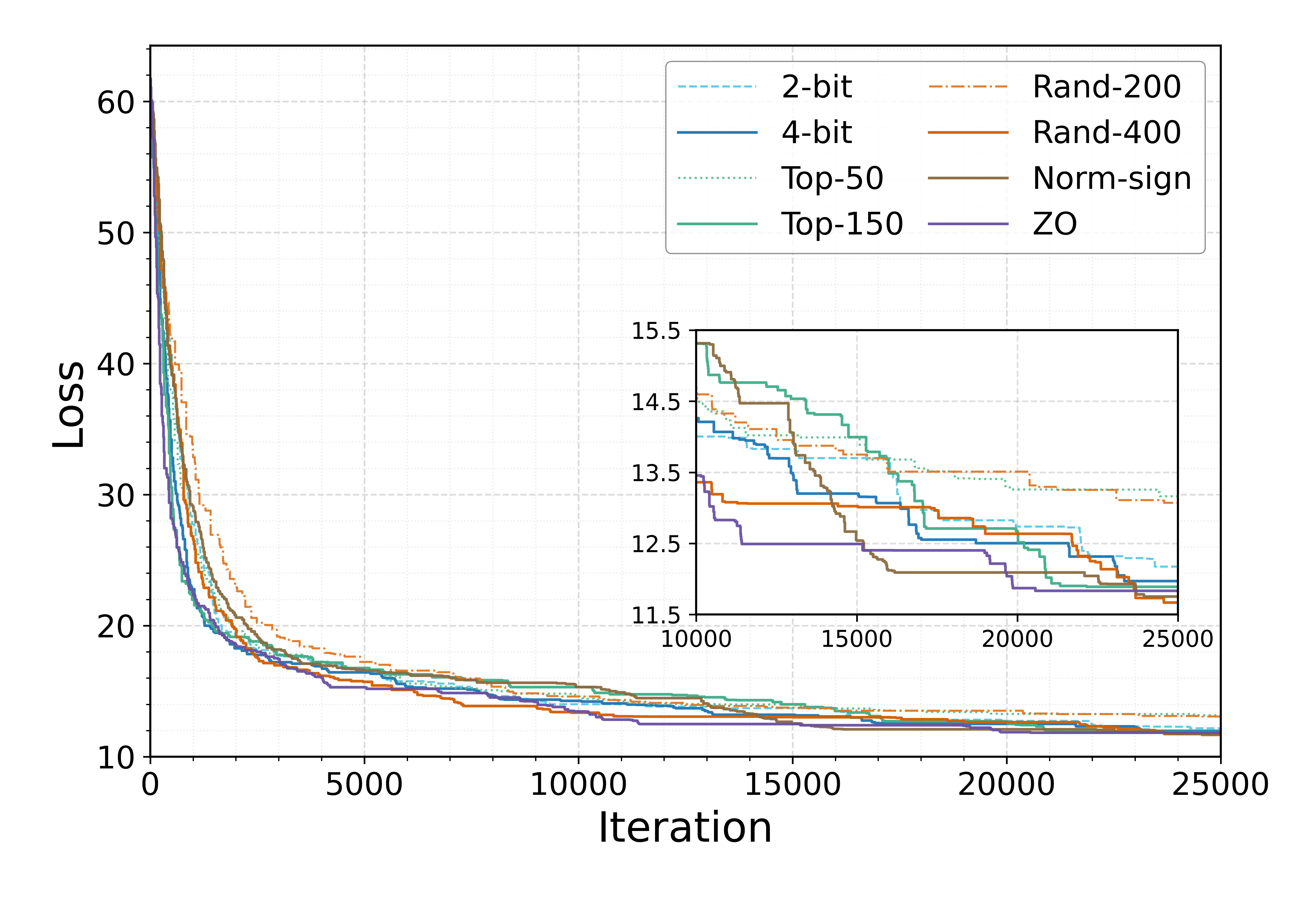}
  \caption{Evolutions of attack loss with respect to the number of iterations.}
  \label{Fig:iteration}
  \par 
  \medskip
  \medskip
  \includegraphics[width=0.48\textwidth]{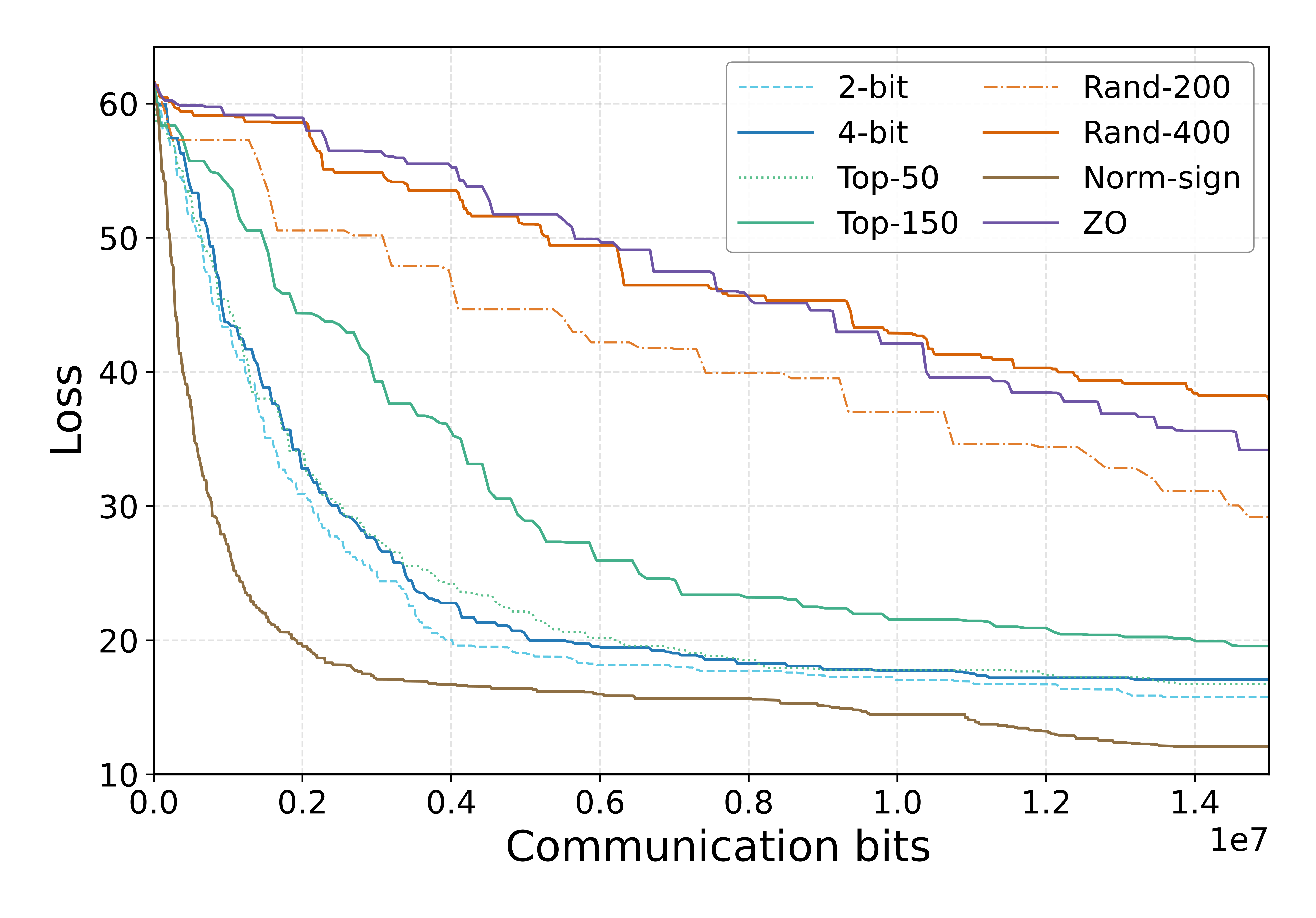}
  \caption{Evolutions of attack loss with respect to the number of inter-agent communication bits.}
  \label{Fig:Bits}
\end{figure}

\subsection{Data Heterogeneity}

To evaluate the performance of the proposed algorithm under different degrees of data heterogeneity, we consider ten heterogeneity levels.
For a fair comparison, we keep the global cost function $f(x)$ fixed by using the same global dataset of 8{,}000  images (800 per digit). 
We allocate the same number of samples to each of the $10$ agents and vary the local class proportions. 
Specifically, we use $\rho_h\in\{0.1,0.2,\ldots,1.0\}$ to control the heterogeneity: agent $i$ contains a $\rho_h$ fraction of digit $(i-1)$ and distributes the remaining $(1-\rho_h)$ fraction uniformly over the other digits.
Thus, $\rho_h=0.1$ yields identical class proportions across agents (least heterogeneous), whereas at $\rho_h=1.0$ they are totally distinct (most heterogeneous).

Table~\ref{Table:HG} summarizes the results under ten levels of data heterogeneity, achieved by the HEDZOC algorithm equipped with the 4-bit compressor after 25{,}000 iterations.
 Overall, the proposed HEDZOC algorithm remains effective across all heterogeneity levels and achieves high attack success rates.
As heterogeneity increases, the attack success rate decreases moderately, consistent with the slight reduction in $\ell_2$ distortion that indicates less aggressive perturbations.

\subsection{Linear Speedup}

To examine scalability to larger networks and linear speedup
in convergence with respect to the number of agents,
we increase $n$ from $10$ to $100$ and repeat the experiments
in Section~\ref{sec:exp:1},
scaling the stepsize $\alpha$ by $\sqrt{10}$ according to the parameter setting of
Theorem~\ref{Thm:nonconvex}, and using the 4-bit compressor.
Fig.~\ref{Fig:linear} depicts the evolution of the attack loss with respect to iterations for $n=10$ and $n=100$,
with slopes of $-0.844$ and $-3.264$, respectively, yielding a magnitude ratio of $3.870$ that exceeds the theoretical factor $\sqrt{10}\approx 3.162$ and supports the predicted linear speedup.

\begin{figure}[!h]
\centering
  \includegraphics[width=0.495\textwidth]{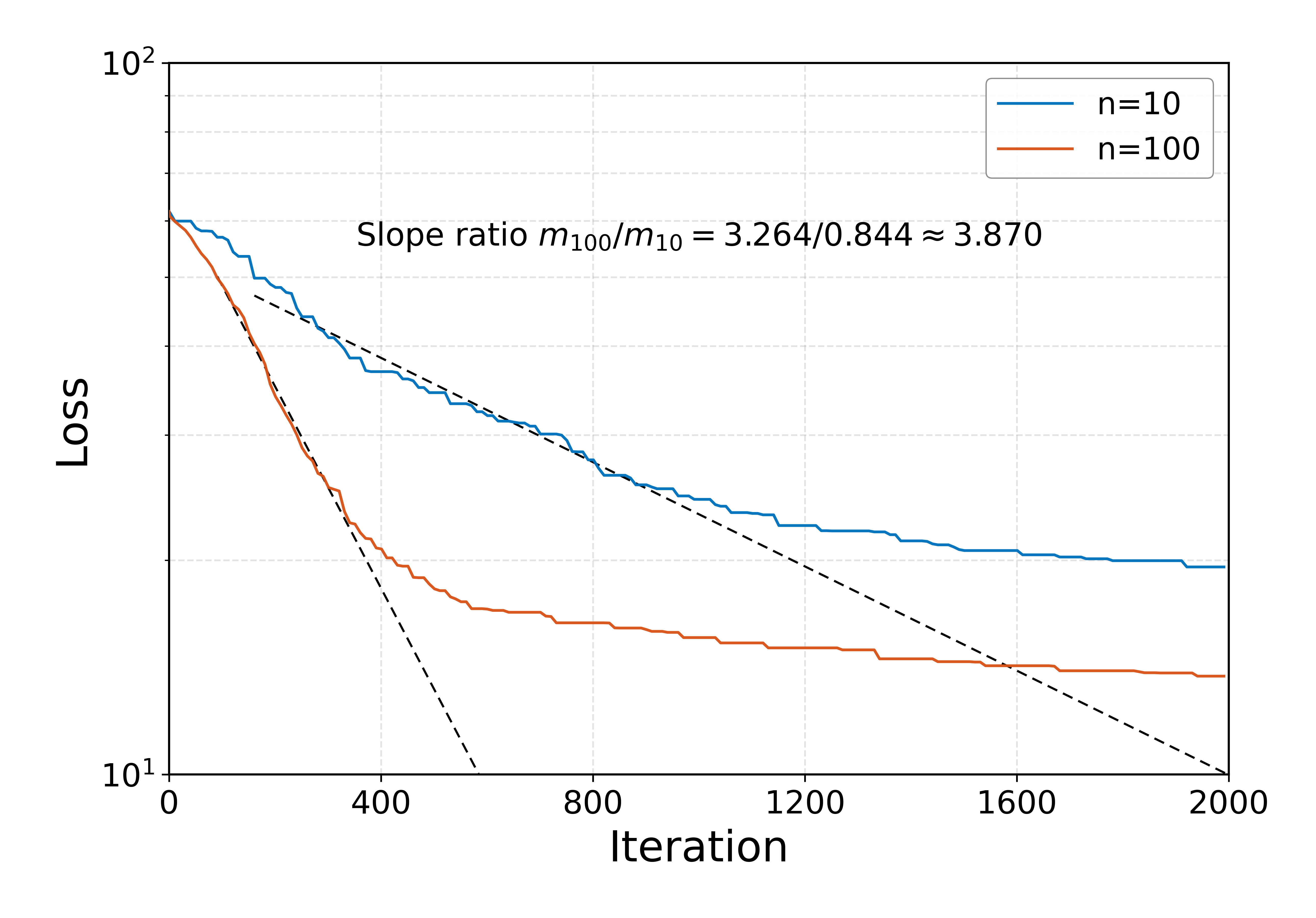}
\caption{Evolutions of attack loss with respect to the number
of iterations for $n=10$ and $n=100$ with the 4-bit compressor.}
  \label{Fig:linear}
\end{figure}

\section{Conclusions} \label{Section:Conclusion}

This paper proposed an algorithm for heterogeneous distributed zeroth-order optimization and provided its convergence analysis, 
 with incorporating communication compression to reduce the communication burden.
Existing convergence analyses for distributed zeroth-order optimization typically depend on relatively strong assumptions such as data homogeneity, $\mathcal{O}(pn)$ function evaluations per iteration, or the P--L condition.
To the best of our knowledge, this is the first work to show that these assumptions can be relaxed or even entirely removed 
while maintaining the fastest theoretical convergence rates achieved by existing methods under data homogeneity and exact communication assumptions.
Future work includes reducing the dependence on data homogeneity and Lipschitz conditions in both federated and centralized learning, accelerating convergence, and exploring general compression techniques.

\pagebreak[4]

\appendices

\section*{Appendix}

\setcounter{subsection}{0}
\renewcommand{\thesubsection}{\Alph{subsection}} 
\numberwithin{lemma}{subsection}
\setcounter{lemma}{0}

\subsection{Constants Used Throughout the Paper} \label{appendix:constant:Thm}
	\begin{align*}		
		&\kappa_1=\max\Big\{\frac{13}{2\rho_2(L)},~\rho_2(L)\Big\},\\
		&\kappa_2(\epsilon_1)=\min\Big\{\frac{\varepsilon_3}{\varepsilon_4},~\frac{\rho^{-1}(L)}{20+32(1+c_1^{-1})}, \\
			&\quad\frac{\sqrt{\varepsilon_{13}^2+4\varepsilon_{14}c_2}-\varepsilon_{13}}{2\varepsilon_{14}},~8\Big\},\\
		&\tilde{\kappa}_0(\epsilon_1,\epsilon_2)=\max\Big\{\varepsilon_{0},
			~\Big(\frac{p(1+\eta_1^2)\tilde{\varepsilon}_9}{ a_1}\Big)^{\frac{1}{3}},~p\epsilon_2\tilde{\varepsilon}_{12}\Big\},\\
		&\underline{\kappa}_3>0~\text{denotes the root of the equation}~a_9\epsilon^2+ \epsilon^3=\tfrac14,\\ 
		&\bar{\kappa}_3>\underline{\kappa}_3~\text{denotes the root of the equation}~a_9\epsilon^2+ \epsilon^3=\tfrac34, \\
		&\tilde{\kappa}_T(\epsilon_1,\epsilon_2)=\max\{1/a_1^2,~\tilde{a}_{10}^2,~\tilde{a}_{12}^2\},\\
		&\kappa_T(\epsilon_1, \epsilon_2, \theta)=\max\Big\{(\frac{\tilde{\kappa}_0(\epsilon_1,\epsilon_2)}{\epsilon_2})^{\frac{1}{\theta}},~(\frac{1}{a_1})^{\frac{1}{\theta}},~(\frac{4a_9p}{n})^{\frac{1}{2\theta-1}},\\
				&\quad(4\tilde{a}_{10}p)^{\frac{1}{3\theta-1}},~(\tilde{a}_{12}p)^{\frac{1}{3\theta-1}}\Big\},\\		
		&\kappa_m(\epsilon_1,\epsilon_2,\epsilon_4,\nu)
			=\max\Big\{\frac{\kappa_0(\epsilon_1,\epsilon_2)}{\epsilon_4},
			~\frac{\epsilon_2}{a_1\epsilon_4},
			~\frac{4a_9\epsilon_2^2}{\epsilon_4^2}\frac{p}{n}+1,\\
			&\quad \sqrt{\frac{2a_{10}\epsilon_2^3p}{\epsilon_4^3}}+1,~\sqrt{\frac{a_{12}\epsilon_2^3p}{2\epsilon_4^3}}+1,~\frac{\epsilon_2}{2\nu\epsilon_4} + \hat{\kappa}_m\Big\},~\hat{\kappa}_m>0,\\
		&\kappa_0(\epsilon_1,\epsilon_2)=\max\Big\{\varepsilon_{0},\Big(\frac{2p(1+\eta_1^2)\varepsilon_9}{ a_1}\Big)^{\frac{1}{2}},~\frac{2\varepsilon_5}{ a_1},
			~\frac{\varepsilon_{10}}{2 a_2}, \\
		&\quad\frac{\delta_0b_{6,k}}{ a_3}, ~p\epsilon_2\varepsilon_{12}\Big\},\\
		&\varepsilon_{0}=\max\{ 1+\frac{5}{2} \ell ^2, 
		~\big(8+16(1+c_1^{-1})+4p(1+\eta_1^2)(1+\\
			&\quad \eta_2^2)(6+ 16(1+c_1^{-1})+ \ell) + 2 a_6\big)^\frac{1}{2} \ell ,~8\rho_2^{-1}(L),\\
			&\quad 4p(1+\eta_1^2)\epsilon_2 \ell\}, \\
		&\varepsilon_1=(1+\epsilon_1)\rho_2^{-1}(L), \\
		&\varepsilon_2=\frac{1}{2}\rho(L)\epsilon_1\epsilon_2+(2+\rho(L))\epsilon_2\\
			& \quad + \big(3\rho^2(L)+4(1+c_1^{-1})\rho^2(L)+2\big)\epsilon_1^2\epsilon_2^2\\
			& \quad +\big(2-\rho_2(L)\big)\epsilon_1\epsilon_2^2+\big(3+\rho(L)
			\big)\epsilon_2^2,\\
		&\varepsilon_3=\frac{2\rho_2(L)\epsilon_1-13}{4},\\
		&\varepsilon_4=\big(3+4(1+c_1^{-1})\big)\rho^2(L)\epsilon_1^2-\rho_2(L)\epsilon_1+\rho(L)\\
			&\quad +2\epsilon_1^2+2\epsilon_1+4,\\	
		&\varepsilon_5=3\epsilon_1^2\epsilon_2^2\rho^2(L)+ (\epsilon_1\epsilon_2^2+\epsilon_2^2)\rho(L)-\frac{1}{2}\epsilon_1\epsilon_2\rho_2(L) \\
			&\quad + \epsilon_1^2\epsilon_2^2+2\epsilon_1\epsilon_2^2+\epsilon_2^2+\epsilon_1\epsilon_2+\frac{1}{2},\\	
		&\varepsilon_6=\max\Big\{\frac{1+\epsilon_1\rho_2(L)}{2},
			~\frac{1+\epsilon_1}{2}+\frac{1}{2\epsilon_1\rho_2^2(L)}\Big\},\\
		&\varepsilon_7=\frac{\epsilon_1\rho_2(L)-1}{2\epsilon_1\rho_2(L)},\\
		&\varepsilon_8=(1+\epsilon_1)\rho_2^{-1}(L)+\rho_2^{-2}(L),\\
		&\varepsilon_{10}=\rho(L)\epsilon_1\epsilon_2+3\rho(L)\epsilon_2^2+\epsilon_1+\epsilon_2+1,\\
		&\varepsilon_{11}=\frac{3\epsilon_4}{2\epsilon_2^2}(\varepsilon_8+\varepsilon_1),\\
		&\varepsilon_{12}=10+16(1+c_1^{-1})+ \ell \\
			&\quad +\frac{\big(4(\epsilon_1+1)^2+2\epsilon_1\epsilon_2+2\epsilon_2\big)\rho_2^{-1}(L)+\rho_2^{-2}(L)}{\epsilon_2} \ell ^2\\
			&\quad +(3+2\rho_2^{-2}(L)+2\varepsilon_8+2\varepsilon_1) \ell ^2,\\	
		&\varepsilon_{13}=\frac{1}{2}\rho(L)\delta_0\epsilon_1+2\delta_0+\rho(L)\delta_0,\\
		&\varepsilon_{14}=\big(3\rho^2(L)+4(1+c_1^{-1})\rho^2(L)+2\big)\delta_0\epsilon_1^2\\
			& \quad +\big(2-\rho_2(L)\big)\delta_0\epsilon_1+\big(3+\rho(L)\big)\delta_0,\\			
		&\tilde{\varepsilon}_9
			=4\big(\rho_2^{-2}(L)+ 2(\epsilon_1+1)^2\rho_2^{-1}(L)\big) \ell ^4\epsilon_2 \\
			&\quad +4\big(2\rho_2^{-2}(L)+(\epsilon_1+1)\rho_2^{-1}(L)+2\big) \ell ^4\epsilon_2^2,\\
		&\tilde{\varepsilon}_{12}=6+16(1+c_1^{-1})+ \ell +\frac{\rho_2^{-2}(L)}{\epsilon_2} \ell ^2\\
			&\quad +\big(2\rho_2^{-2}(L)+\frac{2(\epsilon_1+1)^2+2\epsilon_1\epsilon_2+\epsilon_2}{\epsilon_2}\rho_2^{-1}(L)+2\big) \ell ^2,\\
		&b_{6,k}=\epsilon_1^2\epsilon_2^2+2\epsilon_1\epsilon_2^2+\epsilon_2^2+\epsilon_1\epsilon_2\\
			&\quad +(\epsilon_1\epsilon_2^2+\epsilon_2^2+\frac{1}{2}\epsilon_1\epsilon_2+\epsilon_2)\rho(L) + 3\epsilon_1^2\epsilon_2^2\rho^2(L),\\
		& a_1=\frac{1}{2}(\varepsilon_3\epsilon_2-\varepsilon_4\epsilon_2^2),\\
		& a_2=\frac{1}{16}\epsilon_2-\big(\frac{5}{4}+2(1+c_1^{-1})\big)\rho(L)\epsilon_2^2,\\
		& a_3=\frac{1}{2}c_2-\frac{1}{2}\delta_0\varepsilon_2,\\
		& a_4=2\varepsilon_{12}\sigma^2_1+\frac{\varepsilon_{11}\check{\sigma}_2^2}{p} 
			+4(1+\eta_1^2)\varepsilon_{12}\check{\sigma}_2^2,\\
		& a_5= \ell ^2\Big(\frac{(29+16c_1^{-1})\epsilon_2+\frac{7}{2}}{p} + \frac{1}{4}\Big),\\
		& a_6=2\frac{\rho_2^{-1}(L)}{\epsilon_2}, \hspace{5.8em}  	
			 a_7=2(\sigma^2_1+2(1+\eta_1^2)\check{\sigma}_2^2) \ell, 
			\\
		& a_8=2\ell\Big(\frac{\varepsilon_{11}}{p} + 4(1+\eta_1^2)\varepsilon_{12}\Big), \\
		& a_9=8(1+\eta_1^2)\ell^2,\hspace{4.75em}
			a_{10}=\frac{2\ell^2a_8}{a_1}, \\
		& a_{11}=a_7+\ell^2\kappa_{\mu}^2,
			\hspace{5.0em}
 			a_{12}=\frac{2\ell^2(a_4+a_5\kappa_{\mu}^2)}{a_1},\\
		&\tilde{ a}_4=2(\sigma^2_1+2(1+\eta_1^2)\check{\sigma}_2^2)\tilde{\varepsilon}_{12},\\
		&\tilde{ a}_5= \ell ^2\Big(\frac{(24+16c_1^{-1})\epsilon_2+\frac{5}{2}}{p} + \frac{1}{4}\Big),\\
			&\tilde{ a}_5^{\prime}= \ell ^2\Big(\frac{(24+16c_1^{-1})\epsilon_2+\frac{5}{2}}{p} + \frac{1}{4}  + 2a_6\frac{\epsilon_2}{\varepsilon_0p}\Big),\\
		& \tilde{ a}_8=8(1+\eta_1^2)\ell\tilde{\varepsilon}_{12},\hspace{3.9em}
		 \tilde{a}_{10}=\frac{2\ell^2\tilde{a}_8}{a_1},\\	
		&	\tilde{a}_{12}=\frac{2\ell^2(\tilde{a}_4+\tilde{a}_5\kappa_{\mu}^2)}{a_1},
			 \hspace{2.5em}	c_1=\frac{\delta\omega r}{2}, \\
		& c_2=c_1+2c_1^2, \hspace{6.18em}
		 \tilde{d}_1=\frac{1}{\varepsilon_6}\min\{a_1, 2a_2, 2a_3\}.
			\end{align*}

\subsection{Technical Preliminaries}\label{zero:app-lemmas}
This part collects several standard inequalities and auxiliary results used throughout the proofs.

\subsubsection{Smoothness}

	\begin{lemma}[Lemma~3.4, \cite{bubeck2015convex}]
		Let $f:\mathbb{R}^p\to\mathbb{R}$ be differentiable and $\ell$-smooth.
		Then, for all $x, y \in \mathbb{R}^p$, the following inequalities hold:
		\begin{subequations}
			\begin{align}
				&|f(y) - f(x) - (y - x)^\top \nabla f(x)| \le \frac{\ell}{2} \|y - x\|^2, \label{nonconvex:lemma:lipschitz} \\
				&\frac{1}{2} \|\nabla f(x)\|^2 \le \ell (f(x) - f^*). \label{rePL}
			\end{align}
		\end{subequations}
	\end{lemma}

\subsubsection{Properties of Graph Matrices}
 \begin{lemma}[Lemma 3,\cite{Yi_CommunicationCompression_2023}]
	The matrices $L, E$ are positive semi-definite and $F_M$ is positive definite.
	 Moreover, they satisfy
	\begin{subequations}		
		\begin{align}
			&EL=LE=L,\label{nonconvex:KL-L-eq}\\ 
			&0\le\rho_2(L)E\le L\le\rho(L)E,\label{nonconvex:KL-L-eq2}\\
			&F_ML=LF_M=E,\label{nonconvex:lemma-eq3}\\
			&\rho^{-1}(L){\bf I}_n\leq F_M \le\rho_2^{-1}(L){\bf I}_n.\label{nonconvex:lemma-eq5}
		\end{align}
	\end{subequations}
 \end{lemma}

	\subsubsection{Properties of Series}

		\begin{lemma}\label{series}
		Let $a_1>0$, $a_2>0$, $d\in(0,1)$ and  $m>a_1$. 
		Suppose a nonnegative sequence $\{\psi_k\}$ satisfies
		\begin{align}\label{eq:rec}
			\psi_{k+1}\ \le\ \Big(1-\frac{a_1}{k+m}\Big)\psi_k\ +\ \frac{a_2}{k+m}\,(1-d)^k,\forall k\in\mathbb{N}_0.
		\end{align}
		Then, for all $k\in\mathbb{N}_+$,
			\begin{align}\label{eq:explicit}
				&\psi_k\ \le\ \Big(\frac{m}{k+m}\Big)^{a_1} \psi_0 + \frac{2^{2a_1-1}(t_0+1)t_0^{a_1-1}a_2}{(k+m)^{a_1}}  \nonumber\\
				&\quad + \frac{2^{2a_1-1}(m+1)^{a_1-1}a_2}{(k+m)^{a_1}d} 
					+ \frac{2^{2a_1-1}\Gamma(a_1)a_2}{d_1^{a_1}(k+m)^{a_1}},
			\end{align}
			where $t_0=\lceil\frac{a_1-1}{d}\rceil$ and $\Gamma(\cdot)$ is the Gamma function.
		\end{lemma}

		\begin{proof}
		From \eqref{eq:rec}, it follows that
		\begin{align}\label{eq:var}
			&\psi_k
			\le \prod_{t=0}^{k-1}\Big(1-\frac{a_1}{t+m}\Big)\psi_0 \nonumber\\
			&\quad +\sum_{i=0}^{k-1}\frac{a_2}{i+m}\,(1-d)^i\prod_{t=i+1}^{k-1}\Big(1-\frac{a_1}{t+m}\Big).
		\end{align}
		Since $\log(1-x)\le -x$ for $x\in(0,1)$ and $a_1/(t+m)\in(0,1)$ by $m>a_1$, we have
		\begin{align}\label{prod}
			&\prod_{t=i+1}^{k-1}\Big(1-\frac{a_1}{t+m}\Big)
			\le \exp\!\Big(-a_1\sum_{t=i+1}^{k-1}\frac{1}{t+m}\Big) \nonumber\\
			&\le \exp\!\Big(-a_1\int_{i+m+1}^{k+m}\frac{1}{t}\,dt   \Big) \le \Big(\frac{i+1+m}{k+m}\Big)^{a_1}.
		\end{align}
		Substituting \eqref{prod} into \eqref{eq:var} yields
		\begin{align}\label{eq:after-prod}
		&\psi_k\ \le\ \Big(\frac{m}{k+m}\Big)^{a_1} \psi_0 \nonumber\\
		&\quad + \frac{a_2}{(k+m)^{a_1}}\sum_{i=0}^{k-1}(1-d)^i\,\frac{(i+1+m)^{a_1}}{i+m} \nonumber\\
		&\le\ \Big(\frac{m}{k+m}\Big)^{a_1} \psi_0
		+\frac{2^{a_1} a_2}{(k+m)^{a_1}}\sum_{i=0}^{k-1}(1-d)^i (i+m)^{a_1-1}.
		\end{align}

		Denote $m_0=\lceil m\rceil\le m+1$. For the second term in \eqref{eq:after-prod}, we have
		\begin{subequations}
			\begin{align}
				&\sum_{i=0}^{k-1}(1-d)^i (i+m)^{a_1-1}
				= \sum_{i=0}^{m_0}(1-d)^i (i+m)^{a_1-1} \nonumber\\
				&+\sum_{i=m_0+1}^{k-1}(1-d)^i (i+m)^{a_1-1}\triangleq S_1 + S_2, \\
				&S_1 \le (2(m+1))^{a_1-1}\sum_{i=0}^{\infty}(1-d)^i\ \le\ \frac{(2(m+1))^{a_1-1}}{d},\\
				&\frac{S_2}{2^{a_1-1}} \le \sum_{i=1}^{\infty} i^{a_1-1}e^{-di}
				=\sum_{i=0}^{t_0} i^{a_1-1}e^{-di}+\sum_{i=t_0+1}^{\infty} i^{a_1-1}e^{-di} \nonumber\\
				&\le (t_0+1)t_0^{a_1-1} +  \int_{t_0}^{\infty} t^{a_1-1}e^{-dt}\,dt \nonumber\\
				&\le (t_0+1)t_0^{a_1-1} + \frac{\Gamma(a_1)}{d^{a_1}}, \label{ineq:S2}
			\end{align}
		\end{subequations}
		where the first inequality in \eqref{ineq:S2} follows from $(1-d)^i\le e^{-di}$ for $d\in(0,1)$;
		the second inequality in \eqref{ineq:S2} holds since the function $t^{a_1-1}e^{-dt}$ is decreasing for $t\ge t_0$;
		and the last inequality in \eqref{ineq:S2} follows from the property of the Gamma function.

		Finally, the combination of \eqref{eq:after-prod}--\eqref{ineq:S2} leads to \eqref{eq:explicit}.
		\end{proof}

\subsection{Proof of Lemma~\ref{Lemma:grad prop}} \label{proof:lemma1}

		The inequalities \eqref{zerosg:rand-grad-esti1}--\eqref{zerosg:rand-grad-smooth} 
		are obtained by Lemma~6 in \cite{Yi_Zerothorder_2022,Yi_Zerothorder_2021}.
		We then prove \eqref{zerosg-a5:rand-grad-esti2} and \eqref{zerosg-a5:vkLya-2}.

		Under\hspace{-0.05em} Assumption  \ref{zerosg:ass:zeroth-smooth}, \eqref{zerosg:lemma:uniformsmoothing-equ6-1} holds,\hspace{0.3em}and we bound its first term.
		\begin{align}\label{zerosg:rand-grad2-esti1}
			&2p\mathbb{E}_{\xi_{i,k}}[\|\nabla_{x}F_{i}(x_{i,k},\xi_{i,k})\|^2]\nonumber\\
			&=2p\mathbb{E}_{\xi_{i,k}}[\|\nabla_{x}F_{i}(x_{i,k},\xi_{i,k})-\nabla f_i(x_{i,k})+\nabla f_i(x_{i,k})\|^2]\nonumber\\
			&\le 4p\mathbb{E}_{\xi_{i,k}}[\|\nabla_{x}F_{i}(x_{i,k},\xi_{i,k})-\nabla f_i(x_{i,k})\|^2+\|\nabla f_i(x_{i,k})\|^2]\nonumber\\
			&\le 4p(1+\eta_1^2)\|\nabla f_i(x_{i,k})\|^2+4p\sigma^2_1,
		\end{align}
		where 
		the first inequality holds due to the Cauchy--Schwarz inequality;
		and the last inequality holds due to Assumption~\ref{zerosg:ass:zeroth-variance} and the independence of $x_{i,k}$ and $\xi_{i,k}$.

	Combining \eqref{zerosg:lemma:uniformsmoothing-equ6-1} and \eqref{zerosg:rand-grad2-esti1} yields
	\begin{align}
		&\mathbb{E}_{\mathcal{B}_k}[\|\bsg_k^z\|^2]
		\le 4p(1+\eta_1^2)\|\bsg_k\|^2 + 4np\sigma^2_1 + \frac{1}{2}np^2\ell^2\mu_k^2 \nonumber\\
		&= 4p(1+\eta_1^2)\|\bsg_k-\bsg_k^0+\bsg_k^0\|^2  
			+ 4np\sigma^2_1 + \frac{1}{2}np^2\ell^2\mu_k^2 \nonumber\\
		&\le  8p(1+\eta_1^2)\|\bsg_k-\bsg_k^0\|^2 + 8p(1+\eta_1^2)\|\bsg_k^0\|^2  \nonumber\\
			&+ 4np\sigma^2_1 + \frac{1}{2}np^2\ell^2\mu_k^2, \label{g_k^z0}
	\end{align}
	where the first inequality follows from the definitions of $\mathcal{B}_k$ and $\mu_k$, and the last follows from the Cauchy--Schwarz inequality.

	For the first term on the right-hand side of \eqref{g_k^z0}, Assumption~\ref{zerosg:ass:zeroth-smooth} yields
	\begin{align}
		&\|\bsg^0_{k}-\bsg_{k}\|^2
		=\sum\nolimits_{i=1}^n\|\nabla f_i(\bar{x}_k)-\nabla f_i(x_{i,k})\|^2 \nonumber\\
		&\hspace{-0.3em}\le\sum\nolimits_{i=1}^n\ell^2\|\bar{x}_k-x_{i,k}\|^2 
		= \ell^2\|\bar{\bsx}_{k}-\bsx_{k}\|^2
		=\ell^2\|\bsx_{k}\|^2_{\bsE}.\label{zerosg:gg1}
	\end{align}

	For the second term on the right-hand side of \eqref{g_k^z0},
	by Assumption~\ref{zerosg:ass:zeroth-smooth} and the standard smoothness inequality (see, e.g., Eq.~(3.5) on p.~267 of~\cite{bubeck2015convex})
	\begin{align*}
	\frac{1}{2}\|\nabla f_i(x)\|^2 \le \ell\big(f_i(x)-f_i^*\big),~\forall x\in\mathbb{R}^p,
	\end{align*}
	we obtain 
	\begin{align}\label{zerosg-a5:rand-grad2-esti1.2}
		&\|\bsg^0_{k}\|^2 = \sum_{i=1}^{n}\|\nabla f_i(\bar{x}_k) \|^2
		\le \sum_{i=1}^{n} 2\ell(f_i(\bar{x}_k)-f_i^*) \nonumber\\
		&=  2\ell (\sum_{i=1}^{n}f_i(\bar{x}_k)-nf^*) + 2\ell( nf^*-\sum_{i=1}^{n}f_i^*) \nonumber\\
		&= 2\ell e_{4,k}+n\check{\sigma}^2_2,
	\end{align}
	where the last equality follows from the definitions of $e_{4,k}$, $\tilde{f}(\bsx_{k})$, and $\check{\sigma}_2$.
	As a result, \eqref{g_k^z0}--\eqref{zerosg-a5:rand-grad2-esti1.2} gives \eqref{zerosg-a5:rand-grad-esti2}.
	
	From  the Cauchy--Schwarz inequality, \eqref{zerosg:gg-rand-pd}, and \eqref{zerosg-a5:rand-grad2-esti1.2},  
		\begin{align}
		&\|\bsg^0_{k+1}\|^2
			=\|\bsg^0_{k+1}-\bsg^0_{k}+\bsg^0_{k}\|^2 \nonumber\\
		&\le2(\|\bsg^0_{k+1}-\bsg^0_{k}\|^2+\|\bsg^0_{k}\|^2)  \nonumber\\
		&\le 2(\alpha^2_k\ell ^2\|\bsg_k^z\|^2+2\ell e_{4,k}+n\check{\sigma}^2_2),
	\end{align}
	which completes the proof.

\subsection{Proof of Lemma~\ref{Lemma:Lyap:fix}} \label{appendix:iterative:Lemma}
\subsubsection{Iterative Difference of Lyapunov Component Terms} 

To analyze the iterative difference of the Lyapunov function, we separately characterize the iterative differences of its component terms $e_{1,k}$--$e_{5,k}$ in Lemmas~\ref{Lemma Consensus}--\ref{Lemma Compression}.
In addition to the notations defined in Appendix~\ref{appendix:constant:Thm},
we denote $\bsL=L\otimes {\bf I}_p$, $\hat{\bsx}_k=\col(\hat{x}_{1,k}, \dots,\hat{x}_{n,k})$, $\Delta_k=\frac{1}{\gamma_{k}}-\frac{1}{\gamma_{k+1}}$, and

	\begin{align*}
		&\epsilon_5=\frac{(\beta_k+\gamma_k)^2}{\gamma_k^5}\rho_2^{-1}(L), 
			\hspace{2.3em} \epsilon_6=\frac{\beta_k+\gamma_k}{2\gamma_k^3}\rho_2^{-1}(L)
			+ \frac{1}{2\gamma_k^2},\\	
		&\epsilon_7=\frac{1}{2\gamma_k^2}\rho_2^{-2}(L), 
			\hspace{5.3em} \epsilon_8=\frac{1}{2\gamma_k^2}+\frac{1}{\gamma_k^2}\rho_2^{-2}(L).
	\end{align*}

\begin{lemma}\label{Lemma Consensus}
		Suppose that Assumptions~\ref{ass:graph} and \ref{zerosg:ass:zeroth-smooth} hold. 
		Let $\{\bsx_k\}$ be the sequence generated by Algorithm~\ref{nonconvex:algorithm-pdgd}.
		Then
		\begin{align}
			&\mathbb{E}_{\mathcal{A}_k}[e_{1,k+1}]\le e_{1,k}-\|\bsx_k\|^2_{\frac{\alpha_k\beta_k}{2}\bsL-\frac{\alpha_k}{2}\bsE- \alpha_k(1+5\alpha_k) \ell ^2\bsE}
			\nonumber\\
			&\quad
					+ \mathbb{E}_{\mathcal{C}_k}\big[\|\hat{\bsx}_k\|^2_{\frac{3}{2}\alpha_k^2\beta_k^2\bsL^2}\big]
				+ n \ell ^2\alpha_k(1+5\alpha_k)\mu^2_k \nonumber\\
				&\quad + \frac{\alpha_k}{2}(\beta_k+2\gamma_k)\rho(L)\mathbb{E}_{\mathcal{C}_k}[\|\bsx_k-\hat{\bsx}_k\|^2] \nonumber\\
				&\quad -\alpha_k\gamma_k\mathbb{E}_{\mathcal{C}_k}[\hat{\bsx}^\top_k]\bsE\Big(\bsv_k+\frac{1}{\gamma_k}\bsg_k^0\Big)  + 2\alpha_k^2\mathbb{E}_{\mathcal{B}_k}[\|\bsg_k^z\|^2]\nonumber\\
				&\quad +\Big\|\bsv_k+\frac{1}{\gamma_k}\bsg_k^0\Big\|^2_{\frac{6\alpha_k^2\gamma_k^2\rho(L)+\alpha_k\gamma_k}{4}\bsF}.\label{nonconvex:v1k}
		\end{align}

\end{lemma}

	\begin{lemma}\label{Lemma Optimality 1}
				Suppose that Assumptions~\ref{ass:graph}  and \ref{zerosg:ass:zeroth-smooth} hold, $\{\gamma_k\}$ is non-decreasing, and $\beta_k/\gamma_k=\epsilon_1$. 
		Let $\{\bsx_k\}$ be the sequence generated by Algorithm~\ref{nonconvex:algorithm-pdgd}.
		Then
		\begin{align}
			&\mathbb{E}_{\mathcal{A}_k}[ e_{2,k+1}]\le e_{2,k}  
					+\frac{1}{2}\varepsilon_1(\Delta_k+\Delta_k^2)\mathbb{E}_{\mathcal{A}_k}[\|\bsg_{k+1}^0\|^2]
			\nonumber\\
			&\quad
			 +(1+\Delta_k)\alpha_k(\beta_k+\gamma_k)\mathbb{E}_{\mathcal{C}_k}[\hat{\bsx}^\top_k]\bsE\Big(\bsv_k+\frac{1}{\gamma_k}\bsg_k^0\Big)\nonumber\\
			&\quad+\mathbb{E}_{\mathcal{C}_k}\big[\|\hat{\bsx}_k\|^2_{(1+\Delta_k)\frac{1}{2}\alpha_k^2(\beta_k\gamma_k+\gamma_k^2)\bsL + (1+\Delta_k)\frac{1}{2}\alpha_k^2(\beta_k+\gamma_k)^2\bsE}\big] \nonumber\\
			&\quad +\Big\|\bsv_k+\frac{1}{\gamma_k}\bsg_{k}^0\Big\|^2_{\big(\alpha_k\frac{\gamma_k}{4}+\Delta_k(\frac{\beta_k+\gamma_k}{2\gamma_k}+\frac{\alpha_k\gamma_k}{4})\big)\bsF}\nonumber\\
			&\quad +(1+\Delta_k)(\epsilon_5\alpha_k+\epsilon_6\alpha_k^2) \ell ^2\mathbb{E}_{\mathcal{B}_k}[\|\bar{\bsg}_k^z\|^2].
			\label{zerosg:v2k}
		\end{align}
	\end{lemma}

	\begin{lemma}\label{Lemma Cross}
				Under Assumptions~\ref{ass:graph}  and \ref{zerosg:ass:zeroth-smooth}, if $\{\gamma_k\}$ is non-decreasing 
		with the sequence $\{\bsx_k\}$ generated by Algorithm~\ref{nonconvex:algorithm-pdgd}, then
		\begin{align}
			&\mathbb{E}_{\mathcal{A}_k}[ e_{3,k+1}]
				\le e_{3,k} +\|\bsx_k\|^2_{\alpha_k(\frac{\gamma_k+2}{4}+\frac{1}{2} \ell ^2)\bsE+3\alpha_k^2 \ell ^2\bsE}
			\nonumber\\
			&\quad
			-(1+\Delta_k)\alpha_k\beta_k\mathbb{E}_{\mathcal{C}_k}[\hat{\bsx}_k^\top]\bsE\Big(\bsv_k+\frac{1}{\gamma_k}\bsg_{k}^0\Big)\nonumber\\
			&\quad +\mathbb{E}_{\mathcal{C}_k}\Big[\|\hat{\bsx}_k\|^2_{\alpha_k\gamma_k\bsE
			 +\alpha_k^2\big((\frac{1}{2}\beta_k^2+\gamma_k^2)\bsE-\beta_k\gamma_k\bsL\big)+\frac{1}{2}\Delta_k\alpha_k\beta_k\bsE}\Big] \nonumber\\
			&\quad -\Big\|\bsv_k+\frac{1}{\gamma_k}\bsg_{k}^0\Big\|^2_{\big(\alpha_k(\frac{3}{4}\gamma_k-\rho_2^{-1}(L))-\alpha_k^2\gamma_k^2\rho(L)-\frac{1}{2}\rho(L)\Delta_k\alpha_k\beta_k\big)\bsF}\nonumber\\
			&\quad +(\alpha_k\epsilon_7+\alpha_k^2\epsilon_8) \ell ^2\mathbb{E}_{\mathcal{B}_k}\big[\|\bar{\bsg}_k^z\|^2\big]  +n \ell ^2\alpha_k(\frac{1}{2}+3\alpha_k)\mu_k^2 \nonumber\\
			&\quad + \alpha^2_k\mathbb{E}_{\mathcal{B}_k}[\|\bsg_k^z\|^2] +\frac{\alpha_k}{\gamma_k}\rho_2^{-1}(L)\|\bar{\bsg}_k^\mu\|^2 \nonumber\\
			&\quad+\frac{1}{2}\Delta_k\mathbb{E}_{\mathcal{A}_k}[2e_{1,k+1}+\rho_2^{-2}(L)\|\bsg_{k+1}^0\|^2].
		\label{zerosg:v3k}
			\end{align}

	\end{lemma}

	\begin{lemma}\label{Lemma Optimality 2}
				Suppose that Assumptions~\ref{ass:graph} and \ref{zerosg:ass:zeroth-smooth} hold, \eqref{eq:global_lower_bound} holds, and $\{\gamma_k\}$ is non-decreasing. 
		Let $\{\bsx_k\}$ be the sequence generated by Algorithm~\ref{nonconvex:algorithm-pdgd}.
		Then
		\begin{align}
			&\mathbb{E}_{\mathcal{A}_k}[e_{4,k+1}]
			\le e_{4,k} - \frac{1}{4}\alpha_k\|\bar{\bsg}^\mu_{k}\|^2 + \|\bsx_k\|^2_{\alpha_k \ell ^2\bsE}\nonumber\\
			&\quad + n \ell ^2\alpha_k\mu^2_k-\frac{1}{4}\alpha_k\|\bar{\bsg}_{k}^0\|^2
			+\frac{1}{2}\alpha^2_k \ell \mathbb{E}_{\mathcal{B}_k}[\|\bar{\bsg}^z_{k}\|^2].\label{zerosg:v4k}
		\end{align}
	\end{lemma}

	\begin{lemma} \label{Lemma Compression}
				Suppose that Assumptions~\ref{ass:graph}  and \ref{zerosg:ass:zeroth-smooth} hold, $\{\gamma_k\}$ is non-decreasing, and $\omega\le1/r$.
		Let $\{\bsx_k\}$ be the sequence generated by
		Algorithm~\ref{nonconvex:algorithm-pdgd}.
		Then
		\begin{align}
			&\mathbb{E}_{\mathcal{A}_k}[e_{5,k+1}]
				\le e_{5,k}-c_2\|\bsx_{k}-\bsy_{k}\|^2
			\nonumber\\
			&\quad +4(1+c_1^{-1})\rho^2(L)\alpha_k^2\beta_k^2 \mathbb{E}_{\mathcal{C}_k}[\|\bsx_k-\hat{\bsx}_k\|_\bsE^2] \nonumber\\
			&\quad + \|\bsx_k\|^2_{4(1+c_1^{-1})\alpha^2_k(\beta_k^2\rho^2(L)+4 \ell ^2)\bsE}\nonumber\\
			&\quad + 16(1+c_1^{-1})n \ell ^2\alpha_k^2\mu_k^2 + 8(1+c_1^{-1})\alpha_k^2\mathbb{E}_{\mathcal{B}_k}[\|\bsg_k^z\|^2]\nonumber\\
			&\quad + \Big\|\bsv_k+\frac{1}{\gamma_k}\bsg_k^0\Big\|^2_{4(1+c_1^{-1})\rho(L)\alpha_k^2\gamma_k^2\bsF}.
			\label{nonconvex:xminush}
		\end{align}
	\end{lemma}

\begin{proof}
	\hspace{0.15em}Without \hspace{0.15em}ambiguity, \hspace{0.15em}denote \hspace{0.2em}$\mathcal{C}(\bsx)$\hspace{0.35em}$=$\hspace{0.35em}$\col(\mathcal{C}(x_1),\dots,\mathcal{C}$\\$(x_n))$ and $\bsq_k=\mathcal{C}(\bsx_k-\bsy_k)$.
	To simplify the analysis, we present the compact form of \eqref{nonconvex:kia-algo-dc-a}, \eqref{nonconvex:kia-algo-dc-x} and \eqref{nonconvex:kia-algo-dc-v}.
	\begin{subequations}\label{nonconvex:alg-compress-compact}
			\begin{align}
				\bsy_{k+1}&=\bsy_{k}+\omega\bsq_k,\label{nonconvex:kia-algo-dc-compact-a}\\
				\bsx_{k+1}&=\bsx_k-\alpha_k(\beta_k\bsL\hat{\bsx}_k+\gamma_k\bsv_k+\bsg_k^z),\label{nonconvex:kia-algo-dc-compact-x}\\
				\bsv_{k+1}&=\bsv_k+\alpha_k\gamma_k\bsL\hat{\bsx}_k.\label{nonconvex:kia-algo-dc-compact-v}
			\end{align}	
		\end{subequations}
According to \eqref{nonconvex:kia-algo-dc-compact-v}, and noting that $\sum_{i=1}^{n}L_{ij}=0$ and  $\sum_{i=1}^{n}v_{i,0}={\bf0}_d$, it follows that
	\begin{align}
		\bar{v}_k={\bm 0}_d.\label{nonconvex:vkn}
	\end{align}
Then, combining \eqref{nonconvex:vkn} with \eqref{nonconvex:kia-algo-dc-compact-x}, we further get
	\begin{align}
		&\bar{\bsx}_{k+1}=\bar{\bsx}_{k}-\alpha_k\bar{\bsg}_k.\label{nonconvex:xbardynamic}
	\end{align}
		Now it is ready to prove Lemmas~\ref{Lemma Consensus}--\ref{Lemma Compression}.


\noindent {\bf (i)}
This step is to show the relation between $e_{1,k+1}$ and $e_{1,k}$.
		\begin{align}
			&\mathbb{E}_{\mathcal{B}_k}[e_{1,k+1}]=\mathbb{E}_{\mathcal{B}_k}\Big[\frac{1}{2}\|\bsx_{k+1} \|^2_{\bsE}\Big]\nonumber\\
			&=\mathbb{E}_{\mathcal{B}_k}\Big[\frac{1}{2}\|\bsx_k-\alpha_k(\beta_k\bsL\hat{\bsx}_k+\gamma_k\bsv_k+\bsg_k^{z}) \|^2_{\bsE}\Big]\nonumber\\
			&=\mathbb{E}_{\mathcal{B}_k}\Big[\frac{1}{2}\|\bsx_k\|^2_{\bsE}-\alpha_k\beta_k\bsx^\top_k\bsL\hat{\bsx}_k
			+\|\hat{\bsx}_k\|^2_{\frac{\alpha_k^2\beta_k^2}{2}\bsL^2}\nonumber\\
			&\quad-\alpha_k\gamma_k(\bsx^\top_k-\alpha_k\beta_k\hat{\bsx}_k^\top\bsL)\bsE
			\Big(\bsv_k+\frac{1}{\gamma_k}\bsg_k^{z}\Big)\nonumber\\
			&\quad+\Big\|\bsv_k+\frac{1}{\gamma_k}\bsg_k^{z}\Big\|^2_{\frac{\alpha_k^2\gamma_k^2}{2}\bsE}\Big]\nonumber\\
			&=\frac{1}{2}\|\bsx_k\|^2_{\bsE}-\alpha_k\beta_k\bsx^\top_k\bsL(\bsx_k+\hat{\bsx}_k-\bsx_k)
			+\|\hat{\bsx}_k\|^2_{\frac{\alpha_k^2\beta_k^2}{2}\bsL^2}\nonumber\\
			&\quad-\alpha_k\gamma_k(\bsx^\top_k-\alpha_k\beta_k\hat{\bsx}_k^\top\bsL)\bsE\Big(\bsv_k
			+\frac{1}{\gamma_k}\bsg_k^0\nonumber\\
			&\quad +\frac{1}{\gamma_k}\bsg_k^\mu-\frac{1}{\gamma_k}\bsg_k^0\Big)\nonumber\\
			&\quad+\mathbb{E}_{\mathcal{B}_k}\Big[\Big\|\bsv_k+\frac{1}{\gamma_k}\bsg_k^0
			+\frac{1}{\gamma_k}\bsg_k^{z}-\frac{1}{\gamma_k}\bsg_k^0\Big\|^2_{\frac{\alpha_k^2\gamma_k^2}{2}\bsE}\Big]\nonumber\\
			&\le\frac{1}{2}\|\bsx_k\|^2_{\bsE}-\|\bsx_k\|^2_{\alpha_k\beta_k\bsL}
			+\|\bsx_k\|^2_{\frac{\alpha_k\beta_k}{2}\bsL} \nonumber\\
			&\quad +\|\hat{\bsx}_k-\bsx_k\|^2_{\frac{\alpha_k\beta_k}{2}\bsL} 
			+\|\hat{\bsx}_k\|^2_{\frac{\alpha_k^2\beta_k^2}{2}\bsL^2} \nonumber\\
			&\quad-\alpha_k\gamma_k\bsx^\top_k\bsE\Big(\bsv_k+\frac{1}{\gamma_k}\bsg_k^0\Big) \nonumber\\
			&\quad +\frac{\alpha_k}{2}\|\bsx_k\|^2_{\bsE}
			+\frac{\alpha_k}{2}\|\bsg_k^\mu-\bsg_k^0\|^2\nonumber\\
			&\quad +\|\hat{\bsx}_k\|^2_{\frac{\alpha_k^2\beta_k^2}{2}\bsL^2} 
			 +\frac{\alpha_k^2\gamma_k^2}{2}\Big\|\bsv_k+\frac{1}{\gamma_k}\bsg_k^0\Big\|^2\nonumber\\
			&\quad+\|\hat{\bsx}_k\|^2_{\frac{\alpha_k^2\beta_k^2}{2}\bsL^2}
			+\frac{\alpha_k^2}{2}\|\bsg_k^\mu-\bsg_k^0\|^2\nonumber\\
			&\quad+\alpha_k^2\gamma_k^2\Big\|\bsv_k+\frac{1}{\gamma_k}\bsg_k^0\Big\|^2
			+\alpha_k^2\mathbb{E}_{\mathcal{B}_k}\Big[\|\bsg_k^z-\bsg_k^0\|^2\Big]\nonumber\\
			&=\frac{1}{2}\|\bsx_k\|^2_{\bsE}-\|\bsx_k\|^2_{\frac{\alpha_k\beta_k}{2}\bsL-\frac{\alpha_k}{2}\bsE}
			+\|\hat{\bsx}_k\|^2_{\frac{3\alpha_k^2\beta_k^2}{2}\bsL^2}\nonumber\\
			&\quad+\frac{\alpha_k}{2}(1+\alpha_k)\|\bsg_k^\mu-\bsg_k^0\|^2
			+\|\hat{\bsx}_k-\bsx_k\|^2_{\frac{\alpha_k\beta_k}{2}\bsL}\nonumber\\
			&\quad-\alpha_k\gamma_k(\hat{\bsx}_k+\bsx_k-\hat{\bsx}_k)^\top\bsE
			\Big(\bsv_k+\frac{1}{\gamma_k}\bsg_k^0\Big)\nonumber\\
			&\quad+\frac{3\alpha_k^2\gamma_k^2}{2}\Big\|\bsv_k+\frac{1}{\gamma_k}\bsg_k^0\Big\|^2+\alpha_k^2\mathbb{E}_{\mathcal{B}_k}\Big[\|\bsg_k^z-\bsg_k^0\|^2\Big] \nonumber\\
			&\le\frac{1}{2}\|\bsx_k\|^2_{\bsE}-\|\bsx_k\|^2_{\frac{\alpha_k\beta_k}{2}\bsL-\frac{\alpha_k}{2}\bsE}
			+\|\hat{\bsx}_k\|^2_{\frac{3\alpha_k^2\beta_k^2}{2}\bsL^2}\nonumber\\
			&\quad+\frac{\alpha_k}{2}(1+\alpha_k)\|\bsg_k^\mu-\bsg_k^0\|^2
			+\|\hat{\bsx}_k-\bsx_k\|^2_{\frac{\alpha_k}{2}(\beta_k\bsL+2\rho(L)\gamma_k\bsE)}\nonumber\\
			&\quad-\alpha_k\gamma_k\hat{\bsx}^\top_k\bsE\Big(\bsv_k+\frac{1}{\gamma_k}\bsg_k^0\Big)+\alpha_k^2\mathbb{E}_{\mathcal{B}_k}\Big[\|\bsg_k^z-\bsg_k^0\|^2\Big]\nonumber\\	
			&\quad+\frac{6\alpha_k^2\gamma_k^2+\alpha_k\gamma_k\rho^{-1}(L)}{4}\Big\|\bsv_k+\frac{1}{\gamma_k}\bsg_k^0\Big\|^2
			\nonumber\\
			&\le e_{1,k}-\|\bsx_k\|^2_{\frac{\alpha_k\beta_k}{2}\bsL-\frac{\alpha_k}{2}\bsE
				- \alpha_k(1+5\alpha_k) \ell ^2\bsE} + \|\hat{\bsx}_k\|^2_{\frac{3\alpha_k^2\beta_k^2}{2}\bsL^2}\nonumber\\
			&\quad + n \ell ^2\alpha_k(1+5\alpha_k)\mu^2_k + \frac{\alpha_k}{2}(\beta_k+2\gamma_k)\rho(L)\|\bsx_k-\hat{\bsx}_k\|^2 \nonumber\\
			&\quad -\alpha_k\gamma_k\hat{\bsx}^\top_k\bsE\Big(\bsv_k+\frac{1}{\gamma_k}\bsg_k^0\Big)  + 2\alpha_k^2\mathbb{E}_{\mathcal{B}_k}[\|\bsg_k^z\|^2]\nonumber\\
			&\quad +\Big\|\bsv_k+\frac{1}{\gamma_k}\bsg_k^0\Big\|^2_{\frac{6\alpha_k^2\gamma_k^2\rho(L)+\alpha_k\gamma_k}{4}\bsF},\label{nonconvex:v1k_B_k}
		\end{align}
		where the second, third and fourth equalities hold due to \eqref{nonconvex:kia-algo-dc-compact-x}, \eqref{nonconvex:KL-L-eq} and \eqref{zerosg:rand-grad-esti1}, respectively; the first and second inequalities hold due to the Cauchy--Schwarz inequality and $\rho(\bsE)=1$; and the last  inequality holds due to \eqref{zerosg:rand-grad-esti8}, \eqref{zerosg:rand-grad-esti6} and \eqref{nonconvex:lemma-eq5}.
		Then, since  $\mathcal{B}_k$ and $\mathcal{C}_k$ are independent and $\mathcal{A}_k = \mathcal{B}_k \cup \mathcal{C}_k$, taking the expectation with respect to $\mathcal{A}_k$ on both sides of \eqref{nonconvex:v1k_B_k} yields \eqref{nonconvex:v1k}.

		\noindent {\bf (ii)}
			This step is to show the relation between $e_{2,k+1}$ and $e_{2,k}$.

		\begin{align}
			&e_{2,k+1}=\frac{1}{2}\Big\|\bsv_{k+1}+\frac{1}{\gamma_{k+1}}\bsg_{k+1}^0\Big\|^2_{\frac{\beta_k+\gamma_k}{\gamma_k}\bsF}\nonumber\\
			&\quad=\frac{1}{2}\Big\|\bsv_{k+1}+\frac{1}{\gamma_{k}}\bsg_{k+1}^0
			+\Big(\frac{1}{\gamma_{k+1}}-\frac{1}{\gamma_{k}}\Big)\bsg_{k+1}^0\Big\|^2_{\frac{\beta_k+\gamma_k}{\gamma_k}\bsF}\nonumber\\
			&\quad\le\frac{1}{2}(1+\Delta_k)\Big\|\bsv_{k+1}+\frac{1}{\gamma_{k}}\bsg_{k+1}^0\Big\|^2_{\frac{\beta_k+\gamma_k}{\gamma_k}\bsF}\nonumber\\
			&\qquad+\frac{1}{2}(\Delta_k+\Delta_k^2)\|\bsg_{k+1}^0\|^2_{\frac{\beta_k+\gamma_k}{\gamma_k}\bsF},\label{zerosg:v2k-1}
		\end{align}
		where the inequality holds due to the Cauchy--Schwarz inequality.
		
		For the first term on the right-hand side of \eqref{zerosg:v2k-1}, we have
		\begin{align}
			&\frac{1}{2}\Big\|\bsv_{k+1}+\frac{1}{\gamma_k}\bsg_{k+1}^0\Big\|^2_{\frac{\beta_k+\gamma_k}{\gamma_k}\bsF}\nonumber\\
			&=\frac{1}{2}\Big\|\bsv_k+\frac{1}{\gamma_k}\bsg_{k}^0+\alpha_k\gamma_k\bsL\hat{\bsx}_k
			+\frac{1}{\gamma_k}(\bsg_{k+1}^0-\bsg_{k}^0) \Big\|^2_{\frac{\beta_k+\gamma_k}{\gamma_k}\bsF}\nonumber\\
			&=\frac{1}{2}\Big\|\bsv_k+\frac{1}{\gamma_k}\bsg_{k}^0\Big\|^2_{\frac{\beta_k+\gamma_k}{\gamma_k}\bsF}\nonumber\\ 
			&\quad +\alpha_k(\beta_k+\gamma_k)\hat{\bsx}^\top_k\bsE\Big(\bsv_k+\frac{1}{\gamma_k}\bsg_k^0\Big)\nonumber\\
			&\quad+\|\hat{\bsx}_k\|^2_{\frac{\alpha_k^2\gamma_k}{2}(\beta_k+\gamma_k)\bsL}
			+\frac{1}{2\gamma_k^2}\|\bsg_{k+1}^0-\bsg_{k}^0\|^2_{\frac{\beta_k+\gamma_k}{\gamma_k}\bsF}\nonumber\\
			&\quad +\frac{\beta_k+\gamma_k}{\gamma_k^2}\Big(\bsv_k+\frac{1}{\gamma_k}\bsg_{k}^0
			\Big)^\top\bsF(\bsg_{k+1}^0-\bsg_{k}^0)\nonumber\\
			&\quad +\alpha_k\frac{\beta_k+\gamma_k}{\gamma_k}\hat{\bsx}_k^\top\bsE(\bsg_{k+1}^0-\bsg_{k}^0)\nonumber\\
			&\le\frac{1}{2}\Big\|\bsv_k+\frac{1}{\gamma_k}\bsg_{k}^0\Big\|^2_{\frac{\beta_k+\gamma_k}{\gamma_k}\bsF}\nonumber\\
			 &\quad +\alpha_k(\beta_k+\gamma_k)\hat{\bsx}^\top_k\bsE\Big(\bsv_k+\frac{1}{\gamma_k}\bsg_k^0\Big)\nonumber\\
			&\quad+\|\hat{\bsx}_k\|^2_{\frac{\alpha_k^2\gamma_k}{2}(\beta_k+\gamma_k)\bsL}
			+\|\bsg_{k+1}^0-\bsg_{k}^0\|^2_{\frac{\beta_k+\gamma_k}{2\gamma_k^3}\bsF}\nonumber\\
			&\quad+\Big\|\bsv_k+\frac{1}{\gamma_k}\bsg_{k}^0\Big\|^2_{\frac{\alpha_k\gamma_k}{4}\bsF}
			+\|\bsg_{k+1}^0-\bsg_{k}^0\|^2_{\frac{(\beta_k+\gamma_k)^2}{\alpha_k\gamma_k^5}\bsF}\nonumber\\
			&\quad+\|\hat{\bsx}_k\|^2_{\frac{\alpha_k^2(\beta_k+\gamma_k)^2}{2}\bsE}
			+\frac{1}{2\gamma_k^2}\|\bsg_{k+1}^0-\bsg_{k}^0\|^2\nonumber\\
			&\le \frac{1}{2}\Big\|\bsv_k+\frac{1}{\gamma_k}\bsg_{k}^0\Big\|^2_{\frac{\beta_k+\gamma_k}{\gamma_k}\bsF} 
			 +\Big\|\bsv_k+\frac{1}{\gamma_k}\bsg_{k}^0\Big\|^2_{\frac{\alpha_k\gamma_k}{4}\bsF} \nonumber\\ 
			&\quad +\alpha_k(\beta_k+\gamma_k)\hat{\bsx}^\top_k\bsE\Big(\bsv_k+\frac{1}{\gamma_k}\bsg_k^0\Big)\nonumber\\
			&\quad+\|\hat{\bsx}_k\|^2_{\frac{\alpha_k^2}{2}(\beta_k\gamma_k+\gamma_k^2)\bsL+\frac{\alpha_k^2}{2}(\beta_k+\gamma_k)^2\bsE} \nonumber\\
			&\quad +\big((\frac{\beta_k+\gamma_k}{2\gamma_k^3}+\frac{(\beta_k+\gamma_k)^2}{\alpha_k\gamma_k^5})\rho_2^{-1}(L)
			  + \frac{1}{2\gamma_k^2}\big)\nonumber\\
			&\quad \times\|\bsg_{k+1}^0-\bsg_{k}^0\|^2 \nonumber\\
			&\le\frac{1}{2}\Big\|\bsv_k+\frac{1}{\gamma_k}\bsg_{k}^0\Big\|^2_{\frac{\beta_k+\gamma_k}{\gamma_k}\bsF}
			+\Big\|\bsv_k+\frac{1}{\gamma_k}\bsg_{k}^0\Big\|^2_{\frac{\alpha_k\gamma_k}{4}\bsF}\nonumber\\ 
			&\quad +\alpha_k(\beta_k+\gamma_k)\hat{\bsx}^\top_k\bsE\Big(\bsv_k+\frac{1}{\gamma_k}\bsg_k^0\Big)\nonumber\\
			&\quad+\|\hat{\bsx}_k\|^2_{\frac{\alpha_k^2}{2}(\beta_k\gamma_k+\gamma_k^2)\bsL+\frac{\alpha_k^2}{2}(\beta_k+\gamma_k)^2\bsE} \nonumber\\
			&\quad +(\epsilon_5\alpha_k+\epsilon_6\alpha_k^2) \ell ^2\|\bar{\bsg}_k^z\|^2
			,\label{nonconvex:v2k}
		\end{align}
		where the first equality holds due to \eqref{nonconvex:kia-algo-dc-compact-v}; the second equality holds due to \eqref{nonconvex:KL-L-eq} and \eqref{nonconvex:lemma-eq3}; the first inequality holds due to the Cauchy--Schwarz inequality and $\rho(\bsE)=1$; the second inequality holds due to \eqref{nonconvex:lemma-eq5}; and the last inequality holds due to \eqref{zerosg:gg-rand-pd}.

		For the second term on the right-hand side of \eqref{zerosg:v2k-1}, from \eqref{nonconvex:lemma-eq5}, we have
		\begin{align}\label{zerosg:v2k-4}
			\|\bsg_{k+1}^0\|^2_{\frac{\beta_k+\gamma_k}{\gamma_k}\bsF}
			\le\varepsilon_1\|\bsg_{k+1}^0\|^2.
		\end{align}


		Then, from \eqref{zerosg:v2k-1}--\eqref{zerosg:v2k-4}, we have
		\begin{align}
			&e_{2,k+1}\le e_{2,k}  +(1+\Delta_k)\alpha_k(\beta_k+\gamma_k)\hat{\bsx}^\top_k\bsE\Big(\bsv_k+\frac{1}{\gamma_k}\bsg_k^0\Big)\nonumber\\
			&\quad+\|\hat{\bsx}_k\|^2_{(1+\Delta_k)\frac{1}{2}\alpha_k^2(\beta_k\gamma_k+\gamma_k^2)\bsL + (1+\Delta_k)\frac{1}{2}\alpha_k^2(\beta_k+\gamma_k)^2\bsE} \nonumber\\
			&\quad +\Big\|\bsv_k+\frac{1}{\gamma_k}\bsg_{k}^0\Big\|^2_{\big(\alpha_k\frac{\gamma_k}{4}+\Delta_k(\frac{\beta_k+\gamma_k}{2\gamma_k}+\frac{\alpha_k\gamma_k}{4})\big)\bsF}\nonumber\\
			&\quad +(1+\Delta_k)(\epsilon_5\alpha_k+\epsilon_6\alpha_k^2) \ell ^2\|\bar{\bsg}_k^z\|^2 \nonumber\\
			&\quad +\frac{1}{2}\varepsilon_1(\Delta_k+\Delta_k^2)\|\bsg_{k+1}^0\|^2.
			\label{zerosg:v2k_B_k}
		\end{align}
		Since $\mathcal{B}_k$ and $\mathcal{C}_k$ are independent and $\mathcal{A}_k = \mathcal{B}_k \cup \mathcal{C}_k$, taking the expectation with respect to $\mathcal{A}_k$ on both sides of \eqref{zerosg:v2k_B_k} yields \eqref{zerosg:v2k}.

		\noindent {\bf (iii)}
			This step is to show the relation between $e_{3,k+1}$ and $e_{3,k}$.
		\begin{align}
		&e_{3,k+1} =\bsx_{k+1}^\top\bsE\bsF\Big(\bsv_{k+1}+\frac{1}{\gamma_{k+1}}\bsg_{k+1}^0\Big)\nonumber\\
		&=\bsx_{k+1}^\top\bsE\bsF\Big(\bsv_{k+1}+\frac{1}{\gamma_{k}}\bsg_{k+1}^0
		+\Big(\frac{1}{\gamma_{k+1}}-\frac{1}{\gamma_{k}}\Big)\bsg_{k+1}^0\Big)\nonumber\\
		&=\bsx_{k+1}^\top\bsE\bsF\Big(\bsv_{k+1}+\frac{1}{\gamma_{k}}\bsg_{k+1}^0\Big)
		-\Delta_k\bsx_{k+1}^\top\bsE\bsF\bsg_{k+1}^0\nonumber\\
		&\le\bsx_{k+1}^\top\bsE\bsF\Big(\bsv_{k+1}+\frac{1}{\gamma_{k}}\bsg_{k+1}^0\Big) \nonumber\\
		 &\quad +\frac{1}{2}\Delta_k(\|\bsx_{k+1}\|^2_{\bsE}+\|\bsg_{k+1}^0\|_{\bsF^2}^2).\label{zerosg:v3k-1}
		\end{align}

	For the first term on the right-hand side of \eqref{zerosg:v3k-1}, we have
	\begin{align}
		&\mathbb{E}_{\mathcal{B}_k}\Big[\bsx_{k+1}^\top\bsE\bsF\Big(\bsv_{k+1}+\frac{1}{\gamma_k}\bsg_{k+1}^0\Big)\Big]\nonumber\\
		&=\mathbb{E}_{\mathcal{B}_k}\Big[(\bsx_k-\alpha_k(\beta_k\bsL\hat{\bsx}_k+\gamma_k\bsv_k+\bsg_k^0+\bsg_k^z-\bsg_k^0))^\top
		\bsE\bsF\nonumber\\
		&\quad \times\Big(\bsv_k
		+\frac{1}{\gamma_k}\bsg_{k}^0+\alpha_k\gamma_k\bsL\hat{\bsx}_k+\frac{1}{\gamma_k}(\bsg_{k+1}^0-\bsg_{k}^0)\Big)\Big]\nonumber\\
		&=(\bsx_k^\top\bsE\bsF-\alpha_k(\beta_k+\alpha_k\gamma_k^2)\hat{\bsx}_k^\top\bsE)
		\Big(\bsv_k+\frac{1}{\gamma_k}\bsg_{k}^0\Big)\nonumber\\
		&\quad+\alpha_k\gamma_k\bsx_k^\top\bsE\hat{\bsx}_k-\|\hat{\bsx}_k\|^2_{\alpha_k^2\beta_k\gamma_k\bsL}\nonumber\\
		&\quad
		+\frac{1}{\gamma_k}(\bsx_k^\top\bsE\bsF-\alpha_k\beta_k\hat{\bsx}_k^\top\bsE)\mathbb{E}_{\mathcal{B}_k}[\bsg_{k+1}^0-\bsg_{k}^0]\nonumber\\
		&\quad-\alpha_k(\gamma_k\bsv_k+\bsg_{k}^0+\bsg_k^\mu-\bsg_k^0-\bar{\bsg}_k^\mu)^\top\bsF
		\Big(\bsv_k+\frac{1}{\gamma_k}\bsg_{k}^0\Big)\nonumber\\
		&\quad 
		-\alpha_k\Big(\bsv_k+\frac{1}{\gamma_k}\bsg_{k}^0\Big)^\top\bsE\bsF\mathbb{E}_{\mathcal{B}_k}[\bsg_{k+1}^0-\bsg_{k}^0]\nonumber\\
		&\quad-\alpha_k^2\gamma_k(\bsg_k^\mu-\bsg_k^0)^\top\bsE\hat{\bsx}_k \nonumber\\
		&\quad-\mathbb{E}_{\mathcal{B}_k}\Big[\frac{\alpha_k}{\gamma_k}(\bsg_k^z-\bsg_k^0)^\top\bsE\bsF(\bsg_{k+1}^0-\bsg_{k}^0)\Big] \nonumber\\
		&\le(\bsx_k^\top\bsE\bsF-\alpha_k\beta_k\hat{\bsx}_k^\top\bsE)
		\Big(\bsv_k+\frac{1}{\gamma_k}\bsg_{k}^0\Big) \nonumber\\
	    &\quad +\|\hat{\bsx}_k\|^2_{\frac{\alpha_k^2\gamma_k^2}{2}\bsE}
		+\frac{\alpha_k^2\gamma_k^2}{2}\Big\|\bsv_k+\frac{1}{\gamma_k}\bsg_{k}^0\Big\|^2 \nonumber\\
		&\quad +\|\bsx_k\|^2_{\frac{\alpha_k\gamma_k}{4}\bsE} + \|\hat{\bsx}_k\|^2_{\alpha_k\gamma_k(\bsE-\alpha_k\beta_k\bsL)}\nonumber\\
		&\quad+\|\bsx_k\|^2_{\frac{\alpha_k}{2}\bsE} + \mathbb{E}_{\mathcal{B}_k}\big[\|\bsg_{k+1}^0-\bsg_{k}^0\|^2_{\frac{1}{2\alpha_k\gamma_k^2}\bsF^2}\big]
		\nonumber\\
		&\quad +\|\hat{\bsx}_k\|^2_{\frac{\alpha^2_k\beta^2_k}{2}\bsE}
		+\frac{1}{2\gamma^2_k}\mathbb{E}_{\mathcal{B}_k}[\|\bsg_{k+1}^0-\bsg_{k}^0\|^2]\nonumber\\
		&\quad -\Big\|\bsv_k+\frac{1}{\gamma_k}\bsg_{k}^0\Big\|^2_{\alpha_k\gamma_k\bsF}\nonumber\\
		&\quad+\frac{\alpha_k}{4}\|\bsg_k^\mu-\bsg_k^0\|^2
			+\Big\|\bsv_k+\frac{1}{\gamma_k}\bsg_{k}^0\Big\|^2_{\alpha_k\bsF^2}\nonumber\\
				&\quad	+ \frac{\alpha_k}{\gamma_k}\rho_2^{-1}(L)\|\bar{\bsg}_k^\mu\|^2 
				+  \Big\|\bsv_k+\frac{1}{\gamma_k}\bsg_{k}^0\Big\|^2_{\frac{1}{4}\rho_2(L)\alpha_k\gamma_k \bsF^2} \nonumber\\
		&\quad+ \frac{\alpha_k^2\gamma_k^2}{2}\Big\|\bsv_k+\frac{1}{\gamma_k}\bsg_{k}^0\Big\|^2
		+\mathbb{E}_{\mathcal{B}_k}\big[\|\bsg_{k+1}^0-\bsg_{k}^0\|^2_{\frac{1}{2\gamma_k^2}\bsF^2}\big] \nonumber\\
		&\quad +\frac{\alpha_k^2}{2}\|\bsg_k^\mu-\bsg_k^0\|^2
		 +\|\hat{\bsx}_k\|^2_{\frac{\alpha_k^2\gamma_k^2}{2}\bsE} \nonumber\\
		&\quad +\frac{\alpha_k^2}{2}\mathbb{E}_{\mathcal{B}_k}\big[\|\bsg_k^z-\bsg_k^0\|^2\big]
		+\mathbb{E}_{\mathcal{B}_k}\big[\|\bsg_{k+1}^0-\bsg_{k}^0\|^2_{\frac{1}{2\gamma_k^2}\bsF^2}\big] \nonumber\\
		&=(\bsx_k^\top\bsE\bsF-\alpha_k\beta_k\hat{\bsx}_k^\top\bsE)\Big(\bsv_k+\frac{1}{\gamma_k}\bsg_{k}^0\Big) \nonumber\\
		&\quad +\|\bsx_k\|^2_{\frac{\alpha_k(\gamma_k+2)}{4}\bsE}
		 +\|\hat{\bsx}_k\|^2_{\alpha_k\gamma_k\bsE
			+\alpha_k^2\big((\frac{1}{2}\beta_k^2+\gamma_k^2)\bsE-\beta_k\gamma_k\bsL\big)} \nonumber\\
		&\quad -\Big\|\bsv_k+\frac{1}{\gamma_k}\bsg_{k}^0\Big\|^2_{\alpha_k(\gamma_k\bsF-\frac{1}{4}\rho_2(L)\gamma_k\bsF^2-\bsF^2)-\alpha_k^2\gamma_k^2{\bf I}_{np}}\nonumber\\
		&\quad +\mathbb{E}_{\mathcal{B}_k}\big[\|\bsg_{k+1}^0-\bsg_{k}^0\|^2_{\frac{1}{2\gamma_k^2}{\bf I}_{np}+\frac{1}{2\gamma_k^2}(2+\frac{1}{\alpha_k})\bsF^2}\big] \nonumber\\
		&\quad +\frac{\alpha_k}{4}(1+2\alpha_k)\|\bsg_k^\mu-\bsg_k^0\|^2
		 +\frac{\alpha_k^2}{2}\mathbb{E}_{\mathcal{B}_k}\big[\|\bsg_k^z-\bsg_k^0\|^2\big] \nonumber\\
		&\quad+\frac{\alpha_k}{\gamma_k}\rho_2^{-1}(L)\|\bar{\bsg}_k^\mu\|^2 \nonumber\\
		&\le(\bsx_k^\top\bsE\bsF-\alpha_k\beta_k\hat{\bsx}_k^\top\bsE)\Big(\bsv_k+\frac{1}{\gamma_k}\bsg_{k}^0\Big) \nonumber\\
		&\quad +\|\bsx_k\|^2_{\frac{\alpha_k(\gamma_k+2)}{4}\bsE}
		 +\|\hat{\bsx}_k\|^2_{\alpha_k\gamma_k\bsE
			+\alpha_k^2\big((\frac{1}{2}\beta_k^2+\gamma_k^2)\bsE-\beta_k\gamma_k\bsL\big)} \nonumber\\
		&\quad -\Big\|\bsv_k+\frac{1}{\gamma_k}\bsg_{k}^0\Big\|^2_{\alpha_k(\frac{3}{4}\gamma_k-\rho_2^{-1}(L))\bsF-\alpha_k^2\gamma_k^2\rho(L)\bsF}\nonumber\\
		&\quad + \big(\frac{1}{2\gamma_k^2}+\frac{1}{2\gamma_k^2}(2+\frac{1}{\alpha_k})\rho_2^{-2}(L)\big)\mathbb{E}_{\mathcal{B}_k}\big[\|\bsg_{k+1}^0-\bsg_{k}^0\|^2\big]\nonumber\\
		&\quad +\frac{\alpha_k}{4}(1+2\alpha_k)\|\bsg_k^\mu-\bsg_k^0\|^2
		 +\frac{\alpha_k^2}{2}\mathbb{E}_{\mathcal{B}_k}\big[\|\bsg_k^z-\bsg_k^0\|^2\big] \nonumber\\
		&\quad+\frac{\alpha_k}{\gamma_k}\rho_2^{-1}(L)\|\bar{\bsg}_k^\mu\|^2 
		\nonumber\\
		&\le\bsx_k^\top\bsE\bsF\Big(\bsv_k+\frac{1}{\gamma_k}\bsg_{k}^0\Big)
		-(1+\Delta_k)\alpha_k\beta_k\hat{\bsx}_k^\top\bsE\big(\bsv_k+\frac{1}{\gamma_k}\bsg_{k}^0\big)\nonumber\\
		&\quad +\Delta_k\alpha_k\beta_k\hat{\bsx}_k^\top\bsE\Big(\bsv_k+\frac{1}{\gamma_k}\bsg_{k}^0\Big)\nonumber\\
		&\quad+\|\bsx_k\|^2_{\alpha_k(\frac{\gamma_k+2}{4}+\frac{1}{2} \ell ^2)\bsE+3\alpha_k^2 \ell ^2\bsE}\nonumber\\
		&\quad +\|\hat{\bsx}_k\|^2_{\alpha_k\gamma_k\bsE
		 +\alpha_k^2\big((\frac{1}{2}\beta_k^2+\gamma_k^2)\bsE-\beta_k\gamma_k\bsL\big)} \nonumber\\
		&\quad -\Big\|\bsv_k+\frac{1}{\gamma_k}\bsg_{k}^0\Big\|^2_{\alpha_k(\frac{3}{4}\gamma_k-\rho_2^{-1}(L))\bsF-\alpha_k^2\gamma_k^2\rho(L)\bsF}\nonumber\\
		&\quad +(\alpha_k\epsilon_7+\alpha_k^2\epsilon_8) \ell ^2\mathbb{E}_{\mathcal{B}_k}\big[\|\bar{\bsg}_k^z\|^2\big]  +n \ell ^2\alpha_k(\frac{1}{2}+3\alpha_k)\mu_k^2 \nonumber\\
		&\quad + \alpha^2_k\mathbb{E}_{\mathcal{B}_k}[\|\bsg_k^z\|^2] +\frac{\alpha_k}{\gamma_k}\rho_2^{-1}(L)\|\bar{\bsg}_k^\mu\|^2 
		,\label{nonconvex:v3k}
	\end{align}
	where the first equality holds due to \eqref{nonconvex:kia-algo-dc-compact-x} and \eqref{nonconvex:kia-algo-dc-compact-v}; 
	the second equality holds since \eqref{nonconvex:KL-L-eq}, \eqref{nonconvex:lemma-eq3}, $\bsE=({\bf I}_{n}-\frac{1}{n}{\bf 1}_n{\bf 1}_n)\otimes{\bf I}_p$, \eqref{zerosg:rand-grad-esti1}, and that $x_{i,k}$ and $v_{i,k}$ are independent of $\mathcal{B}_k$; the first inequality holds due to the Cauchy--Schwarz inequality and $\rho(\bsE)=1$; the second inequality holds due to \eqref{nonconvex:lemma-eq5}; and the last  inequality holds due to \eqref{zerosg:rand-grad-esti8}, \eqref{zerosg:rand-grad-esti6}, and \eqref{zerosg:gg-rand-pd}.

	For the third term on the right-hand side of \eqref{nonconvex:v3k}, we have
		\begin{align}\label{zerosg:v3k-3}
			&\Delta_k\alpha_k\beta_k\hat{\bsx}_k^\top\bsE\Big(\bsv_k+\frac{1}{\gamma_k}\bsg_{k}^0\Big) \nonumber\\
			&\le\|\hat{\bsx}_k\|^2_{\frac{1}{2}\Delta_k\alpha_k\beta_k\bsE}
			+\Big\|\bsv_k+\frac{1}{\gamma_k}\bsg_k^0\Big\|^2_{\frac{1}{2}\rho(L)\Delta_k\alpha_k\beta_k\bsF}.
		\end{align}
		
		Then, from \eqref{zerosg:v3k-1}--\eqref{zerosg:v3k-3} and \eqref{nonconvex:lemma-eq5}, we have
		\begin{align}
			&\mathbb{E}_{\mathcal{B}_k}[e_{3,k+1}]\nonumber\\
			&\quad\le e_{3,k}
			-(1+\Delta_k)\alpha_k\beta_k\hat{\bsx}_k^\top\bsE\Big(\bsv_k+\frac{1}{\gamma_k}\bsg_{k}^0\Big)\nonumber\\
			&\quad+\|\bsx_k\|^2_{\alpha_k(\frac{\gamma_k+2}{4}+\frac{1}{2} \ell ^2)\bsE+3\alpha_k^2 \ell ^2\bsE}\nonumber\\
			&\quad +\|\hat{\bsx}_k\|^2_{\alpha_k\gamma_k\bsE
			 +\alpha_k^2\big((\frac{1}{2}\beta_k^2+\gamma_k^2)\bsE-\beta_k\gamma_k\bsL\big)+\frac{1}{2}\Delta_k\alpha_k\beta_k\bsE} \nonumber\\
			&\quad -\Big\|\bsv_k+\frac{1}{\gamma_k}\bsg_{k}^0\Big\|^2_{\big(\alpha_k(\frac{3}{4}\gamma_k-\rho_2^{-1}(L))-\alpha_k^2\gamma_k^2\rho(L) -\frac{1}{2}\rho(L)\Delta_k\alpha_k\beta_k\big)\bsF}\nonumber\\
			&\quad +(\alpha_k\epsilon_7+\alpha_k^2\epsilon_8) \ell ^2\mathbb{E}_{\mathcal{B}_k}\big[\|\bar{\bsg}_k^z\|^2\big]  +n \ell ^2\alpha_k(\frac{1}{2}+3\alpha_k)\mu_k^2 \nonumber\\
			&\quad + \alpha^2_k\mathbb{E}_{\mathcal{B}_k}[\|\bsg_k^z\|^2] +\frac{\alpha_k}{\gamma_k}\rho_2^{-1}(L)\|\bar{\bsg}_k^\mu\|^2 \nonumber\\
			&\quad+\frac{1}{2}\Delta_k\mathbb{E}_{\mathcal{B}_k}[2e_{1,k+1}+\rho_2^{-2}(L)\|\bsg_{k+1}^0\|^2].
		\label{zerosg:v3k_B_k}
		\end{align}
		Since $\mathcal{B}_k$ and $\mathcal{C}_k$ are independent and $\mathcal{A}_k = \mathcal{B}_k \cup \mathcal{C}_k$, taking the expectation with respect to $\mathcal{A}_k$ on both sides of \eqref{zerosg:v3k_B_k} yields \eqref{zerosg:v3k}.

		\noindent {\bf (iv)}
			This step is to show the relation between $e_{4,k+1}$ and $e_{4,k}$.	
		\begin{align}
			&\mathbb{E}_{\mathcal{B}_k}[e_{4,k+1}]
			=\mathbb{E}_{\mathcal{B}_k}[\tilde{f}(\bar{\bsx}_{k+1})-nf^*]\nonumber\\
			&=\mathbb{E}_{\mathcal{B}_k}[\tilde{f}(\bar{\bsx}_k)-nf^*+\tilde{f}(\bar{\bsx}_{k+1})
			-\tilde{f}(\bar{\bsx}_k)]\nonumber\\
			&\le\mathbb{E}_{\mathcal{B}_k}\Big[\tilde{f}(\bar{\bsx}_k)-nf^*
			-\alpha_k(\bar{\bsg}_{k}^z)^\top\bsg^0_k
			+\frac{1}{2}\alpha^2_k \ell \|\bar{\bsg}_{k}^z\|^2\Big]\nonumber\\
			&=e_{4,k}
			-\alpha_k(\bar{\bsg}_{k}^\mu)^\top\bsg^0_k
			+\frac{1}{2}\alpha^2_k \ell \mathbb{E}_{\mathcal{B}_k}[\|\bar{\bsg}_{k}^z\|^2]\nonumber\\
			&=e_{4,k}
			-\alpha_k(\bar{\bsg}_{k}^\mu)^\top\bar{\bsg}^0_k
			+\frac{1}{2}\alpha^2_k \ell \mathbb{E}_{\mathcal{B}_k}[\|\bar{\bsg}_{k}^z\|^2]\nonumber\\
			&=e_{4,k}
			-\frac{1}{2}\alpha_k(\bar{\bsg}_{k}^\mu)^\top(\bar{\bsg}^\mu_k+\bar{\bsg}^0_k-\bar{\bsg}^\mu_k)\nonumber\\
			&\quad-\frac{1}{2}\alpha_k(\bar{\bsg}^\mu_{k}-\bar{\bsg}^0_k+\bar{\bsg}^0_k)^\top\bar{\bsg}^0_k
			+\frac{1}{2}\alpha^2_k \ell \mathbb{E}_{\mathcal{B}_k}[\|\bar{\bsg}_{k}^z\|^2]\nonumber\\
			&\le e_{4,k}-\frac{1}{4}\alpha_k(\|\bar{\bsg}^\mu_{k}\|^2
			-\|\bar{\bsg}^0_k-\bar{\bsg}^\mu_k\|^2+\|\bar{\bsg}_{k}^0\|^2\nonumber\\
			&\quad-\|\bar{\bsg}^0_k-\bar{\bsg}^\mu_k\|^2)
			+\frac{1}{2}\alpha^2_k \ell \mathbb{E}_{\mathcal{B}_k}[\|\bar{\bsg}_{k}^z\|^2]\nonumber\\
			&= e_{4,k}-\frac{1}{4}\alpha_k\|\bar{\bsg}^\mu_{k}\|^2
			+\frac{1}{2}\alpha_k\|\bar{\bsg}^0_k-\bar{\bsg}^\mu_k\|^2\nonumber\\
			&\quad-\frac{1}{4}\alpha_k\|\bar{\bsg}_{k}^0\|^2
			+\frac{1}{2}\alpha^2_k \ell \mathbb{E}_{\mathcal{B}_k}[\|\bar{\bsg}_{k}^z\|^2]\nonumber\\
			&\le e_{4,k} - \frac{1}{4}\alpha_k\|\bar{\bsg}^\mu_{k}\|^2 + \|\bsx_k\|^2_{\alpha_k \ell ^2\bsE}\nonumber\\
			&\quad + n \ell ^2\alpha_k\mu^2_k-\frac{1}{4}\alpha_k\|\bar{\bsg}_{k}^0\|^2
			+\frac{1}{2}\alpha^2_k \ell \mathbb{E}_{\mathcal{B}_k}[\|\bar{\bsg}^z_{k}\|^2],\label{zerosg:v4k_B_k}
			\end{align}
		where the first inequality follows from the smoothness of $\tilde{f}$, together with \eqref{nonconvex:lemma:lipschitz} and \eqref{nonconvex:xbardynamic}; the third equality holds since \eqref{zerosg:rand-grad-esti1} and that $x_{i,k}$ and $v_{i,k}$ are independent of $\mathcal{B}_k$; the fourth equality holds due to $(\bar{\bsg}_{k}^\mu)^\top\bsg^0_k=(\bsg_{k}^\mu)^\top\bsH\bsg^0_k=(\bsg_{k}^\mu)^\top\bsH\bsH\bsg^0_k
		=(\bar{\bsg}_{k}^\mu)^\top\bar{\bsg}^0_k$; the second inequality holds due to the Cauchy--Schwarz inequality; and the last inequality holds due to \eqref{zerosg:rand-grad-esti9}.
		Then, taking the expectation with respect to $\mathcal{A}_k$ on both sides of \eqref{zerosg:v4k_B_k} yields \eqref{zerosg:v4k}.

		\noindent {\bf (v)}
		 	 This step is to show the relation between $e_{5,k+1}$ and $e_{5,k}$.	
		
		Denote $\mathcal{C}_r(\bsx)=\mathcal{C}(\bsx)/r$, then we have
		\begin{align}
			&\mathbb{E}_{\mathcal{C}_k}[\|\bsx_{k+1}-\bsy_{k+1}\|^2]\nonumber\\
			&=\mathbb{E}_{\mathcal{C}_k}[\|\bsx_{k+1}-\bsx_{k}+\bsx_{k}-\bsy_{k}-\omega\bsq_k\|^2]\nonumber\\
			&=\mathbb{E}_{\mathcal{C}_k}[\|\bsx_{k+1}-\bsx_{k}+(1-\omega r)(\bsx_{k}-\bsy_{k})\nonumber\\
			&\quad+\omega r(\bsx_{k}-\bsy_{k}-\mathcal{C}_r(\bsx_{k}-\bsy_{k}))\|^2]\nonumber\\
			&\le(1+c_1^{-1})\mathbb{E}_{\mathcal{C}_k}[\|\bsx_{k+1}-\bsx_{k}\|^2]\nonumber\\
			&\quad+(1+c_1)\mathbb{E}_{\mathcal{C}_k}[\|(1-\omega r)(\bsx_{k}-\bsy_{k})\nonumber\\
			&\quad+\omega r(\bsx_{k}-\bsy_{k}-\mathcal{C}_r(\bsx_{k}-\bsy_{k}))\|^2]\nonumber\\
			&\le(1+c_1^{-1})\mathbb{E}_{\mathcal{C}_k}[\|\bsx_{k+1}-\bsx_{k}\|^2]\nonumber\\
			&\quad+(1+c_1)(1-\omega r)\mathbb{E}_{\mathcal{C}_k}[\|\bsx_{k}-\bsy_{k}\|^2]\nonumber\\
			&\quad+(1+c_1)\omega r\mathbb{E}_{\mathcal{C}_k}[\|\bsx_{k}-\bsy_{k}-\mathcal{C}_r(\bsx_{k}-\bsy_{k})\|^2]\nonumber\\
			&\le(1+c_1^{-1})\mathbb{E}_{\mathcal{C}_k}[\|\bsx_{k+1}-\bsx_{k}\|^2]\nonumber\\
			&\quad+(1+c_1)(1-\omega r)\|\bsx_{k}-\bsy_{k}\|^2\nonumber\\
			&\quad+(1+c_1)\omega r(1-\delta)\|\bsx_{k}-\bsy_{k}\|^2\nonumber\\
			&=(1+c_1^{-1})\mathbb{E}_{\mathcal{C}_k}[\|\bsx_{k+1}-\bsx_{k}\|^2]\nonumber\\
			&\quad+(1-c_1-2c_1^2)\|\bsx_{k}-\bsy_{k}\|^2,
			\label{nonconvex:xminush_compress_B_k}
		\end{align}
		where the first and second equalities hold due to \eqref{nonconvex:kia-algo-dc-compact-a} and $\bsq_k=\mathcal{C}(\bsx_k-\bsy_k)$, respectively; the first inequality holds due to the Cauchy--Schwarz inequality and $c_1>0$; the second inequality holds due to the Cauchy--Schwarz inequality and $\omega r\in(0,1]$; 
		the last inequality holds due to  $\bsx_k$ and $\bsy_k$ being independent of $\mathcal{C}_k$, along with \eqref{nonconvex:ass:compression_equ_scaling};
		and the last equality holds due to $c_1=\omega r/2$.
		Then, taking the expectation with respect to $\mathcal{A}_k$ on both sides of \eqref{nonconvex:xminush_compress_B_k}, we have
		\begin{align}
			&\mathbb{E}_{\mathcal{A}_k}[\|\bsx_{k+1}-\bsy_{k+1}\|^2]\nonumber\\
			&\le(1+c_1^{-1})\mathbb{E}_{\mathcal{A}_k}[\|\bsx_{k+1}-\bsx_{k}\|^2]\nonumber\\
			&\quad+(1-c_1-2c_1^2)\|\bsx_{k}-\bsy_{k}\|^2.
			\label{nonconvex:xminush_compress}
		\end{align}

		For the first term on the right-hand side of \eqref{nonconvex:xminush_compress}, we have
		\begin{align}
			&\|\bsx_{k+1}-\bsx_k\|^2=\alpha_k^2\|\beta_k\bsL\hat{\bsx}_k+\gamma_k\bsv_k+\bsg_k^z\|^2\nonumber\\
			&=\alpha_k^2\|\beta_k\bsL(\hat{\bsx}_k-\bsx_k)+\beta_k\bsL\bsx_k+\gamma_k\bsv_k
			+\bsg_k^0+\bsg_k^z-\bsg_k^0\|^2\nonumber\\
			&\le4\alpha_k^2\big(\beta_k^2\|\hat{\bsx}_k-\bsx_k\|^2_{\bsL^2}+\beta_k^2\|\bsx_k\|^2_{\bsL^2}+\|\gamma_k\bsv_k+\bsg_k^0\|^2 \nonumber\\
			&\qquad  +\|\bsg_k^z-\bsg_k^0\|^2\big),
		\end{align}
		where the first equality holds due to \eqref{nonconvex:kia-algo-dc-compact-x}; and the first inequality holds due to the Cauchy--Schwarz inequality.
		Then, taking expectation with respect to $\mathcal{A}_k$, from the independence of $\mathcal{B}_k$ and $\mathcal{C}_k$, $\mathcal{A}_k = \mathcal{B}_k \cup \mathcal{C}_k$,  \eqref{nonconvex:KL-L-eq2}, \eqref{nonconvex:lemma-eq5} and \eqref{zerosg:rand-grad-esti6}, we have
		\begin{align}\label{nonconvex:xkoneminusx}
			&\mathbb{E}_{\mathcal{A}_k}[\|\bsx_{k+1}-\bsx_{k}\|^2]
			\le 4\rho^2(L)\alpha_k^2\beta_k^2 \mathbb{E}_{\mathcal{C}_k}[\|\hat{\bsx}_k-\bsx_k\|_\bsE^2] 
			\nonumber\\
            &\quad
			+\|\bsx_k\|^2_{4\alpha_k^2(\beta_k^2\rho^2(L)+4 \ell ^2)\bsE}+ \Big\|\bsv_k+\frac{1}{\gamma_k}\bsg_k^0\Big\|^2_{4\rho(L)\alpha_k^2\gamma_k^2\bsF} 
			\nonumber\\
			&\quad+ 16n \ell ^2\alpha_k^2\mu_k^2 
			+ 8\alpha_k^2\mathbb{E}_{\mathcal{B}_k}[\|\bsg_k^z\|^2]. 	
		\end{align}	
	
		Then, from \eqref{nonconvex:xminush_compress} and \eqref{nonconvex:xkoneminusx}, we have \eqref{nonconvex:xminush}.
	\end{proof}

	\subsubsection{Main proof} 
			To prove Lemma~\ref{Lemma:Lyap:fix}, it suffices to establish a more general version of Lemma~\ref{zerosg:lemma:sg:L1}.
			Before proceeding, 
			 we introduce the following additional notations for the subsequent analysis:		
			\begin{align*}
				&b_{1,k}=4p(1+\eta_1^2)\rho_2^{-2}(L) \ell ^4\frac{\alpha_k}{\gamma_k^2}\\
				&\quad +4p(1+\eta_1^2)\Big(2\rho_2^{-2}(L)+ (\epsilon_1+1)\rho_2^{-1}(L)+2\Big) \ell ^4\frac{\alpha_k^2}{\gamma_k^2}\\
				&\quad +8p(1+\eta_1^2)(\epsilon_1+1)^2\rho_2^{-1}(L) \ell ^4\frac{\alpha_k}{\gamma_k^3}\\
				&\quad+\Big(\frac{1}{2}+ \ell ^2\Big)\alpha_k\Delta_k + \Big(5+16p(1+\eta_1^2)\\
				&\qquad +8p(1+\eta_1^2)\varepsilon_8 \ell ^2\Big) \ell ^2\alpha_k^2\Delta_k\\
				&\quad+ 8p(1+\eta_1^2)\varepsilon_1 \ell ^4\alpha_k^2\Delta_k^2\\
				&\quad+ 4p(1+\eta_1^2)\Big((\epsilon_1+1)\rho_2^{-1}(L)+1\Big) \ell ^4\frac{\alpha_k^2\Delta_k}{\gamma_k^2}\\
				&\quad+ 8p(1+\eta_1^2)(\epsilon_1+1)^2\rho_2^{-1}(L) \ell ^4\frac{\alpha_k\Delta_k}{\gamma_k^3}
				,\\	
				&b_{2,k}=2 a_2 - \frac{1}{2}\Delta_k(\rho(L)\epsilon_1\epsilon_2+3\rho(L)\epsilon_2^2+\epsilon_1+\epsilon_2+1),\\
					&b^0_{2,k}=\alpha_k(\frac{1}{4}\gamma_k-\rho_2^{-1}(L))-\big(\frac{5}{2}+4(1+c_1^{-1})\big)\rho(L)\epsilon_2^2\\
					&\quad -\frac{1}{2}\Delta_k(\rho(L)\epsilon_1\epsilon_2+3\rho(L)\epsilon_2^2+\epsilon_1+\epsilon_2+1),\\
				&b_{3,k}=\varepsilon_8\frac{\Delta_k}{\alpha_k^2}
					+\varepsilon_1\frac{\Delta_k^2}{\alpha_k^2},\\
				&b_{4,k} = 2b^0_{4,k}+2\ell^2(\varepsilon_8\Delta_k+\varepsilon_1\Delta_k^2), \\
				&b^0_{4,k}=3 + 8(1+c_1^{-1}) +\frac{1}{2} \ell +\frac{\rho_2^{-2}(L) \ell ^2}{2\epsilon_2}\frac{1}{\gamma_k}\\
					&\quad +(\rho_2^{-2}(L)+\frac{2(\epsilon_1+1)^2+\epsilon_1\epsilon_2+\epsilon_2}{2\epsilon_2}\rho_2^{-1}(L)+1) \ell ^2\frac{1}{\gamma_k^2}\\
					&\quad +2\Delta_k + (\frac{2(\epsilon_1+1)^2+\epsilon_1\epsilon_2+\epsilon_2}{2\epsilon_2}\rho_2^{-1}(L)+\frac{1}{2}) \ell ^2\frac{\Delta_k}{\gamma_k^2}
					,\\
				&b_{5,k}= n \ell ^2\Big(\frac{5}{2}+\big(8+16(1+c_1^{-1})+\frac{1}{4}p^2b_{4,k}\big)\alpha_k\\
					&\quad +\Delta_k+5\alpha_k\Delta_k\Big),\\
				&\bsM_{1,k}=\frac{\beta_k}{2}\bsL-\frac{9\gamma_k}{4}\bsE-(1+\frac{5}{2} \ell ^2)\bsE,\\	
				&\bsM^0_{2,k}= 3\beta_k^2\bsL^2 - (\beta_k\gamma_k-\gamma_k^2)\bsL + \Big(4(1+c_1^{-1})\rho^2(L)\beta_k^2\\
					&\quad + 2\beta_k^2 +2\beta_k\gamma_k +3\gamma_k^2 + 8 \ell ^2+16(1+c_1^{-1}) \ell ^2\Big)\bsE,\\
				&\bsM_{2,k}=\bsM^0_{2,k}
					\\
					&\quad +4p(1+\eta_1^2)\big(6+ 16(1+c_1^{-1})+ \ell \big) \ell ^2\bsE,\\
				&\bsM_{3}=3\epsilon_1^2\epsilon_2^2\bsL^2 + (\epsilon_1\epsilon_2^2+\epsilon_2^2-\frac{1}{2}\epsilon_1\epsilon_2)\bsL \\
					&\quad + (\epsilon_1^2\epsilon_2^2+2\epsilon_1\epsilon_2^2+\epsilon_2^2+\epsilon_1\epsilon_2+\frac{1}{2})\bsE,\\	
				&\bsM_{4,k}=\frac{\rho(L)}{2}(\beta_k+2\gamma_k){\bf I}_{np} + 2\gamma_k\bsE,\\
				&\bsM_{5,k}=3\beta_k^2\bsL^2-(\beta_k\gamma_k-\gamma_k^2)\bsL + \Big(4(1+c_1^{-1})\rho^2(L)\beta_k^2\\
					&\quad + 2\beta_k^2 +2\beta_k\gamma_k +3\gamma_k^2 
					\Big)\bsE,\\
				&\bsM_{6}=\frac{\rho(L)}{2}(\epsilon_1\epsilon_2+2\epsilon_2){\bf I}_{np} + 3\epsilon_1^2\epsilon_2^2\bsL^2 +(\epsilon_1\epsilon_2^2+\epsilon_2^2)\bsL \\
					&\quad + (\epsilon_1^2\epsilon_2^2+2\epsilon_1\epsilon_2^2+\epsilon_2^2+\epsilon_1\epsilon_2)\bsE.	
			\end{align*}

		\begin{lemma} \label{zerosg:lemma:sg:L1}
			Suppose that Assumptions~\ref{ass:i:lower-bounded}--\ref{zerosg:ass:zeroth-variance} hold, $\{\gamma_k\}$ is non-decreasing, $\beta_k/\gamma_k=\epsilon_1$, $\alpha_k\gamma_k=\epsilon_2$, $\epsilon_1>\kappa_1$, $\epsilon_2>0$, and $\gamma_k\ge\varepsilon_{0}$. 
			Then  
				\begin{align}
					&\mathbb{E}_{\mathcal{A}_k}[\mathcal{L}_{1,k+1}]  \nonumber\\
					&\quad \le \mathcal{L}_{1,k} -\|\bsx_k\|^2_{(2 a_1-\varepsilon_5\Delta_k-b_{1,k})\bsE}
					-\Big\|\bsv_k+\frac{1}{\gamma_k}\bsg_{k}^0\Big\|^2_{b_{2,k}\bsF}\nonumber\\
					&\quad -\frac{1}{4}\alpha_k\|\bar{\bsg}^0_{k}\|^2 + 2\ell(b_{3,k}+4p(1+\eta_1^2)b_{4,k})\alpha_k^2 e_{4,k}\nonumber\\
					&\quad +2pn\sigma^2_1b_{4,k}\alpha_k^2 + n\check{\sigma}_2^2(b_{3,k}+4p(1+\eta_1^2)b_{4,k})\alpha_k^2 \nonumber\\
					&\quad +b_{5,k}\alpha_k\mu_k^2	- (c_2-\delta_0\varepsilon_2-\delta_0b_{6,k}\Delta_k)\|\bsx_{k}-\bsy_{k}\|^2.
										\label{zerosg:sgproof-vkLya:L1}
				\end{align}
		\end{lemma}

		\begin{proof}

		We start to show the relation between $\mathcal{L}_{1,k+1}$ and $\mathcal{L}_{1,k}$.
		We have
		\begin{align}
			&\mathbb{E}_{\mathcal{A}}[\mathcal{L}_{1,k+1}]\nonumber\\
			&\le  \mathcal{L}_{1,k}+\frac{1}{2}\Delta_k\|\bsx_k\|^2_{\bsE} \nonumber\\
			&\quad - (1+\Delta_k)\|\bsx_k\|^2_{\frac{\alpha_k\beta_k}{2}\bsL-\frac{\alpha_k}{2}\bsE-\alpha_k(1+5\alpha_k) \ell ^2\bsE}\nonumber\\
			&\quad+ (1+\Delta_k)\mathbb{E}_{\mathcal{C}_k}\Big[\|\hat{\bsx}_k-\bsx_k+\bsx_k\|^2_{\frac{3}{2}\alpha_k^2\beta_k^2\bsL^2}\Big] \nonumber\\
			&\quad + (1+\Delta_k)n \ell ^2\alpha_k(1+5\alpha_k)\mu^2_k\nonumber\\
			&\quad + (1+\Delta_k)\frac{\alpha_k}{2}(\beta_k+2\gamma_k)\rho(L)\mathbb{E}_{\mathcal{C}_k}[\|\bsx_k-\hat{\bsx}_k\|^2] \nonumber\\
			&\quad + 2(1+\Delta_k)\alpha^2_k\mathbb{E}_{\mathcal{B}_k}[\|\bsg_k^z\|^2]\nonumber\\
			&\quad + (1+\Delta_k)\Big\|\bsv_k+\frac{1}{\gamma_k}\bsg_k^0\Big\|^2_{\frac{6\alpha_k^2\gamma_k^2\rho(L)+\alpha_k\gamma_k}{4}\bsF} 
			\nonumber\\ 
			&\quad +\frac{1}{2}\varepsilon_1(\Delta_k+\Delta_k^2)\mathbb{E}_{\mathcal{A}_k}[\|\bsg_{k+1}^0\|^2] \nonumber\\
			&\quad+\mathbb{E}_{\mathcal{C}_k}\Big[\|\hat{\bsx}_k-\bsx_k+\bsx_k\|^2_{(1+\Delta_k)\frac{1}{2}\alpha_k^2(\beta_k\gamma_k+\gamma_k^2)\bsL }\Big] \nonumber\\
			&\quad+\mathbb{E}_{\mathcal{C}_k}\Big[\|\hat{\bsx}_k-\bsx_k+\bsx_k\|^2_{(1+\Delta_k)\frac{1}{2}\alpha_k^2(\beta_k+\gamma_k)^2\bsE}\Big] \nonumber\\
			&\quad +\Big\|\bsv_k+\frac{1}{\gamma_k}\bsg_{k}^0\Big\|^2_{\big(\alpha_k\frac{\gamma_k}{4}+\Delta_k(\frac{\beta_k+\gamma_k}{2\gamma_k}+\frac{\alpha_k\gamma_k}{4})\big)\bsF}\nonumber\\
			&\quad +(1+\Delta_k)(\epsilon_5\alpha_k+\epsilon_6\alpha_k^2) \ell ^2\mathbb{E}_{\mathcal{B}_k}\big[\|\bar{\bsg}_k^z\|^2\big] 
			\nonumber\\ 
			&\quad+\|\bsx_k\|^2_{\alpha_k(\frac{\gamma_k+2}{4}+\frac{1}{2} \ell ^2)\bsE+3\alpha_k^2 \ell ^2\bsE}\nonumber\\
			&\quad +\mathbb{E}_{\mathcal{C}_k}\Big[\|\hat{\bsx}_k-\bsx_k+\bsx_k\|^2_{\alpha_k\gamma_k\bsE
				+\alpha_k^2\big((\frac{1}{2}\beta_k^2+\gamma_k^2)\bsE-\beta_k\gamma_k\bsL\big)}\Big] \nonumber\\
			&\quad +\mathbb{E}_{\mathcal{C}_k}\Big[\|\hat{\bsx}_k-\bsx_k+\bsx_k\|^2_{\frac{1}{2}\Delta_k\alpha_k\beta_k\bsE} \Big]\nonumber\\
			&\quad -\Big\|\bsv_k+\frac{1}{\gamma_k}\bsg_{k}^0\Big\|^2_{\big(\alpha_k(\frac{3}{4}\gamma_k-\rho_2^{-1}(L))-\alpha_k^2\gamma_k^2\rho(L) -\frac{1}{2}\rho(L)\Delta_k\alpha_k\beta_k\big)\bsF}\nonumber\\
			&\quad +(\alpha_k\epsilon_7+\alpha_k^2\epsilon_8) \ell ^2\mathbb{E}_{\mathcal{B}_k}\big[\|\bar{\bsg}_k^z\|^2\big]  +n \ell ^2\alpha_k(\frac{1}{2}+3\alpha_k)\mu_k^2 \nonumber\\
			&\quad + \alpha^2_k\mathbb{E}_{\mathcal{B}_k}[\|\bsg_k^z\|^2] +\frac{\alpha_k}{\gamma_k}\rho_2^{-1}(L)\|\bar{\bsg}_k^\mu\|^2 \nonumber\\
			&\quad+\frac{1}{2}\rho_2^{-2}(L)\Delta_k\mathbb{E}_{\mathcal{A}_k}[\|\bsg_{k+1}^0\|^2]
			\nonumber\\ 
			&\quad- \frac{1}{4}\alpha_k\|\bar{\bsg}^\mu_{k}\|^2 + \|\bsx_k\|^2_{\alpha_k \ell ^2\bsE}\nonumber\\
			&\quad + n \ell ^2\alpha_k\mu^2_k-\frac{1}{4}\alpha_k\|\bar{\bsg}_{k}^0\|^2
			+\frac{1}{2}\alpha^2_k \ell \mathbb{E}_{\mathcal{B}_k}[\|\bar{\bsg}^z_{k}\|^2]
			\nonumber\\ 
			&\quad - c_2\|\bsx_{k}-\bsy_{k}\|^2 \nonumber\\
			&\quad +4(1+c_1^{-1})\rho^2(L)\alpha_k^2\beta_k^2 \mathbb{E}_{\mathcal{C}_k}[\|\bsx_k-\hat{\bsx}_k\|_\bsE^2] \nonumber\\
			&\quad + \|\bsx_k\|^2_{4(1+c_1^{-1})\alpha^2_k(\beta_k^2\rho^2(L)+4 \ell ^2)\bsE}\nonumber\\
			&\quad + 16(1+c_1^{-1})n \ell ^2\alpha_k^2\mu_k^2 + 8(1+c_1^{-1})\alpha_k^2\mathbb{E}_{\mathcal{B}_k}[\|\bsg_k^z\|^2]\nonumber\\
			&\quad + \Big\|\bsv_k+\frac{1}{\gamma_k}\bsg_k^0\Big\|^2_{4(1+c_1^{-1})\rho(L)\alpha_k^2\gamma_k^2\bsF}
			\nonumber\\ 
			&\le \mathcal{L}_{1,k} - \|\bsx_k\|^2_{\alpha_k\bsM_{1,k}-\alpha_k^2\bsM^0_{2,k}-\Delta_k\bsM_{3}}\nonumber\\
			&\quad + \|\bsx_k\|^2_{((\frac{1}{2}+ \ell ^2)\alpha_k\Delta_k+5 \ell ^2\alpha_k^2\Delta_k)\bsE} \nonumber\\
			&\quad + \mathbb{E}_{\mathcal{C}_k}\big[\|\bsx_k-\hat{\bsx}_k\|^2_{\alpha_k\bsM_{4,k}+\alpha_k^2\bsM_{5,k}
			 + \Delta_k\bsM_6}\big] \nonumber\\
			&\quad - \Big\|\bsv_k+\frac{1}{\gamma_k}\bsg_{k}^0\Big\|^2_{b^0_{2,k}\bsF}
			  -\frac{1}{4}\alpha_k\|\bar{\bsg}_{k}^0\|^2+b^0_{4,k}\alpha^2_k 
			  \mathbb{E}_{\mathcal{B}_k}[\|\bsg_k^z\|^2] \nonumber\\
			&\quad + \frac{1}{2}(\varepsilon_8\Delta_k+\varepsilon_1\Delta_k^2)
			\mathbb{E}_{\mathcal{A}_k}[\|\bsg_{k+1}^0\|^2] - c_2\|\bsx_{k}-\bsy_{k}\|^2 \nonumber\\
			&\quad + n \ell ^2\Big(\frac{5}{2}+\big(8+16(1+c_1^{-1})\big)\alpha_k+\Delta_k+5\alpha_k\Delta_k\Big)\alpha_k\mu_k^2 \nonumber\\
				&\quad -\alpha_k\Big(\frac{1}{4}-\frac{\rho_2^{-1}(L)}{\gamma_k}\Big) \|\bar{\bsg}_{k}^{\mu}\|^2 \nonumber\\			
			&\le  \mathcal{L}_{1,k} - \|\bsx_k\|^2_{\alpha_k\bsM_{1,k}-\alpha_k^2\bsM_{2,k}-\Delta_k\bsM_{3}-b_{1,k}\bsE} \nonumber\\
			&\quad - \Big\|\bsv_k+\frac{1}{\gamma_k}\bsg_{k}^0\Big\|^2_{b^0_{2,k}\bsF} -\frac{1}{4}\alpha_k\|\bar{\bsg}^0_{k}\|^2 \nonumber\\
			&\quad  + 2\ell(b_{3,k}+4p(1+\eta_1^2)b_{4,k})\alpha_k^2 e_{4,k}\nonumber\\
			&\quad + 2pn\sigma^2_1b_{4,k}\alpha_k^2+n\check{\sigma}_2^2(b_{3,k}+4p(1+\eta_1^2)b_{4,k})\alpha_k^2 \nonumber\\
			&\quad +b_{5,k}\alpha_k\mu_k^2 - (c_2-\delta_0\varepsilon_2-\delta_0b_{6,k}\Delta_k)\|\bsx_{k}-\bsy_{k}\|^2 \nonumber\\		
			&\quad -\alpha_k\Big(\frac{1}{4}-\frac{\rho_2^{-1}(L)}{\gamma_k}\Big) \|\bar{\bsg}_{k}^{\mu}\|^2   \label{nonconvex:vkLya_xhat}
		\end{align}
		where the first inequality holds due to \eqref{nonconvex:v1k}--\eqref{nonconvex:xminush};
		  the second inequality holds due to $\|\hat{\bsx}_k\|^2=\|\hat{\bsx}_k-\bsx_k+\bsx_k\|^2\le2\|\hat{\bsx}_k-\bsx_k\|^2+2\|\bsx_k\|^2$, the independence between $\bsx_k$ and $\mathcal{C}_k$, $\|\bar{\bsg}^z_{k}\|^2\le\|\bsg^z_{k}\|^2$, $\beta_k=\epsilon_1\gamma_k$ and  $\alpha_k=\epsilon_2/\gamma_k$; 
		   and the last inequality hods due to \eqref{zerosg-a5:rand-grad-esti2}, \eqref{zerosg-a5:vkLya-2},
			 \eqref{nonconvex:ass:compression_equ_scaling} , the independence of $\bsx_k$ and $\bsy_k$ from $\mathcal{C}_k$, \eqref{nonconvex:KL-L-eq2}, $\rho(\bsE)=1$, $\gamma_k\ge\varepsilon_{0}\ge1$ and $\alpha_k\le\epsilon_2$.

		Next, we scale and bound some coefficients to prove the Lemma~\ref{zerosg:lemma:sg:L1}.

		From \eqref{nonconvex:KL-L-eq2}, $\beta_k=\epsilon_1\gamma_k$, $\epsilon_1>\kappa_1\ge13/(2\rho_2(L))$, $\gamma_k\ge\varepsilon_{0}\ge1+ 5\ell ^2/2$, and $\alpha_k=\epsilon_2/\gamma_k$, we have
		\begin{align}\label{zerosg:m1-rand-pd}
		\alpha_k\bsM_{1,k}\ge\varepsilon_3\epsilon_2\bsE.
		\end{align}
		From \eqref{nonconvex:KL-L-eq2}, $\beta_k=\epsilon_1\gamma_k$, $\gamma_k\ge\varepsilon_{0}\ge\big(8+16(1+c_1^{-1})+4p(1+\eta_1^2)(6+ 16(1+c_1^{-1})+ \ell )\big)^\frac{1}{2} \ell $, and $\alpha_k=\epsilon_2/\gamma_k$, we have
		\begin{align}\label{zerosg:m2-rand-pd}
		\alpha_k^2\bsM_{2,k}\le\varepsilon_4\epsilon_2^2\bsE.
		\end{align}
		From \eqref{nonconvex:KL-L-eq2}, $\beta_k=\epsilon_1\gamma_k$,  and $\alpha_k=\epsilon_2/\gamma_k$, we have
		\begin{align}\label{zerosg:m3-rand-pd}
		\bsM_{3}\le\varepsilon_5\bsE.
		\end{align}
		From $\gamma_k\ge\varepsilon_{0}\ge 8\rho_2^{-1}(L)$ and $\alpha_k=\epsilon_2/\gamma_k$, we have
		\begin{align}\label{zerosg:vkLya-b1}
		b^0_{2,k}\ge&b_{2,k}.
		\end{align}
		From $\gamma_k\ge\varepsilon_{0}\ge 8\rho_2^{-1}(L)$, we have
		\begin{align}\label{zerosg:vkLya-b2}
			\frac{1}{4}-\frac{\rho_2^{-1}(L)}{\gamma_k}  \ge 0
		\end{align}

		Finally, from \eqref{nonconvex:vkLya_xhat}--\eqref{zerosg:vkLya-b2}, we know that \eqref{zerosg:sgproof-vkLya:L1} holds.
	\end{proof}

We are now ready to prove Lemma~\ref{Lemma:Lyap:fix}.
Noting $\gamma_k = \gamma \ge \tilde{\kappa}_0(\epsilon_1,\epsilon_2) \ge \varepsilon_0$, all the conditions in Lemma~\ref{zerosg:lemma:sg:L1} are satisfied.
Substituting $\beta_k=\beta=\epsilon_1\gamma$, $\gamma_k=\gamma$, $\alpha_k=\alpha=\epsilon_2/\gamma$, and $\Delta_k=0$ into 
	\eqref{zerosg:sgproof-vkLya:L1}, one has
	\begin{align}
		&\mathbb{E}_{\mathcal{A}_k}[\mathcal{L}_{1,k+1}]\
			\le  \mathcal{L}_{1,k}-\|\bsx_k\|^2_{(2 a_1-\tilde{b}_{1})\bsE}
		-\|\bsv_k+\frac{1}{\gamma_k}\bsg_{k}^0\|^2_{2a_2\bsF}\nonumber\\
		&\quad-\frac{1}{4}\alpha\|\bar{\bsg}^0_{k}\|^2  + 8p(1+\eta_1^2)\ell \tilde{b}_{4}\alpha^2e_{4,k}\nonumber\\
		&\quad+2pn(\sigma^2_1+2(1+\eta_1^2)\check{\sigma}_2^2)\tilde{b}_{4}\alpha^2+\tilde{b}_{5}\alpha\mu_k^2,\nonumber\\
		&\quad -2 a_3\|\bsx_{k}-\bsy_{k}\|^2, \label{zerosg:sgproof-vkLyaT}
	\end{align}
	where
	\begin{align*}
		&\tilde{b}_{1}=4p(1+\eta_1^2)\rho_2^{-2}(L) \ell ^4\frac{\alpha}{\gamma^2}\\
			&\quad +4p(1+\eta_1^2)\Big(2\rho_2^{-2}(L) + (\epsilon_1+1)\rho_2^{-1}(L)+2\Big) \ell ^4\frac{\alpha^2}{\gamma^2}\\
			&\quad +8p(1+\eta_1^2)(\epsilon_1+1)^2\rho_2^{-1}(L) \ell ^4\frac{\alpha}{\gamma^3}
			,\\	
		&\tilde{b}_{4}=6 + 16(1+c_1^{-1})+  \ell  + \frac{\rho_2^{-2}(L) \ell ^2}{\epsilon_2}\frac{1}{\gamma}\\
			&\quad +(2\rho_2^{-2}(L)+\frac{2(\epsilon_1+1)^2+\epsilon_1\epsilon_2+\epsilon_2}{\epsilon_2}\rho_2^{-1}(L)+2) \ell ^2\frac{1}{\gamma^2},\\
		&\tilde{b}_{5}= n \ell ^2\Big(\frac{5}{2}+\big(8+16(1+c_1^{-1})+\frac{1}{4}p^2\tilde{b}_{4}\big)\alpha\Big).\\
	\end{align*}


	From $\gamma\ge\tilde{\kappa}_0(\epsilon_1,\epsilon_2)\ge\max\{\varepsilon_0, 1, \Big(\frac{p(1+\eta_1^2)\tilde{\varepsilon}_9}{ a_1}\Big)^{\frac{1}{3}}\}$, we have
	\begin{align}
		2 a_1-\tilde{b}_{1}\ge2 a_1-\frac{p(1+\eta_1^2)\tilde{\varepsilon}_9}{\gamma^3}\ge a_1.
	\end{align}


	From $\gamma\ge1$, we have
	\begin{align} \label{tilde_b4}
	\tilde{b}_4\le\tilde{\varepsilon}_{12}.
	\end{align}

	From $\gamma\ge\tilde{\kappa}_0(\epsilon_1,\epsilon_2)\ge~p\epsilon_2\tilde{\varepsilon}_{12}$, we have
		\begin{align}\label{zerosg:b3kb4keta:1}
		\tilde{b}_4\alpha\le \frac{\epsilon_2\tilde{\varepsilon}_{12}}{\gamma}
		\le\frac{1}{p}.
		\end{align}

		From \eqref{zerosg:b3kb4keta:1} and $\gamma\ge1$, we have 
	\begin{align}
		\tilde{b}_5\le pn\tilde{a}_5.
	\end{align}

	From $\epsilon_1>\kappa_1\ge13/(2\rho_2(L))$, we have
	$\varepsilon_3>0$.
	From $\rho(L)\epsilon_1\ge\rho_2(L)\epsilon_1\ge13/2$, we have $\varepsilon_4\ge7\rho^2(L)\epsilon_1^2-\rho(L)\epsilon_1>0$.
	From $\varepsilon_3>0$, $\varepsilon_4>0$, and $\epsilon_2 < \kappa_2(\epsilon_1) \le\min\{\frac{\varepsilon_3}{\varepsilon_4},~\frac{\rho^{-1}(L)}{20+32(1+c_1^{-1})}\}$, we have
	\begin{align}
	a_1>0~\text{and}~
	a_2>0.\label{zerosg:kappa4-6}
	\end{align}

	From $\epsilon_2 < \kappa_2(\epsilon_1) \le\frac{\sqrt{\varepsilon_{13}^2+4\varepsilon_{14}c_2}-\varepsilon_{13}}{2\varepsilon_{14}}$, we have
	\begin{align}\label{zerosg:c1}
		a_3=\frac{1}{2}(c_2-\delta_0\varepsilon_2)=\frac{1}{2}(c_2-\varepsilon_{13}\epsilon_2-\varepsilon_{14}\epsilon_2^2)>0.
	\end{align}

	From \eqref{tilde_b4}, we have
	\begin{align} \label{tilde_a8}
		8p(1+\eta_1^2)\ell \tilde{b}_{4}\alpha^2e_{4,k} 
		\le p \tilde{a}_8 \alpha^2 e_{4,k}.
	\end{align}

	Finally, from \eqref{zerosg:sgproof-vkLyaT}--\eqref{tilde_a8}, we have \eqref{zerosg:sgproof-vkLya2T}.


\subsection{Proof of Lemma \ref{zerosg:lemma:sg2-T:23}} \label{appendix:proof:lemma3}


\noindent {\bf (i)}
	Under the setting of Lemma~\ref{zerosg:lemma:sg2-T:23}, inequalities \eqref{nonconvex:v1k}--\eqref{nonconvex:xminush} hold.


			
		From \eqref{zerosg:v4k}, \eqref{zerosg:rand-grad-esti5} and \eqref{zerosg-a5:rand-grad-esti2}, 
		\begin{align}
		&\mathbb{E}_{\mathcal{A}_k}[e_{4,k+1}] \le e_{4,k}-\frac{1}{4}\alpha_k\|\bar{\bsg}^\mu_{k}\|^2
		+\|\bsx_k\|^2_{\alpha_k \ell ^2\bsE}+n \ell ^2\alpha_k\mu^2_k\nonumber\\
		&\quad-\frac{1}{4}\alpha_k\|\bar{\bsg}_{k}^0\|^2
		+\frac{1}{2}\alpha_k^2 \ell \Big(
			 \frac{16p(1+\eta_1^2)\ell }{n}e_{4,k} +\|\bar{\bsg}^\mu_{k}\|^2\nonumber\\
			&\quad 	+\frac{8p(1+\eta_1^2)\ell ^2}{n} \|\bsx_{k}\|^2_{\bsE}  +4p\sigma^2_1+8p(1+\eta_1^2)\check{\sigma}_2^2 \nonumber\\
			&\quad + \frac{1}{2}p^2 \ell ^2\mu_k^2
			\Big).\label{zerosg:v4kspeed-diminishing}
		\end{align}
		Rescaling $\alpha_k=\epsilon_2/\gamma_k$, and $\gamma_k\ge \varepsilon_0\ge4p(1+\eta_1^2)\epsilon_2 \ell $, we obtain
		\begin{subequations}
			\begin{align}
				&\frac{4p(1+\eta_1^2)}{n}\alpha_k^2 \ell ^3
					\le\alpha_k \ell ^2,\label{zerosg:v4kspeed-diminishing-1.2}\\
				&\frac{1}{2}\alpha_k^2 \ell \le\frac{1}{4}\alpha_k,\label{zerosg:v4kspeed-diminishing-1.3}\\
				&\frac{1}{4}p^2\alpha_k^2 \ell ^3\le p \ell ^2\alpha_k.\label{zerosg:v4kspeed-diminishing-1.4}
			\end{align}
		\end{subequations}
		Then \eqref{zerosg:v4kspeed} follows from \eqref{zerosg:v4kspeed-diminishing}--\eqref{zerosg:v4kspeed-diminishing-1.4}.

\noindent {\bf (ii)}
		 We establish the relation between $\mathcal{L}_{2,k+1}$ and $\mathcal{L}_{2,k}$ by following a procedure similar to the proof of Lemma~\ref{Lemma:Lyap:fix}.
		Denote $b_{5,k}^{\prime} = b_{5,k} + 2n \ell^2 a_6 \epsilon_2/\varepsilon_0$ and $\bsM_{2,k}^{\prime} = \bsM_{2,k} + 2 a_6 \ell^2 \bsE$.
		Similar to the derivation of~\eqref{zerosg:sgproof-vkLya:L1}, 
				we obtain
		\begin{align}
			&\mathbb{E}_{\mathcal{A}_k}[\mathcal{L}_{2,k+1}]   
			  \le  \mathcal{L}_{2,k} - \|\bsx_k\|^2_{\alpha_k\bsM_{1,k}-\alpha_k^2\bsM_{2,k}-\Delta_k\bsM_{3}-b_{1,k}\bsE} \nonumber\\
			&\quad - \Big\|\bsv_k+\frac{1}{\gamma_k}\bsg_{k}^0\Big\|^2_{b^0_{2,k}\bsF} 
				+ \frac{\alpha_k}{\gamma_k}\rho_2^{-1}(L)\|\bar{\bsg}_k^\mu - \bar{\bsg}_k^0 + \bar{\bsg}_k^0\|^2 \nonumber\\
			&\quad  + 2\ell(b_{3,k}+4p(1+\eta_1^2)b_{4,k})\alpha_k^2 e_{4,k}\nonumber\\
			&\quad + 2pn\sigma^2_1b_{4,k}\alpha_k^2+n\check{\sigma}_2^2(b_{3,k}+4p(1+\eta_1^2)b_{4,k})\alpha_k^2 \nonumber\\
			&\quad +b_{5,k}\alpha_k\mu_k^2 - (c_2-\delta_0\varepsilon_2-\delta_0b_{6,k}\Delta_k)\|\bsx_{k}-\bsy_{k}\|^2 \nonumber\\
			&\le  \mathcal{L}_{2,k} - \|\bsx_k\|^2_{\alpha_k\bsM_{1,k}-\alpha_k^2\bsM_{2,k}^{\prime}-\Delta_k\bsM_{3}-b_{1,k}\bsE} \nonumber\\
			&\quad - \Big\|\bsv_k+\frac{1}{\gamma_k}\bsg_{k}^0\Big\|^2_{b^0_{2,k}\bsF} 
				+ b_{5,k}\alpha_k\mu_k^2 + 2n \ell^2 a_6 \alpha_k^2 \mu_k^2 \nonumber\\
			&\quad + a_6 \alpha_k^2 \|\bar{\bsg}^0_{k}\|^2  + 2\ell(b_{3,k}+4p(1+\eta_1^2)b_{4,k})\alpha_k^2 e_{4,k} \nonumber\\
			&\quad + 2pn\sigma^2_1b_{4,k}\alpha_k^2 + n\check{\sigma}_2^2\big(b_{3,k}+4p(1+\eta_1^2)b_{4,k}\big)\alpha_k^2 \nonumber\\
			&\quad - (c_2-\delta_0\varepsilon_2-\delta_0b_{6,k}\Delta_k)\|\bsx_{k}-\bsy_{k}\|^2 \nonumber\\
			&\le  \mathcal{L}_{2,k} -\|\bsx_k\|^2_{(2 a_1-\varepsilon_5\Delta_k-b_{1,k})\bsE}
				-\Big\|\bsv_k+\frac{1}{\gamma_k}\bsg_{k}^0\Big\|^2_{b_{2,k}\bsF}\nonumber\\
			&\quad + a_6 \alpha_k^2 \|\bar{\bsg}^0_{k}\|^2  + 2\ell(b_{3,k}+4p(1+\eta_1^2)b_{4,k})\alpha_k^2 e_{4,k} \nonumber\\
			&\quad + 2pn\sigma^2_1b_{4,k}\alpha_k^2 + n\check{\sigma}_2^2\big(b_{3,k}+4p(1+\eta_1^2)b_{4,k}\big)\alpha_k^2 \nonumber\\
			&\quad 	+ b_{5,k}^{\prime}\alpha_k\mu_k^2 
			 - (c_2-\delta_0\varepsilon_2-\delta_0b_{6,k}\Delta_k)\|\bsx_{k}-\bsy_{k}\|^2, \label{L2:0}
		\end{align}
		where the first inequality holds in the same manner as~\eqref{nonconvex:vkLya_xhat};
		the second inequality hold due to $\|\bar{\bsg}_k^\mu - \bar{\bsg}_k^0 + \bar{\bsg}_k^0\|^2 \le 2\|  \bar{\bsg}_k^0\|^2 +2 \|\bar{\bsg}_k^\mu - \bar{\bsg}_k^0\| $, \eqref{zerosg:rand-grad-esti9}, and $\gamma_k=\epsilon_2/\alpha_k$;
		the last inequality holds in the similar manner as~\eqref{zerosg:m1-rand-pd}--\eqref{zerosg:vkLya-b2}, with \(\alpha_k = \epsilon_2/\gamma_k \le \epsilon_2/\varepsilon_0\).
	Then, following similar steps to those in~\eqref{zerosg:sgproof-vkLyaT}--\eqref{tilde_a8}, we obtain~\eqref{zerosg:sgproof-vkLya2T-bounded} from~\eqref{L2:0}.

\subsection{Proof of Theorem~\ref{Thm:PL:know}} \label{Proof:thm:know}
The key distinction from the previous proofs is the setting of time-varying algorithm parameters, which makes the analysis more challenging.
We therefore provide time-varying counterparts of the lemmas used for Theorems~\ref{Thm:nonconvex} and~\ref{Thm:PL:unknow}, namely Lemmas~\ref{zerosg:lemma:sg2}--\ref{L3_bound:time-varying1}.
Before proceeding, in addition to the notations defined in Appendices~\ref{appendix:constant:Thm} and \ref{appendix:iterative:Lemma}, 
 we introduce the following constants for the subsequent analysis:
\begin{align*}
&\varepsilon_9
=4\big(\rho_2^{-2}(L)+4(\epsilon_1+1)^2\rho_2^{-1}(L)\big) \ell ^4\epsilon_2 \\
&\quad+\frac{(1+2 \ell ^2)\epsilon_2}{2p(1+\eta_1^2)}
+\big(\frac{5}{p(1+\eta_1^2)}+16\big) \ell ^2\epsilon_2^2\\
&\quad + 4\big(2\rho_2^{-2}(L)+2(\epsilon_1+1)\rho_2^{-1}(L)+2\varepsilon_8+2\varepsilon_1+3\big) \ell ^4\epsilon_2^2,\\
&\varepsilon_{15}=
\frac{4\epsilon_2^2(a_9\bar{L} + a_7+ \ell ^2\kappa_{\mu}^2)}{\epsilon_4^2(\frac{\nu\epsilon_2}{2\epsilon_4}-1)},\\
& a_5^{\prime}= \ell ^2\Big(\frac{(29+16c_1^{-1})\epsilon_2+\frac{7}{2}}{p} + \frac{1}{4}  +2a_6\frac{\epsilon_2}{\varepsilon_0p} \Big),\\
&d_1=\frac{1}{\varepsilon_6}\min\{a_1, a_2, a_3\}.
\end{align*}




\begin{lemma}\label{zerosg:lemma:sg2}
	Suppose  Assumptions~\ref{ass:i:lower-bounded}--\ref{zerosg:ass:zeroth-variance} hold,  $\alpha_k=\epsilon_2/\gamma_k$,  $\beta_k=\epsilon_1\gamma_k$, and $\gamma_k=\epsilon_4(k+m)$, 
	where $\epsilon_1>\kappa_1$, $\epsilon_2\in(0,\kappa_2(\epsilon_1))$, $\epsilon_4\ge \kappa_0(\epsilon_1,\epsilon_2)/m$, and $m\ge 1$. Let $\{\bsx_k\}$ be the sequence generated by Algorithm~\ref{nonconvex:algorithm-pdgd}. 
	Then
	\begin{align}
	&\mathbb{E}_{\mathcal{A}_k}[\mathcal{L}_{1,k+1}] \nonumber\\
	&\quad \le  \mathcal{L}_{1,k}- a_1\|\bsx_k\|^2_{\bsE}
	- a_2\Big\|\bsv_k+\frac{1}{\gamma}\bsg_{k}^0\Big\|^2_{\bsF} - a_3\|\bsx_{k}-\bsy_{k}\|^2 \nonumber\\
	&\quad -\frac{1}{4}\alpha_k\|\bar{\bsg}^0_{k}\|^2  + pa_8\alpha_k^2e_{4,k} + pna_4\alpha_k^2+pna_5\alpha_k\mu_k^2.
		\label{zerosg:sgproof-vkLya2}
	\end{align}
\end{lemma}
\begin{proof}

Noting that $\epsilon_1>\kappa_1$ and $\gamma_k=\epsilon_4(k+m)\ge\epsilon_4 m\ge \kappa_0(\epsilon_1,\epsilon_2)\ge\varepsilon_{0}$, we know that all the conditions in Lemma~\ref{zerosg:lemma:sg:L1} are satisfied, so \eqref{zerosg:sgproof-vkLya:L1} hold.

Next, we prove \eqref{zerosg:sgproof-vkLya2} by scaling and bounding some coefficients.

Recalling $\gamma_k=\epsilon_4(k+m)$,
\begin{align}\label{zerosg:omegak}
\Delta_k&=\frac{1}{\gamma_{k}}-\frac{1}{\gamma_{k+1}}
=\frac{1}{\epsilon_4}(\frac{1}{(k+m)}-\frac{1}{(k+m+1)})
\nonumber\\
&\le\frac{1}{\epsilon_4(k+m)(k+m+1)}
\le\frac{\epsilon_4}{\gamma_k^2}.
\end{align}
From $\gamma_k=\epsilon_4(k+m)\ge\epsilon_4 m\ge \kappa_0(\epsilon_1,\epsilon_2)\ge\varepsilon_{0}\ge1$, we have $\Delta_k\le1$.	


From \eqref{zerosg:omegak}, $\alpha_k=\epsilon_2/\gamma_k$, $\gamma_k\ge1$, $\Delta_k\le1$, $\epsilon_4\ge \kappa_0(\epsilon_1,\epsilon_2)/m
\ge(\frac{2p(1+\eta_1^2)\varepsilon_9}{ a_1})^{\frac{1}{2}}\times\frac{1}{m}$, and $m\ge1$, we have
\begin{align}
b_{1,k}\le\frac{p(1+\eta_1^2)\varepsilon_9}{\epsilon_4^2m^2}\le\frac{ a_1}{2} .\label{zerosg:b1k}
\end{align}

From \eqref{zerosg:omegak}--\eqref{zerosg:b1k}, $\epsilon_4\ge  \kappa_0(\epsilon_1,\epsilon_2)/m\ge 2\varepsilon_5/(a_1m)$, $m\ge1$, and \eqref{zerosg:kappa4-6}, whether $\varepsilon_5>0$ or not, we have
\begin{align}
2 a_1-\varepsilon_5\Delta_k-b_{1,k}
\ge2 a_1-\frac{\varepsilon_5}{\epsilon_4 m}-\frac{ a_1}{2}
\ge a_1>0.\label{zerosg:varepsilon4}
\end{align}

From \eqref{zerosg:omegak}, $\alpha_k=\epsilon_2/\gamma_k$, $\epsilon_4\ge  \kappa_0(\epsilon_1,\epsilon_2)/m\ge \varepsilon_{10}/(2 a_2m) $, and \eqref{zerosg:kappa4-6}, we have
\begin{align}
b_{2,k}\ge2 a_2-\frac{\varepsilon_{10}}{2\epsilon_4 m}
\ge a_2>0.\label{zerosg:b2k}
\end{align}


From $\eqref{zerosg:omegak}, \gamma_k\ge1$ and $\Delta_k\le1$, we have
\begin{subequations}
\begin{align}
b_{3,k}&\le \varepsilon_{11},\label{zerosg:b3k}\\
b_{4,k}&\le \varepsilon_{12}.\label{zerosg:b4k}
\end{align}
\end{subequations}

From $\epsilon_4\ge \kappa_0(\epsilon_1,\epsilon_2)/m\ge p\epsilon_2\varepsilon_{12}/m$, $\gamma_k\ge1$, $\Delta_k\le1$ and \eqref{zerosg:b4k}, we have
	\begin{align}\label{zerosg:b3kb4keta}
	&b_{4,k}\alpha_k\le \frac{\epsilon_2\varepsilon_{12}}{\epsilon_4 m}
	\le\frac{1}{p}.
	\end{align}

From \eqref{zerosg:b3kb4keta}, $\alpha_k=\epsilon_2/\gamma_k$, $\gamma_k\ge1$, and $\Delta_k\le1$, we have
\begin{align}
b_{5,k}\le pn a_5.\label{zerosg:b5k}
\end{align}


From \eqref{zerosg:omegak} and $\epsilon_4\ge \kappa_0(\epsilon_1,\epsilon_2)/m \ge \delta_0b_{6,k}/(a_3m)$, we have
\begin{align}\label{zerosg:c2}
	\delta_0b_{6,k}\Delta_k\le \frac{\delta_0b_{6,k}}{\epsilon_4 m}\le  a_3.
\end{align}

From \eqref{zerosg:c1} and \eqref{zerosg:c2}, we have
\begin{align}\label{zerosg:c}
	c_2-\delta_0\varepsilon_2-\delta_0b_{6,k}\Delta_k\ge  a_3>0.
\end{align}

Then, from \eqref{zerosg:sgproof-vkLya:L1}, \eqref{zerosg:varepsilon4}--\eqref{zerosg:b5k} and \eqref{zerosg:c}, we know that \eqref{zerosg:sgproof-vkLya2} holds.
\end{proof}

\begin{lemma}\label{zerosg:lemma:sg2:23}
	%
	Under the same conditions as in Lemma~\ref{zerosg:lemma:sg2}, let ${\mathbf{x}_k}$ be the sequence generated by Algorithm~\ref{nonconvex:algorithm-pdgd}. Then
	\begin{subequations}
	\begin{align}
	&\mathbb{E}_{\mathcal{A}_k}[e_{4,k+1}]	
		\le  e_{4,k} + \frac{p}{n}a_9\alpha_k^2e_{4,k} +\|\bsx_k\|^2_{2\alpha_k  \ell ^2\bsE}\nonumber\\ 
		&\quad -\frac{1}{4}\alpha_k\|\bar{\bsg}_{k}^0\|^2 + p a_7\alpha_k^2
			+(n+p) \ell ^2\alpha_k\mu^2_k,
		\label{zerosg:v4kspeed-diminishing-2}\\
		&\mathbb{E}_{\mathcal{A}_k}[\mathcal{L}_{2,k+1}] \nonumber\\
		& \le   \mathcal{L}_{2,k}- a_1\|\bsx_k\|^2_{\bsE} - a_2\Big\|\bsv_k+\frac{1}{\gamma}\bsg_{k}^0\Big\|^2_{\bsF} - a_3\|\bsx_{k}-\bsy_{k}\|^2 \nonumber\\
		&\quad + a_6\alpha_k^2\|\bar{\bsg}^0_{k}\|^2  + pa_8\alpha_k^2e_{4,k} + pna_4\alpha_k^2+pna_5^{\prime}\alpha_k\mu_k^2.
		\label{zerosg:sgproof-vkLya2-bounded}
	\end{align}
	\end{subequations}
\end{lemma}

\begin{proof}

	The proof of \eqref{zerosg:v4kspeed-diminishing-2} follows the same argument as in \eqref{zerosg:v4kspeed}.
	In a similar way to the proof of \eqref{zerosg:sgproof-vkLya2}, the combination of inequalities \eqref{L2:0}, \eqref{zerosg:varepsilon4}--\eqref{zerosg:b4k}, \eqref{zerosg:b5k}, and \eqref{zerosg:c} leads to \eqref{zerosg:sgproof-vkLya2-bounded}.
\end{proof}

\begin{lemma} \label{L3_bound:time-varying1}
	Under the same setting as in Theorem~\ref{Thm:PL:know},  
		 consider  the sequence $\{\bsx_k\}$ generated by Algorithm~\ref{nonconvex:algorithm-pdgd}. 
		 Then
there exists a constant $\bar{L}>0$ such that
	\begin{align}
		&~~~~\qquad \mathbb{E}[e_{4,k}]\le n\bar{L},~\forall k\in[0,T], \label{L3_bound_ineq:TV}\\
		&\bar{L} \le \frac{2e_{4,0}}{n} + 4\ell^2\frac{{\mathcal{L}}_{1,0}}{n} 
			+ \frac{a_{11}}{2a_9}
					+ 2 = \mathcal{O}(1). \label{Lbar:bound:TV}
	\end{align}
\end{lemma}

\begin{proof}
			We prove by \eqref{L3_bound_ineq:TV}  by mathematical induction. 

		For $k=0$, the result follows directly from the definition.

		For $k>0$, suppose that the statement holds for $\tau = 0,1,\ldots,k-1$, 
		namely,
		\[
			\mathbb{E}[e_{4,\tau}] \le n\bar{L}, \quad \forall\, \tau=0,1,\ldots,k-1.
		\]

		It is clear that all the conditions in Lemmas~\ref{zerosg:lemma:sg2} and \ref{zerosg:lemma:sg2:23} are satisfied.
		Then, from \eqref{zerosg:sgproof-vkLya2} and the monotonicity of $\alpha_k$, we get
		\begin{align}
			&\alpha_{\tau+1}\mathbb{E}_{\mathcal{A}_{\tau}}[\mathcal{L}_{1,\tau+1}] \le \alpha_{\tau}\mathbb{E}_{\mathcal{A}_{\tau}}[\mathcal{L}_{1,\tau+1}]
			 \le \alpha_{\tau}\mathcal{L}_{1,\tau} \nonumber\\
			&
				- a_1\alpha_{\tau}\|\bsx_\tau\|^2_{\bsE}  
				+ pa_8\alpha_{\tau}^3n\bar{L} 
				 + pna_4\alpha_{\tau}^3+pna_5\kappa_{\mu}^2\alpha_{\tau}^3, \nonumber\\
			&\qquad\qquad \qquad \qquad \qquad \quad  \forall \tau = 0,1,\dots,k-1.
			\label{zerosg:sgproof-vkLya2-induction}
		\end{align}

		Taking expectation with respect to $\mathcal{F}_T$, summing \eqref{zerosg:sgproof-vkLya2-induction} over $ \tau\in[0,k-1]$ yields
		\begin{align}
			&\sum_{\tau=0}^{k-1}\mathbb{E}[\alpha_{\tau}\|\bsx_\tau\|^2_{\bsE}]
			\le\frac{1}{a_1}\Big(\alpha_0{\mathcal{L}}_{1,0} + 
				pna_8\bar{L}\sum_{\tau=0}^{k-1}\alpha_{\tau}^3 \nonumber\\
			& +pn(a_4+a_5\kappa_{\mu}^2)\sum_{\tau=0}^{k-1}\alpha_{\tau}^3 \Big).
			\label{x_bound:TV}
		\end{align}

    Similarly, the expected sum of \eqref{zerosg:v4kspeed-diminishing-2} over $\tau \in [0,k-1]$ with respect to $\mathcal{F}_T$, combined with \eqref{x_bound:TV}, yields		
	\begin{align}
			&\mathbb{E}[e_{4,k}]
			\le e_{4,0} + \frac{2\ell^2}{a_1}\alpha_0{\mathcal{L}}_{1,0} 
				+ \frac{2\ell^2a_8}{a_1}pn\bar{L}\sum_{\tau=0}^{k-1}\alpha_{\tau}^3 \nonumber\\
				&\quad +\frac{2\ell^2(a_4+a_5\kappa_{\mu}^2)}{a_1}pn\sum_{\tau=0}^{k-1}\alpha_{\tau}^3  \nonumber\\
				&\quad + pa_9\bar{L}\sum_{\tau=0}^{k-1}\alpha_{\tau}^2
				+ p(a_7 + \ell^2\kappa_{\mu}^2)\sum_{\tau=0}^{k-1}\alpha_{\tau}^2 \nonumber\\
			&= e_{4,0} + \frac{2\ell^2}{a_1}\alpha_0{\mathcal{L}}_{1,0} 
				 +\Big(\frac{p}{n}a_9\sum_{\tau=0}^{k-1}\alpha_{\tau}^2 + 
				 \frac{2\ell^2a_8}{a_1}p\sum_{\tau=0}^{k-1}\alpha_{\tau}^3   \Big)n\bar{L}
				\nonumber\\
				&\quad + p(a_7 + \ell^2\kappa_{\mu}^2)\sum_{\tau=0}^{k-1}\alpha_{\tau}^2
					 +\frac{2\ell^2(a_4+a_5\kappa_{\mu}^2)}{a_1}pn\sum_{\tau=0}^{k-1}\alpha_{\tau}^3 \nonumber\\
			&\le e_{4,0} + \frac{2\ell^2}{a_1}\alpha_0{\mathcal{L}}_{1,0} 
					+\Big(\frac{a_9\epsilon_2^2p}{\epsilon_4^2(m-1)n} 
					+  \frac{a_{10}\epsilon_2^3p}{2\epsilon_4^3(m-1)^2}   \Big)n\bar{L}
					\nonumber\\
				&\quad + \frac{a_{11}\epsilon_2^2p}{\epsilon_4^2(m-1)} 
					+ \frac{a_{12}\epsilon_2^3pn}{2\epsilon_4^3(m-1)^2} \nonumber\\
			&\le e_{4,0} + 2\ell^2{\mathcal{L}}_{1,0} 
					+ \frac{1}{2}n\bar{L} + \frac{a_{11}n}{4a_9} + n, \label{T:use3}
		\end{align}
		where the second inequality follows from the bound $\sum_{\tau=0}^{k-1}\frac{1}{(\tau+m)^s} \le \frac{1}{(s-1)(m-1)^{s-1}}$ for $k \ge 1$, $m > 1$, and $s = 2, 3$;
		and the last inequality holds due to $\alpha_0=\frac{\epsilon_2}{\epsilon_4m}$ and $m\ge\kappa_m\ge\max\big\{\frac{\epsilon_2}{a_1\epsilon_4},
	~\frac{4a_9\epsilon_2^2}{\epsilon_4^2}\frac{p}{n}+1,~\sqrt{\frac{2a_{10}\epsilon_2^3p}{\epsilon_4^3}}+1,~\sqrt{\frac{a_{12}\epsilon_2^3p}{2\epsilon_4^3}}+1\big\}$.
		
		As long as 
		\begin{align}
			\bar{L}\ge \frac{2e_{4,0}}{n} + 4\ell^2\frac{{\mathcal{L}}_{1,0}}{n} 
			+ \frac{a_{11}}{2a_9}
					+ 2, \nonumber
		\end{align}
		one has $\mathbb{E}[e_{4,k}]\le\bar{L}$.
		Thus \eqref{L3_bound_ineq:TV} follows from $m\ge\kappa_m(\epsilon_1, \epsilon_2, \epsilon_4)=\max\{\frac{4a_9\epsilon_2^2}{\epsilon_4^2}\frac{p}{n}+1,
			~\sqrt{\frac{2a_{10}\epsilon_2^3p}{\epsilon_4^3}}+1\}$.
		Moreover, 
		an upper bound of $\bar{L}$ can be derived as
		\begin{align}
			&\bar{L} \le \max\Big\{ \frac{e_{4,0}}{n},~\frac{2e_{4,0}}{n} + 4\ell^2\frac{{\mathcal{L}}_{1,0}}{n} 
			+ \frac{a_{11}}{2a_9} + 2 \Big\}\nonumber\\
			&=\frac{2e_{4,0}}{n} + 4\ell^2\frac{{\mathcal{L}}_{1,0}}{n} 
			+ \frac{a_{11}}{2a_9}
					+ 2. \label{nL:TV:0}
		\end{align} 
		Since $a_{9},a_{11}$ do not depend on either the dimension $p$ or the communication network, together with $e_{4,0}=\mathcal{O}(n)$ and $\mathcal{L}_{1,0}=\mathcal{O}(n)$,  \eqref{nL:TV:0} further yields \eqref{Lbar:bound:TV}.
\end{proof}


Building upon Lemmas~\ref{zerosg:lemma:sg2}--\ref{L3_bound:time-varying1}, we proceed to prove Theorem~\ref{Thm:PL:know}.
From $m \ge \kappa_m(\epsilon_1,\epsilon_2,\epsilon_4) \ge \kappa_0(\epsilon_1,\epsilon_2)/\epsilon_4$, it follows that $\epsilon_4 \ge \kappa_0(\epsilon_1,\epsilon_2)/m$, and hence all the conditions in these lemmas are satisfied.

From \eqref{zerosg:rand-grad-smooth} and \eqref{L3_bound_ineq:TV}, one has
\begin{align}\label{nonconvex:gg3:le:TV}
	\mathbb{E}[\|\bar{\bsg}^0_k\|^2]=n\mathbb{E}[\|\nabla f(\bar{x}_k)\|^2]\le2\ell \mathbb{E}[e_{4,k}]\le2n\ell\bar{L}.
\end{align}
Taking the expectation of \eqref{zerosg:sgproof-vkLya2-bounded} with respect to $\mathcal{F}_T$ and combining it with \eqref{nonconvex:gg3:le:TV}, \eqref{L3_bound_ineq:TV}, \eqref{zerosg:vkLya3.1}, and \eqref{zerosg:step:eta1t1}, we obtain
\begin{align}\label{zerosg:vkLya4-bound-pl-1}
	&\mathbb{E}[\mathcal{L}_{2,k+1}]\le(1-d_1)\mathbb{E}[\mathcal{L}_{2,k}] \nonumber\\
	&\quad + \frac{pn(\frac{2a_6\ell\bar{L}}{p} + a_8\bar{L} + a_4+ a_5^{\prime}\kappa_{\mu}^2)\epsilon_2^2}{\epsilon_4^2(k+m)^{2}}.
\end{align}
Noting from \eqref{zerosg:vkLya2-a1-bounded-thm2} that $0 \le d_1 \le 1/2$, 
combining \eqref{zerosg:vkLya4-bound-pl-1}, \eqref{zerosg:vkLya3}, and Lemma~5 in \cite{Yi_Zerothorder_2022,Yi_Zerothorder_2021} yields
\begin{align}\label{zerosg:lemma:sequence-equ6-bounded-pl-speed}
\mathbb{E}[\|\bsx_k\|_{\bsE}^2]\le\frac{1}{\varepsilon_7}\mathbb{E}[\mathcal{L}_{2,k}]
	=\mathcal{O}(\frac{pn}{(k+m)^2}) + \mathcal{O}\big(n(1-d_1)^k\big),
\end{align}
which gives \eqref{zerosg:thm-sg-diminishing-equ2.1bounded}.

Similarly, the expectation of \eqref{zerosg:v4kspeed-diminishing-2} with respect to $\mathcal{F}_T$, 
in combination with \eqref{nonconvex:gg3}, \eqref{zerosg:lemma:sequence-equ6-bounded-pl-speed}, \eqref{nonconvex:gg3:le:TV}, and \eqref{zerosg:step:eta1t1}, yields
\begin{align}
	&\mathbb{E}[e_{4,k+1}] \nonumber\\
	&\le \Big(1-\frac{\nu\epsilon_2}{2\epsilon_4(k+m)}\Big)\mathbb{E}[e_{4,k}] + \frac{p(a_9\bar{L} + a_7 + \ell^2k_\mu^2)\epsilon_2^2}{\epsilon_4^2(k+m)^2} \nonumber\\
		&\quad + \frac{2\ell^2\epsilon_2}{\epsilon_4(k+m)}\mathcal{O}(\frac{pn}{(k+m)^2})
		 + \frac{2\ell^2\epsilon_2}{\epsilon_4(k+m)}\mathcal{O}\big(n(1-d_1)^k\big). \nonumber\\
\end{align}
Noting that $\nu\epsilon_2/(2\epsilon_4) > 2$ as $\epsilon_4<\nu\epsilon_2/4$, following the similar proof of Lemma~5 in \cite{Yi_Zerothorder_2022,Yi_Zerothorder_2021}, and using Lemma~\ref{series}, we obtain
\begin{align}
&\mathbb{E}[f(\bar{x}_{T})-f^*]\nonumber\\
&\le\mathcal{O}\Big(\frac{m^{\frac{\nu\epsilon_2}{2\epsilon_4}}}{(T+m)^{\frac{\nu\epsilon_2}{2\epsilon_4}}}\Big) +\mathcal{O}\Big(\frac{p}{n(T+m)^{2}}\Big)+\frac{\varepsilon_{15}p}{n(T+m)}\nonumber\\
&+\mathcal{O}\Big(\frac{m^{\frac{\nu\epsilon_2}{2\epsilon_4}}}{(T+m)^{\frac{\nu\epsilon_2}{2\epsilon_4}}}\Big) +\mathcal{O}\Big(\frac{p}{(T+m)^{3}}\Big) +\mathcal{O}\Big(\frac{p}{(T+m)^{2}}\Big) \nonumber\\
&+\mathcal{O}\Big(\frac{m^{\frac{\nu\epsilon_2}{2\epsilon_4}}}{(T+m)^{\frac{\nu\epsilon_2}{2\epsilon_4}}}\Big) + \mathcal{O}\Big(\frac{1}{(T+m)^{\frac{\nu\epsilon_2}{2\epsilon_4}}}\Big) \nonumber\\
& +\mathcal{O}\Big(\frac{m^{\frac{\nu\epsilon_2}{2\epsilon_4}-1}}{(T+m)^{\frac{\nu\epsilon_2}{2\epsilon_4}}}\Big) +  \mathcal{O}\Big(\frac{1}{(T+m)^{\frac{\nu\epsilon_2}{2\epsilon_4}}}\Big). \nonumber
\end{align}
Since $m=\mathcal{O}(p)$, we get \eqref{zerosg:thm-sg-diminishing-equ2bounded}, 
with the coefficient of the leading term
$p/(nT)$ bounded by
\begin{align}
	&\qquad\frac{4\epsilon_2^2(a_9\bar{L} + a_7+ \ell ^2\kappa_{\mu}^2)}{\epsilon_4^2(\frac{\nu\epsilon_2}{2\epsilon_4}-1)}
\le\frac{64(a_9\bar{L} + a_7+ \ell ^2\kappa_{\mu}^2)}{\kappa_4(2-\kappa_4)\nu^2}=\nonumber\\
&\frac{1}{\nu^2}\mathcal{O}\Big(
\ell\sigma_1^2+\ell(1+\eta_1^2)\check\sigma_2^2 
+\ell^2(1+\eta_1^2)
\big(f(\bar{x}_0)\hspace{-0.1em}-\hspace{-0.1em}f^*\hspace{-0.1em}+\hspace{-0.1em}1\hspace{-0.1em}+\hspace{-0.1em}\ell^2\big)\hspace{-0.2em}\Big)\hspace{-0.1em}.
\end{align}


\section*{Acknowledgment}
The authors thank Dr.~Shengjun Zhang for sharing codes.

\clearpage

\end{document}

%% file: styledsymbols.tex
\def\bsg{{\mathbf{g}}}

\def\bsq{{\mathbf{q}}}

\def\bsv{{\mathbf{v}}}

\def\bsx{{\mathbf{x}}}
\def\bsy{{\mathbf{y}}}

\def\bsE{{\mathbf{E}}}
\def\bsF{{\mathbf{F}}}

\def\bsH{{\mathbf{H}}}

\def\bsL{{\mathbf{L}}}

\def\bsM{{\mathbf{M}}}